\documentclass[reqno]{amsart} 
\usepackage[margin=1.55in]{geometry}
\usepackage{graphicx, amsmath, amssymb, amsfonts, amsthm, stmaryrd, amscd}
\usepackage[usenames, dvipsnames]{xcolor}
\usepackage{enumerate, cite, latexsym}
\usepackage[alphabetic,initials,nobysame]{amsrefs}
\usepackage[normalem]{ulem}
\usepackage{ifpdf}
\ifpdf \RequirePackage[colorlinks=true, , linkcolor=Blue, 
citecolor=Blue, pdftex, linktocpage]{hyperref} \fi
\usepackage{xparse}
\usepackage[symbol]{footmisc}
\usepackage{booktabs}

\usepackage[all]{xy}
\usepackage{enumerate}
\usepackage{mathtools}

\usepackage{xpatch}

\makeatletter
\xpatchcmd{\@tocline}
  {\hfil\hbox to\@pnumwidth{\@tocpagenum{#7}}\par}
  {\ifnum#1<0\hfill\else\dotfill\fi\hbox to\@pnumwidth{\@tocpagenum{#7}}\par}
  {}{}
\makeatother

\patchcmd{\section}{\scshape}{\bfseries}{}{}
\makeatletter

\makeatletter
\newcommand*{\transpose}{%
	{\mathpalette\@transpose{}}%
}
\newcommand*{\@transpose}[2]{%
	\raisebox{\depth}{$\m@th#1\intercal$}%
}
\makeatother
\makeatletter
\newcommand*{\da@rightarrow}{\mathchar"0\hexnumber@\symAMSa 4B }
\newcommand*{\da@leftarrow}{\mathchar"0\hexnumber@\symAMSa 4C }
\newcommand*{\xdashrightarrow}[2][]{%
	\mathrel{%
		\mathpalette{\da@xarrow{#1}{#2}{}\da@rightarrow{\,}{}}{}%
	}%
}
\newcommand{\xdashleftarrow}[2][]{%
	\mathrel{%
		\mathpalette{\da@xarrow{#1}{#2}\da@leftarrow{}{}{\,}}{}%
	}%
}
\newcommand*{\da@xarrow}[7]{%
	\sbox0{$\ifx#7\scriptstyle\scriptscriptstyle\else\scriptstyle\fi#5#1#6\m@th$}%
	\sbox2{$\ifx#7\scriptstyle\scriptscriptstyle\else\scriptstyle\fi#5#2#6\m@th$}%
	\sbox4{$#7\dabar@\m@th$}%
	\dimen@=\wd0 %
	\ifdim\wd2 >\dimen@
	\dimen@=\wd2 %
	\fi
	\count@=2 %
	\def\da@bars{\dabar@\dabar@}%
	\@whiledim\count@\wd4<\dimen@\do{%
		\advance\count@\@ne
		\expandafter\def\expandafter\da@bars\expandafter{%
			\da@bars
			\dabar@ 
		}%
	}%
	\mathrel{#3}%
	\mathrel{%
		\mathop{\da@bars}\limits
		\ifx\\#1\\%
		\else
		_{\copy0}%
		\fi
		\ifx\\#2\\%
		\else
		^{\copy2}%
		\fi
	}%
	\mathrel{#4}%
}
\makeatother

\newcommand{\md}[1]{\ensuremath{(\operatorname{mod}\, #1)}}

\newcommand\Z{\mathbb{Z}}
\newcommand\R{\mathbb{R}}
\newcommand\C{\mathbb{C}}
\newcommand\N{\mathbb{N}}
\newcommand\Q{\mathbb{Q}} 
\newcommand{\A}{\mathbb{A}}
\newcommand{\F}{\mathbb{F}}

\newcommand\eps{\varepsilon}
\newcommand{\sumstar}{\sideset{}{^\star}\sum}
\newcommand{\sumflat}{\sideset{}{^\flat}\sum}

\newcommand{\br}{\overline}

\renewcommand{\leq}{\leqslant}
\renewcommand{\geq}{\geqslant}

\DeclareMathOperator{\SL}{SL}

\DeclareMathOperator{\GL}{GL}

\def\eps{\varepsilon}

\DeclareMathOperator{\PGL}{PGL}

\DeclareMathOperator{\Kl}{Kl}

\DeclareMathOperator{\Sym}{Sym}

\DeclareMathOperator{\sym}{Sym}
\DeclareMathOperator{\hyp}{Hyp}

\DeclareMathOperator{\Sp}{Sp}

\DeclareMathOperator{\M}{Main}
\DeclareMathOperator{\Er}{Err}

\theoremstyle{plain} 
\newtheorem{theorem}{Theorem}[section] 
\newtheorem{lemma} [theorem] {Lemma}
 
\newtheorem{corollary} [theorem] {Corollary} 
\newtheorem{proposition} [theorem] {Proposition}

\theoremstyle{definition}

\newtheorem{remark}[theorem]{Remark}

\numberwithin{equation}{section}

\newtheoremstyle{itplain} 
{6pt}                    
{5pt\topsep}                    
{\itshape}                   
{}                           
{\itshape}                   
{.}                          
{5pt plus 1pt minus 1pt}                       
{}  

\theoremstyle{itplain} 

\newtheorem*{lemma*}{Lemma}
\newtheorem*{remark*}{Remark}
\newtheorem*{proposition*}{Proposition}
\newtheorem*{definition*}{Definition}
\newtheorem*{example*}{Example}\newtheorem*{note*}{Note}

\newtheorem*{results*}{Results}

\usepackage[displaymath,textmath,sections,graphics]{preview}
\PreviewEnvironment{align*}
\PreviewEnvironment{multline*}
\PreviewEnvironment{tabular}
\PreviewEnvironment{verbatim}
\PreviewEnvironment{lstlisting}
\PreviewEnvironment*{frame}
\PreviewEnvironment*{alert}
\PreviewEnvironment*{emph}
\PreviewEnvironment*{textbf}
\PreviewEnvironment*{pn}

\begin{document}
	
	\author[Pratim Mitra]{Pratim Mitra}
	\address{ISI\\Statistics and Mathematics Unit\\
		Kolkata 70010\\
		India}
	\email{pratim2018mitra@gmail.com}
	
	\date{\today}

	
	\title[The subconvexity problem for symmetric square $L$-functions]{The subconvexity problem for symmetric square\\ $L$-functions in Level aspect}
	
	\thanks{}

    \begin{abstract}
        In this paper, we address the subconvexity problem in level aspect for symmetric square $L$-functions for cuspidal automorphic representation of $\mathrm{GL}_2(\mathbb{Q})$ with a prescribed local ramification at prime $p$. To be more precise, let $\pi$ be a tempered cuspidal automorphic representation of conductor $q(\pi)=p^2$ with a non-quadratic central character of conductor $p$. We prove that if the corresponding local representation $\pi_p$ belongs to a suitable class of representations $\mathcal S$, then 
        \[
        L\left(\frac{1}{2},\,\mathrm{Sym}^2\pi\right)\ll_{\varepsilon, \pi_\infty} q(\mathrm{Sym}^2\pi)^{\frac{1}{4}-\frac{1}{168}+o(1)},
        \]
        where implied constant depends polynomially on the spectral parameters of $\pi_\infty$. This is the first instance of level-aspect subconvex bound for $L$-functions of a $\mathrm{GL}_3(\mathbb Q)$ automorphic representation. Our approach is based on the delta-symbol method. Apart from some standard analytic number theoretic tools, Katz's theory of hypergeometric sums, and Deligne's proof of Weil-conjectures play an important role in the proof.
    \end{abstract}

    \maketitle
	
	\setcounter{tocdepth}{1}
	
	\tableofcontents

    \section{Introduction}

    One of the most fundamental problems in the theory of automorphic forms and $L$-functions is to measure the growth of the values of $L$-functions at the center of the critical strip. Let $\Pi$ be an automorphic representation of a general linear group $\GL_n(\Q)$ for $n\geq 1$, and let $L(s, \Pi)$ be its associated $L$-function. To estimate the growth Iwaniec-Sarnak introduced the so-called \textit{analytic conductor} at $1/2$, $Q(\Pi)$, defined by, 
    \[
    Q(\Pi):=q(\Pi)q_\infty\left({1}/{2},\,\Pi\right)= q(\Pi)\prod_{j=1}^d(2+|\mu_{j, \infty}|),
    \]
    where, $q(\Pi)$ is the \textit{level} or \textit{arithmetic conductor} and $\mu_{j, \infty}$ are the spectral parameters. By Phragm\'en-Lindel\"of principal, interpolating the trivial growths of the $L$-function on the outside of the critical strip we get, 
    \[
    L\left(\frac{1}{2},\,\Pi\right)\ll Q(\Pi)^{\frac{1}{4}+\eps}.
    \]
    This trivial bound is called as the \textit{convexity bound}\footnote[1]{Heath-Brown observed that by more elegant analysis one can also remove this $\eps$ from the exponent} and improving upon the convexity exponent by some absolute constant $\delta>0$ 
    \[
    L\left(\frac{1}{2},\,\Pi\right)\ll_\eps Q(\Pi)^{\frac{1}{4}-\delta+\eps},
    \]
    is known as the \textit{subconvexity problem} for the central $L$-value. 

    In general, solving the subconveixty problem is very difficult, and it has been one of the most strong testing grounds for the strength of existing technology in the literature. $\GL(1)$ and $\GL(2)$ are the only known instances where this problem has been solved completely; $\GL(1)$ case was solved by~\cites{MR0132733, MR0148626, MR0485727}, and in the $\GL(2)$ case after substantial efforts by many people notably~\cite{MR1207474, MR1258904, MR1923476}, the problem was solved by~\cite{MR2653249}.

    Beyond $\GL(2)$ this problem is still not well-understood. But we can reformulate this problem with respect to each parameters instead of the whole analytic conductor. The subconvexity problem with respect to the parameter $q(\Pi)$ while keeping the spectral parameters fixed is called as subconvexity in arithmetic-conductor aspect or level-aspect; and the same with respect to spectral parameters while keeping the arithmetic conductor fixed is called as spectral aspect subconvexity. In this article we are actually interested in the former. 

    First breakthrough was due to Munshi~\cite{MR3418527} who proved subconvexity for $\GL(3)$ representations $\Pi$ such that $\Pi\simeq\pi\otimes\chi$, where $\pi$ is an unramified cuspidal representation for $\GL(3)$, and $\chi$ is a primitive Dirichlet character of modulo prime $p$. And recently, Hu-Nelson~\cite{902690390} generalized this result for all higher-rank setting where the conductor of the twisting character is a perfect square $N^2$.

    Let $\Pi\simeq\pi_1\otimes\pi_2$ be a Rankin-Selberg convolution of cuspidal representations for $\GL(2)$, which corresponds to an automorphic representation of $\GL(4)$. The level-aspect subconvexity for $\Pi$ is known when the conductor of one of the representations is fixed, or when both of their conductors are coprime and have comparable sizes~\cite{MR4576022,  MR2207235, MR2653249, MR3156856,  MR3078642}. Recently, Hu-Michel-Nelson~\cite{900541809} established this result for the case where both representations (with trivial central characters) are ramified at the same places and the conductor of the convolution drops by a certain amount. 

    Now we consider a representation $\Pi\simeq\pi_1\otimes\pi_2$, where $\pi_1$ (resp. $\pi_2$) is a cuspidal automorphic representation of $\GL(3)$ (resp. $\GL(2)$); and $\Pi$ corresponds to an automorphic representation of $\GL(6)$. In this case this problem has been solved in very few special cases, summarized below. In~\cite{MR4705885} Kumar-Munshi-Singh solved the case when both of their conductors are coprime and have comparable sizes. And recently Munshi~\cite{MR5049991} has solved the case when $\GL(3)$ form is fixed but $\GL(2)$ form (with trivial central character) has conductor $p^j$ ($j\geq 2$).

    Rankin-Selberg convolution is a way of constructing higher-rank automorphic representation from the lower-rank ones. Apart from the Rankin-Selberg convolution another way of constructing higher-rank forms from the lower-rank ones are symmetric power representations. For example, Gelber-Jacquet~\cite{MR0533066} showed that the symmetric square lift of a $\GL(2)$ form $\pi$, denoted by $\sym^2\pi$, corresponds to an automorphic representation of $\GL(3)$.
    
	In this work, we address the level aspect subconvexity problem for $\GL(3)$ automorphic representations which remained wide open until now. A particular important example is the subconvexity for symmetric-square $L$-functions, which has profound applications in number theory and related areas. 
    Here, we obtain subconvex bound for symmetric-square $L$-functions in the level-aspect when the underlying $\GL(2)$-form has level $p^2$ and non-trivial nebentype of conductor $p$--- specifically, in this case where the arithmetic conductor $q(\sym^2\pi)$ does not drop. For precise statement, see the following section.

    \subsection{Main results} Let $p$ be a large prime. To state our main result we need to define a class of local representations.

    Let $\pi\simeq\bigotimes_\nu^\prime\pi_\nu$ be a cuspidal automorphic representation for $\GL(2)$ of arithmetic conductor $q(\pi)=p^2$, and unitary central character $\omega_\pi$ of conductor $p$. Ramification of $\pi$ is concentrated only at place $p$. Suppose the local representation $\pi_p$ is one of the following:
    \begin{enumerate}[(a)]
        \item $\pi_p$ is ramified principal series i.e. $\pi_p\simeq\pi(\tilde\mu_1, \tilde\mu_2)$, where $\tilde\mu_1$ and $\tilde\mu_2$ are      two distinct primitive Dirichlet characters modulo $p$ such that $\mu_1, \mu_2, \mu_1\mu_2, \mu_1\mu_2^{-1}$ are non-quadratic, and here by $\tilde\chi$ we denote the $\Q_p^\times$ lift of the Dirichlet character $\chi$ modulo $p$.
        \item $\pi_p$ is a twist-minimal\footnote[2]{We call a representation $\pi_p$ to be \textit{twist-minimal} if its conductor $q(\pi_p)$ is smallest among the conductors $q(\pi_p\otimes\omega)$ of all twists $\pi_p\otimes\omega$ by one-dimensional characters $\omega$ of $\Q_p^\times$.} supercuspidal representation whose associated Green's character $\eta_\pi:\F_{p^2}^{\times}\to\C^\times$ is such that $\eta_\pi^{p\pm 1}$ are both non-quadratic.
    \end{enumerate}

    \begin{note*}
        We have proved in Lemma~\ref{thm: computation of conductor} that if $\pi_p$ is supercuspidal then it is a depth-zero supercuspidal, and hence we have the existence of Green's character.
    \end{note*}

     To each local representation we attach an exotic hypergeometric sum (defined by the equation~\eqref{eq:exotic hypergeo}), 
     \[
     \begin{cases}
         \text{Hyp}(\F_p^2\times\F_p^2; 1, \rho; \bar\mu_1, \bar\mu_2; \lambda)& \text{if $\pi_p$ is of type (a)},\\
         \text{Hyp}(\F_p^2\times\F_{p^2}; 1, \rho; \eta_\pi; \lambda)&\text{if $\pi_p$ is of type (b).}
     \end{cases}
     \]
     Here $\rho$ is the quadratic character modulo $p$. To each exponential sum there is a geometrically irreducible $\ell$-adic sheaves, and let $G_{\text{geom}}^0$ be the identity component of its geometric monodromy group. Now we are ready to define our class $\mathcal S$, given by, 
     \[
     \mathcal S:=\{\pi_p~\text{is of type}~(a)\,\text{or}~(b): G_{\text{geom}}^0\neq\{1\}\}.
     \]
     Satement of our main result is the following, 
     \begin{theorem}\label{thm: thm1}
         Let $\pi\simeq\bigotimes_\nu^\prime\pi_\nu$ be a cuspidal automorphic representation for $\GL_2(\Q)$ of arithmetic conductor $q(\pi)=p^2$, with the central character $\omega_\pi$, the adelic lift of a primitive non-quadratic Dirichlet character $\chi_\pi$ modulo $p$. If $\pi_\nu$ is tempered for all $\nu<\infty$, $\pi_p\in\mathcal S$, and the spectral parameters of $\pi_\infty$ be bounded, then we have the following subconvex bound, 
         \[
         L\left(\frac{1}{2},\, \sym^2\pi\right)\ll_{\eps, \pi_\infty}q(\sym^2\pi)^{\frac{1}{4}-\frac{1}{168}+\eps},
         \]
         where the implied constants depends on $\eps$, and polynomially on spectral parameters of $\pi_\infty$.
     \end{theorem}

      We say a prime $p$ is \textit{admissible} if it does not satisfy any of the following congruences: 
     \[
     p\equiv 1\md{12},\hspace{0.1cm}p\equiv 1\md{20},\hspace{0.1cm}p\equiv 17, 19\md{24},\hspace{0.1cm} p\equiv 41\md{60},\hspace{0.1cm} p\equiv 49\md{60}.
     \]
     \begin{corollary}\label{thm:cor main}
         Let $p$ be any admissible prime. Let $\pi\simeq\bigotimes_\nu^\prime\pi_\nu$ be a cuspidal automorphic representation for $\GL_2(\Q)$ of artihmetic conductor $q(\pi)=p^2$, with the central character $\omega_\pi$, the adelic lift of a primitive non-quadratic Dirichlet character $\chi_\pi$ modulo $p$. If $\pi_p$ is either of the type (a) or (b); $\pi_\nu$ is tempered for all $\nu<\infty$, and the spectral parameters of $\pi_\infty$ be bounded, then we have the following subconvex bound,
         \[
         L\left(\frac{1}{2},\,\sym^2\pi\right)\ll_{\eps, \pi_\infty}q(\sym^2\pi)^{\frac{1}{4}-\frac{1}{168}+\eps}.
         \]
     \end{corollary}
     
     \begin{remark}
         It is worth comparing our result with those of~\cite{MR4705885,  MR5049991}, as mentioned in the previous section that both of which treated certain cases of the subconvexity problem for $\GL(3)\times\GL(2)$;  the former treated the ``hybrid-level'' case, and the latter treated the ``$\GL(2)$ pure-level'' case. The pure-level case, as opposed to the hybrid case, is significantly more difficult because of its higher ``arithmetic complexity''. Our result can be though of as a $\GL(3)\times\GL(2)$ subconvexity problem in ``$\GL(3)$ pure-level'' case, where the $\GL(2)$ form varies over the continuous spectrum rather than the discrete spectrum. Note that even though our result is a $\GL(3)$ pure-level case, it is not a \textit{genuine} $\GL(3)$ pure-level result, because this level actually comes from the $\GL(2)$ form under the symmetric square lift.
     \end{remark}
     
     \begin{remark}\hspace{0.1cm}
         \begin{enumerate}[1.]
             \item Our motivation for the choice of the central character and the local representations $\pi_p$ is primarily to achieve ``uniform growth'' in the conductor i.e., no drop in the conductor, of $\sym^2\pi$.
             \item The difficulty of this problem can be compared to with that of subconvexity problem for symmetric square $L$-functions associated with a $\GL(2)$ cusp form $f$ of level $p$ with trivial nebentype.. In both of the case under the symmetric-square lift the conductor exponent increase by $1$, but the presence of the non-trivial nebentype makes the problem more tactable.
             \item The temperedness assumption on $\pi$---namely, the bound $\lambda_\pi(n)\ll n^{o(1)}$ on the Fourier coefficients---is used in only one place, specifically in Lemma~\ref{thm: pointwise bd of alpha}. 
             \item This result would also holds if the local representation $\pi_p$ is the special representation, but modulo some cases.
         \end{enumerate}
     \end{remark}

     An immediate corollary of the theorem is the following, 
     \begin{corollary}
         Let $\pi$ be as in theorem, we have
         \[
         L\left(\frac{1}{2},\,\pi\otimes\pi\right)\ll_{\eps, \pi_\infty} q(\pi\otimes\pi)^{\frac{1}{4}-\frac{1}{168}+\eps}.
         \]
     \end{corollary}

     \begin{remark}
         Our approach is based on the DFI delta-symbol method with a ``conductor reduction'' trick pioneered by Munshi. This method has been proven successful in various level-aspect problems, such as those studied in~\cite{MR3357122, MR3994569, MR4705885,  MR5049991}, but the effectiveness and adaptability of this method in further arithmetic problems in the level aspect still remain a topic of discussion. Our delta-symbol based approach establishes that it is still applicable and effective in level-aspect case for higher-rank forms.
     \end{remark}

	\subsection{An overview of the proof}

    In this section we present a high level details of the proof. By approximate functional equation central $L$-value, $L(1/2, \sym^2\pi)$, can be approximate by the sum of following type:
    \[
    S_\pi(X)=\sum_{n\sim X}\lambda_\pi(n^2),
    \]
    with $X\ll p^{3/2+\eps}$. Trivial bound of this sum i.e., $|S_\pi(X)|\ll Xp^\eps$, gives the trivial (convexity) bound for $L(1/2, \sym^2\pi)$. Hence the subconvexity bound for this $L$-value is essentially equivalent to a non-trivial bound of the following type:
    \[
    S_\pi(X)\ll Xp^{-\delta},
    \]
    for some absolute constant $\delta>0$. In this article we are able to obtain the following bound
    \[
    S_\pi(X)\ll \sqrt{X}p^{\frac{3}{4}-\frac{1}{56}+\eps},
    \]
    and it is non-trivial in the range $p^{\frac{3}{2}-\frac{1}{28}}\ll X\ll p^{\frac{3}{2}+\eps}$. Combining this with the trivial estimate, we obtain our desired subconvexity bound. In this sketch we focus just on the generic case, i.e. $X=p^{3/2}$. 

    So our main object is to obtain an estimate like following, 
    \[
    \sum_{n\sim p^{3/2}}\lambda_\pi(n^2)\ll p^{3/2-\delta}.
    \]
    For simplicity we drop $\eps$ and all weight functions from our expressions.
    \subsubsection{Applying $\delta$-method} As the sum is running over a sparse set i.e. squares, we need to separate squares from the Fourier coefficients. We use Kronecker delta-symbol to rewrite the sum as, 
    \[
    S_\pi(X)=\sum_{m\sim p^3}\sum_{n\sim p^{3/2}}\lambda_\pi(m)\delta(m=n^2)
    \]
    Instead of using the Fourier expansion of Duke-Friedlander-Iwaniec to expand the delta-symbol, we first perform the ``conductor reduction trick'': we first detect the congruence $m\equiv n^2\md{p}$, and then apply the delta-symbol's expansion to the reduced quotient $\frac{m-n^2}{p}$ i.e., we rewrite the delta-symbol as, 
    \[
    \delta(m=n^2)=\delta(m\equiv n^2\md{p})\times\delta\left(\frac{m-n^2}{p}\right)
    \]
    In Section~\ref{sec:delta method}, we use the additive characters modulo $p$ to detect the congruence condition, and the Fourier expansion to expand the last delta-symbol. Generically it transform $S_\pi(X)$ into,
    \[
    \frac{1}{p^2}\underset{p\nmid q}{\sumflat_{q\sim p}}\frac{1}{q}\hspace{0.1cm}\sumstar_{a\md{pq}}\sum_{m\sim p^3}\lambda_\pi(m)e\left(\frac{am}{pq}\right)\sum_{n\sim p^{3/2}}e\left(-\frac{an^2}{pq}\right).
    \]
    Trivial estimate of this sum is $p^3\times p^{3/2}$, so due to the expansion we loss $p^3$. Therefore to obtain a non-trivial we need to recover this loss and little more.
    \subsubsection{Application of summation formul\ae} To recover this loss, our first step is to apply summation formul\ae\ to dualize both of the $m$, and $n$ sums. Poisson summation formula transforms the $n$-sum into
    \[
    \sum_{n\sim p^{3/2}}e\left(-\frac{an^2}{pq}\right)\approx\frac{p}{\sqrt q}\sum_{n\ll p^{1/2}}G(-a, -n; pq),
    \]
    where $G(-a, -n; pq)$ is the quadratic Gau\ss\ sum defined by the equation~\eqref{eq:defn of G_q(m ,n)}. But the application of Vorono\"i summation formula for $m$-sum is slightly tricky due to the ``joint ramification" issue, as moduli of additive character is neither divisible by and the conductor of $\pi$ nor co-prime to it. So, to by pass this issue we make an use of additive reciprocity and a change of basis, which ultimately yields roughly, 
    \[
    \sum_{m\sim p^3}\lambda_\pi(m)e\left(\frac{am}{pq}\right)\approx p^{3/2}\chi_\pi(q)\sum_{m\ll p}\frac{\lambda_{\bar\pi}(m)}{\sqrt{m}}e\left(-\frac{\overline{ap}m}{q}\right)T(aq, m; p),
    \]
    where, 
    \[
    T(x, y; p)=\frac{1}{\sqrt{p}}\sumstar_{\chi\md{p}}\chi(\bar{x}y)\eps(\chi)\eps(\pi\otimes\bar\chi)
    \]
    Here $\eps(\chi)$ is the sign of the Gauss sum, and $\eps(\pi\otimes\bar\chi)$ is the root number of the twisted form.
    Now combining these evaluations we have, 
    \[
    S_\pi(X)\approx \underset{p\nmid q}{\sumflat_{q\sim p}}\chi_\pi(q)\sum_{m\ll p}\sum_{n\ll p^{1/2}}\frac{\lambda_{\bar\pi}(m)}{\sqrt{m}}K(m, n, pq)
    \]
    Here $K(m, n; pq)$ is an exponential sum given by the eq.~\eqref{eq: defn of K(m, n, pq)}, and all of these has been carried out in Section~\ref{sec: sec3}. 
    \subsubsection{Evaluation of character sum} We have evaluated the exponential sum $K(m, n, pq)$ in the next Section~\ref{sec: eva of char sum}. Roughly its evaluation is given by, 
    \[
    K(m ,n, pq)\approx H(n^2/m; p)\left(\frac{m-n^2}{q}\right),
    \]
    where $H(n^2/m; p)$ is given by normalized hypergeometric exponential sum depending on the type of $\pi_p$, 
    \[
    H(\lambda; p):= \begin{cases}
         \text{Hyp}(\F_p^2\times\F_p^2; 1, \rho; \bar\mu_1, \bar\mu_2; \lambda)& \text{if $\pi_p$ is of type (a)},\\
         \text{Hyp}(\F_p^2\times\F_{p^2}; 1, \rho; \eta_\pi; \lambda)&\text{if $\pi_p$ is of type (b).}
     \end{cases}
    \]
    Square root cancellation of this sum and its geometric properties has established in Section~\ref{sec: eva of char sum}

    \subsubsection{Simplifications} Therefore by simplifications we have, 
    \[
    S_\pi(X)\approx \underset{p\nmid q}{\sumflat_{q\sim p}}\chi_\pi(q)\sum_{m\ll p}\sum_{n\ll p^{1/2}}\frac{\lambda_{\bar\pi}(m)}{\sqrt{m}}H(n^2/m; p)\left(\frac{m-n^2}{q}\right)
    \]
    Identifying $m-n^2=r$ it becomes, 
    \[
    S_\pi(X)\approx \underset{p\nmid q}{\sumflat_{q\sim p}}\sum_{r\ll p}\chi_\pi(q)\alpha_\pi(r)\left(\frac{r}{q}\right)
    \]
    where, 
    \[
    \alpha_\pi(r):=\underset{m-n^2=r}{\sum_{m\ll p}\sum_{n\ll p^{1/2}}}\frac{\lambda_\pi(m)}{\sqrt{m}}H(n^2/m; p).
    \]
    So the problem has been reduced to an estimation of a certain bilinear sum. In Lemma~\ref{thm: pointwise bd of alpha} we have proved that the sequence $\alpha_\pi(r)$ does not grow much with $p$, namely $\alpha_\pi(r)\ll p^\eps$. Now triangle inequality yields,
    \[
    S(X)\ll p^2.
    \]
    To obtain trivial bound we just need to save $\sqrt{p}$, and it can be done by using the quadratic large sieve (Lemma~\ref{thm:large sieve}). 

    So to obtain non-trivial bound we exploit more the quadratic character and the sequence $\alpha_\pi(r)$. We write $r=st^2$ with $s\sim S$, $t\sim T$, and $\mu^2(s)=1$, such that $ST^2\ll p$. We rewrite, 
    \[
    S(X)\approx \sum_{t\sim T}\underset{(q,\,pt)=1}{\sumflat_{q\sim p}}\sumflat _{s\sim S}\chi_\pi(q)\alpha_\pi(st^2)\left(\frac{s}{q}\right)
    \]
     Now to estimate it non-trivially we have treat this differently depending on the size of $T$.
     
         \subsubsection{Non-generic case} If the size of $T>p^\delta$, \textit{non-generic case}, then $S\ll p^{1-2\delta}$. By the quadratic reciprocity we rewrite $S(X)$ as, 
         \[
         S(X)\approx \sum_{t\sim T}\sumflat _{s\sim S}\alpha_\pi(st^2)\underset{(q,\,pt)=1}{\sumflat_{q\sim p}}\chi_\pi(q)\left(\frac{q}{s}\right)
         \]
         This dyadic decomposition saves $p^\delta$. Now observe that square-root of the conductor of $\chi_\pi(q)\left(\frac{q}{s}\right)$ is at-most $p^{1-\delta}$ which is smaller than size of $q$. So we could apply the Poisson summation formula over $q$, and this would save $p^\delta$. The dual length of $q$ becomes small $p^{1-2\delta}$, so quadratic large sieve at this stage saves little less, $p^{1/2-\delta}$. Finally the total saving over the triangle inequality bound $p^2$ would be 
         \[
         p^\delta\times p^\delta\times p^{1/2-\delta}=p^{1/2+\delta}.
         \]
         So in this non-generic case we obtain a non-trivial bound, 
         \[
         S(X)\ll p^{3/2-\delta}.
         \]
         The analysis of this case has been carried out in the Section~\ref{sec: non-generic}.
         \subsubsection{Generic case} Now we turn to the generic case, $T\leq p^\delta$. This has been treated in the Section~\ref{sec: generic}. We begin by applying the quadratic large sieve inequality,
         \[
         S(X)\ll pT\,\sup_{t\sim T}\left(\sum_{s\sim S}|\alpha_\pi(st^2)|^2\right)^{1/2}.
         \]
         Point-wise estimate of $\alpha_\pi(\cdot)$ bounds the $l^2$-norm trivially, namely $\sqrt{p}$. We write, 
         \[
         \sum_{s\sim S}|\alpha_\pi(st^2)|^2\approx\sum_{m_1\ll p}\sum_{m_2\ll p}{\underset{n_1^2-n_2^2=m_1-m_2}{\sum_{n_1\ll \sqrt p}\sum_{n_2\ll \sqrt{p}}}}\frac{\lambda_\pi(m_1)}{\sqrt{m_1}}\frac{\lambda_{\bar\pi}(m_2)}{\sqrt{m_2}}H(n_1^2/m_1; p)\bar{H}(n_2^2/m_2; p)
         \]
         In this stage we could have remove both the Fourier coefficients at the same time. But the diagonal contribution, $p$, would again give us the trivial bound! So to improve the diagonal, we first write $m_2$ in terms of $m_1$, $n_1$, and $n_2$,
         \[
         \sum_{m_1\ll p}\frac{\lambda_\pi(m_1)}{\sqrt{m_1}}\sum_{n_1\ll \sqrt{p}}\sum_{n_1\ll\sqrt{p}}\frac{\lambda_{\bar\pi}(m_1-n_1^2+n_2^2)}{\sqrt{m_1-n_1^2+n_2^2}}H(n_1^2/m_1; p)\bar{H}\left(n_2^2/\left(m_1-n_1^2+n_2^2\right); p\right).
         \]
         Now we remove the Fourier coefficient $\lambda_\pi(m_1)$ by using the Cauchy's inequality, therefore we bound above sum by $\sqrt{\tilde\Omega}$,
         \[
         \tilde\Omega:=\sum_{m_1\ll p}\left|\sum_{n_1, n_2\ll \sqrt{p}}\frac{\lambda_{\bar\pi}(m_1-n_1^2+n_2^2)}{\sqrt{m_1-n_1^2+n_2^2}}H(n_1^2/m_1; p)\bar{H}\left(n_2^2/\left(m_1-n_1^2+n_2^2\right); p\right)\right|^2.
         \]
         The diagonal contribution of this is $p$. To estimate the off-diagonal part we open the square and rewrite it as, 
         \[
         \sum_{m_1\ll p}\sum_{|h|\ll p}\frac{\lambda_{\pi}(m_1)}{\sqrt{m_1}}\frac{\lambda_{\bar\pi}(m_1+h)}{\sqrt{m_1+h}}\underset{n_1^2-n_2^2-n_3^2+n_4^2=h}{\sum_{n_1,\dots, n_4\ll \sqrt{p}}}\prod_{j=0}^{1} \mathfrak e(n_{2j+1}, n_{2j+2}, m_1)^{\sigma^j},
         \]
         where $\sigma=c$ is the complex conjugation, and 
         \[
         \mathfrak e(n_{j+1}, n_{j+2}, m_1):=H^\sigma(n_j^2/m_j; p)H(n_{j+1}^2/(m_1-n_j^2+n_{j+1}^2); p).
         \]
         Now we remove both of the Fourier coefficients by applying Cauchy's inequality on the outer sums. We could bound the sum by, 
         \[
         \left(\sum_{m\ll p}\sum_{|h|\ll p}\left|\underset{n_1^2-n_2^2-n_3^2+n_4^2=h}{\sum_{n_1,\dots, n_4\ll \sqrt{p}}}\prod_{j=0}^{1} \mathfrak e(n_{2j+1}, n_{2j+2}, m_1)^{\sigma^j}\right|^2\right)^{1/2}
         \]
         Opening the absolute value square and taking the $h$-sum inside, we could rewrite the expression under the square-root,
         \[
         \underset{n_1^2-n_2^2-n_3^2+n_4^2=n_5^2-n_6^2-n_7^2+n_8^2}{\sum_{n_1, \dots, n_4\ll \sqrt{p}}\,\sum_{n_5, \dots, n_8\ll \sqrt{p}}}\sum_{m\ll p}\prod_{j=0}^{1}\mathfrak e(n_{2j+1}, n_{2j+2}, m_1)^{\sigma^j}\mathfrak e(n_{2j+5}, n_{2j+6}, m_1)^{\sigma^{j+1}}.
         \]
         The size of the outer sums is about, 
         \[
         \sharp\{(n_1, \dots, n_8)\in\N^8| n_1^2+n_6^2+n_7^2+n_4^2=n_5^2+n_2^2+n_3^2+n_8^2;\, n_i\ll \sqrt{p}\}=O(p^3).
         \]
         So, the trivial estimation of the $m_1$-sum we could only recover the trivial bound for $S(X)$. Therefore obtaining a non-trivial bound for $S(X)$ has reduced to a cancellation in the following exponential sum (which could be obtained by applying Poisson summation on $m_1$-sum), 
         \[
         \sum_{x\in\F_p}\prod_{j=0}^{1}\mathfrak e(n_{2j+1}, n_{2j+2}, x)^{\sigma^j}\mathfrak e(n_{2j+5}, n_{2j+6}, x)^{\sigma^{j+1}}e_p\left(m_1x\right).
         \]
         This is an algebraic-exponential sum of sum-product type. In the Section~\ref{sec: exp sums}, we use methods from $\ell$-adic cohomology especially the Riemann hypothesis for finite fields to obtain square-root cancellation as long as the identity component of the geometric monodromy group is big, namely $\SL(2)$. 

         Therefore the in the generic case we have, 
         \[
         S(X)\ll p^{\frac{3}{2}-\frac{1}{16}+\delta}.
         \]

     Now optimizing $\delta$ we get, 
     \[
     \sum_{n\sim p^{3/2}}\lambda_\pi(n^2)\ll p^{\frac{3}{2}-\frac{1}{32}}.
     \]
     Note that the size of $\delta$ differs from the actual size because of some other easy non-generic case. 

    \begin{remark}
        The initial approach of this paper is similar to the methods in~\cite{900580107, MR5049991}. Our main innovation lies in the generic case. In this case, instead of removing both of the Fourier coefficients we remove one Fourier coefficient at first because otherwise the diagonal would contain too few points, namely $p^\eps$, yielding non savings. Therefore by repeated applications Cauchy's inequality we save enough in the diagonal through point-counting, and in the off-diagonal through an estimation of exponential sum.
    \end{remark}
	\subsection{Notations and conventions}\label{n&c}
    We use the popular notation $e(x):=e^{2\pi i x}$, and $\eps$ would denote an arbitrary small positive number. $X=O_a(Y)$ and $X\ll_{a} Y$ would denote $X\ll C(a)p^\eps Y$ for some constant $C$ defendin on $a$. If $Y\ll X\ll Y$ then we denote it by $X\asymp Y$. For brevity we often omit the underlying smooth function of the form $V(m/M)$ in a sums or a integral by suitable asymptotic notations mentioned above.
	
	\section{Symmetric-square $L$-functions}

    Let $\pi$ be an irreducible cuspidal automorphic representation of $\GL_2(\A_\Q)$ with unitary central character $\omega_\pi$ of $\A^\times/\Q^\times$. One can write it as the restricted tensor product $\pi\simeq\bigotimes_{\nu}^\prime\pi_{\nu}$, where $(\pi_\nu)_\nu$ is a sequence of unitary $\GL_2(\Q_\nu)$-representations, all but finitely many of which are unramified principle series representations. To $\pi$ one associates an $L$-function, on the right half-plane $\{s : \text{Re}(s)> 1\}$ it is given by 
    $$
    L(s, \pi)=\prod_{\nu<\infty} L_\nu(s, \pi)=\sum_{n\geq 1} \frac{\lambda_\pi(n)}{n^s}, 
    $$
    where $\lambda_\pi(n)$ are Hecke-eigenvalues. At unramified places the local $L$-functions $L_\nu(s, \pi)$ are degree-$2$ Euler factors,
    $$
    L_\nu(s, \pi)^{-1}=\left(1-\frac{\alpha_\pi(\nu)}{\nu^s}\right)\left(1-\frac{\beta_\pi(\nu)}{\nu^s}\right)
    $$
    where local Hecke-eigenvalues $\alpha_\pi(\nu)$, and $\beta_\pi(\nu)$ are known as Satake parameters, and for ramified places $\nu$ it depends on the type of local representation $\pi_\nu$. The local $L$-function at the infinite place is given by, 
    \begin{equation*}
        L_\infty(s, \pi)=\prod_{j=1}^2\Gamma_\R(s+\mu_{i, \infty}(\pi));\hspace{0.2cm}\Gamma_\R(s):=\pi^{-s/2}\Gamma(s/2).
    \end{equation*}
    where $\mu_{i, \infty}(\pi)$'s are spectral parameters of $\pi$. The completed $L$-function admits a functional equation of the shape, 
    \begin{equation*}
        \Lambda(s, \pi)=\eps(\pi)q(\pi)^{s-\frac{1}{2}}\br{\Lambda(1-\bar s, \pi)},
    \end{equation*}
    where 
    $$
    \Lambda(s, \pi)=L_\infty(s, \pi)L(s, \pi),\hspace{0.2cm}|\eps(\pi)|=1,\hspace{0.2cm}\text{and}\hspace{0.2cm}q(\pi)\in\Z_{\geq1}.
    $$
    The constant $\eps(\pi)$ is the root number, and $q(\pi)$ is the arithmetic conductor.
    
    Now we consider the symmetric-square representation, 
    \begin{equation*}
        \Sym^2:\GL(2)\to\GL(3),
    \end{equation*}
    defined as, 
    \begin{equation*}
        \Sym^2
        \begin{pmatrix}
            \alpha&\\
            &\beta
        \end{pmatrix}
        :=\begin{pmatrix}
            \alpha^2&&\\
            &\alpha\beta&\\
            && \beta^2
        \end{pmatrix}
    \end{equation*}
    The symmetric-square $L$-function is given by an Euler product of local $L$-functions of degree $3$, and converging locally absolutely on some right half-plane $\{s: \text{Re}(s)\gg1\}$
    \begin{equation*}
        L(s, \pi, \sym^2)=\prod_{\nu<\infty}L_\nu(s, \pi, \sym^2)
    \end{equation*}
    For unramified places $\nu$ local $L$-function is defined by, 
    \begin{equation*}
        \begin{split}
            L_\nu(s, \pi, \sym^2)^{-1}&=\text{det}\left(I-{p^{-s}}\,{\sym^2(\text{diag}(\alpha_\pi(v), \beta_\pi(v)))}\right)\\
            &=\left(1-\frac{\alpha_\pi(v)^2}{p^s}\right)\left(1-\frac{\alpha_\pi(v)\beta_\pi(v)}{p^s}\right)\left(1-\frac{\beta_\pi(v)^2}{p^s}\right).
        \end{split}
    \end{equation*}
    This definition can also be realized as a quotient of local $L$-functions,
    \begin{equation*}
        L_\nu(s, \pi, \sym^2)=\frac{L_\nu(s, \pi\otimes\pi)}{L_\nu(s, \omega_\pi)}.
    \end{equation*}
    Note that the second definition is valid also for ramified places. The local $L$-function at Archimedean place can be written as a product of Gamma factors and spectral parameters of $\pi$. Completion of the symmetric-square $L$-function satisfies the functional equation,
    \begin{equation*}
        \Lambda(s, \pi, \sym^2)=\eps(\pi, \sym^2)q(\pi, \sym^2)^{s-\frac{1}{2}}\br{\Lambda(1-\bar s, \pi, \sym^2)},
    \end{equation*}
    where $\eps(\pi, \sym^2)$ is the root number, and $q(\pi, \sym^2)$ is the arithmetic conductor. 
    
    Gelbert and Jacquet in~\cite{MR0533066} have showed that $L(s, \pi, \sym^2)$ is an $L$-function of a $\GL_3(\A_\Q)$ automorphic representation which is denoted by $\sym^2\pi$. And we denote this $L$-function by,
    \begin{equation*}
        L(s, \sym^2\pi):=L(s, \pi, \sym^2).
    \end{equation*}
    It can be meromorphically continued to entire $\C$ except the possible pole at $s=1$.\\
    
    Let $p$ be a large prime. From now on we restrict ourselves to the representation $\pi$ as in the Theorem~\ref{thm: thm1}, unless otherwise specified.
    Combining the definition of local $L$-factors we have the following factorization (see~\cite{MR0382176}),
    \begin{equation*}
        L(s, \sym^2\pi)=\left(\sum_{d|p^\infty}\frac{\lambda_{\sym^2\pi}(d)}{d^s}\right)L(2s, \chi_\pi^2)\sum_{n\geq 1}\frac{\lambda_\pi(n^2)}{n^s},
    \end{equation*} 
    with~\cite[eq. (3.3)]{MR2119720}
    \begin{equation}\label{eq: bound}
        \lambda_{\sym^2\pi}(d)\ll_\eps d^{2\theta+\eps};\hspace{0.2cm}\theta\leq7/64\;\; (\text{Kim and Sarnak}).
    \end{equation}
    This allows us to approximate the $L$-value, in critical strip, by finite Dirichlet series, 
    \begin{lemma}[Approximate functional equation]
		Let $A\gg1$ be a large number, and $G(u)=\left(\cos{\frac{\pi u}{4A}}\right)^{-5A}$. We have, 
		\begin{equation}\label{eq: afe}
			L\left(\frac{1}{2},\; \sym^2\pi\right)=\sum_{n\geq 1}\frac{\lambda_\pi(n^2)}{\sqrt{n}}V_{1/2}\left(\frac{n}{\sqrt{\mathcal Q}}\right)+\eps(\pi, \sym^2)\sum_{n\geq 1}\frac{\lambda_{\bar \pi}(n^2)}{\sqrt{n}}\tilde{V}_{1/2}\left(\frac{n}{\sqrt{\mathcal Q}}\right).
		\end{equation}
		where $\mathcal Q=Q(\sym^2\pi)$, and the weight functions are given by,
		\begin{equation*}
			\begin{split}
				&V_{1/2}(y)=\sum_{d|p^\infty}\frac{\lambda_{\sym^2\pi}(d)}{\sqrt{d}}W(dy)\\
				&W(y)=\frac{1}{2\pi i}\int_{(2)}\frac{L_\infty(1/2+u, \sym^2\pi)}{L_\infty(1/2, \sym^2\pi)}L(1+2u, \chi_\pi^2)G(u)y^{-u}\frac{du}{u}.
			\end{split}
		\end{equation*}
		$\tilde V_{1/2}$ is defined in a similar way as $V$ except $\lambda_{\sym^2\pi}$(resp. $\chi_\pi$) is replaced by $\lambda_{\sym^2\bar \pi}$ (resp. $\bar\chi_\pi$). Set $q_\infty=q_\infty(1/2, \sym^2\pi)$, for any $j\geq0$ we have, 
		$$
		y^j V_{1/2}^{(j)}(y)\ll_{j, A} \left(1+\frac{y}{\sqrt{q_\infty}}\right)^{-A}.
		$$
	\end{lemma}
	\begin{proof}
		\eqref{eq: afe} is a standard application of contour shifting and functional equation (see~\cite[Theorem 5.3]{MR2061214}) for a proof, and the last estimate of the weight function follows from the bound~\eqref{eq: bound} see the proof of~\cite[Lemma 3.1]{MR2119720}. 
	\end{proof}
    
	From the above lemma we could truncate the $n$-sum up-to $\mathcal Q^{1/2+\eps}$ at a cost of negligible error, 
	$$
	L\left(1/2, \sym^2\pi\right)\ll_\eps \left|\sum_{n\ll \mathcal Q^{1/2+\eps}}\frac{\lambda_\pi(n^2)}{\sqrt n}V\left(\frac{n}{\sqrt{\mathcal Q}}\right)\right|
	$$
	
	\begin{lemma}[Smooth dyadic-partition of unity]
		There is a smooth non-negative function $W$ supported on $[1/2, 2]$, satisfying $x^jW^{(j)}(x)\ll 1$ and a dyadic-partition $(2^k)_{k\geq 0}$ such that for any $x\geq 1$ we have, 
		$$
		\sum_{k\geq0} W\left(\frac{x}{2^k}\right)=1.
		$$
	\end{lemma}
	\begin{proof}
		\cite[Lemme 2]{MR0783533}.
	\end{proof}
	Applying this lemma we decompose the long $n$-sum into smooth dyadically localized sums, 
	\begin{equation}\label{eq:appx by s_f(x)}
		L(1/2, \sym^2\pi)\ll_{A, \eps}p^\eps\sum_{X\ll \mathcal Q^{1/2+\eps}}\frac{S_\pi(X)}{\sqrt{X}},
	\end{equation}
	where, 
	\begin{equation}\label{eq:s_f(x)}
		S_\pi(X):=\sum_{n}\lambda_\pi(n^2)V\left(\frac{n}{X}\right)
	\end{equation}
	with a smooth weight function $V$ supported on $[1/2, 2]$ given by,
	$$
	V(x):=\frac{1}{\sqrt x}W(x)V_{1/2}\left(\frac{xX}{\sqrt \mathcal Q}\right),\hspace{2mm}\text{satisfies}\hspace{1mm} x^jV^{(j)}(x)\ll_j (q_\infty)^{j}.
	$$
    Note that by the trivial bound of~\eqref{eq:s_f(x)} $S_\pi(X)\ll X^{1+\eps}$, and the approximation \eqref{eq:appx by s_f(x)} we recover the convexity bound 
    $$
    L(1/2, \sym^2\pi)\ll_\eps \mathcal Q^{{1}/{4}+\eps}=q(\sym^2\pi)^{1/4+\eps}.
    $$
	So the subconvexity problem will follow once we are able to bound $S_\pi(X)$ non-trivially and this will be done in our upcoming sections.
	First we need to compute the arithmetic conductor of $\sym^2\pi$, 
	\begin{lemma}[Computation of $q(\sym^2\pi)$]\label{thm: computation of conductor}
		Let $\pi\simeq\bigotimes_\nu^\prime\pi_\nu$ be a cuspidal automorphic representation for $\GL_2(\Q)$ of artihmetic conductor $q(\pi)=p^2$, with the central character $\omega_\pi$, the adelic lift of a primitive non-quadratic Dirichlet character $\chi_\pi$ modulo $p$. Suppose $\pi_p$ is either of type (a) or type (b) then the arithmetic conductor $q(\sym^2\pi)$ is $p^3$.
	\end{lemma}
	
	\begin{proof}
		Let $\pi_{p}$ be of type (a) i.e., the ramified principal series $\pi(\tilde\mu_1, \tilde\mu_2)$ and the corresponding local Weil-Delinge representation, say $\sigma_p$, is given by $\tilde\mu_1\oplus \tilde\mu_2$. The symmetric-square lift of this representation $\sym^2\sigma_p$ is the local Weil-Deligne representation associated to $\sym^2\pi_{p}$. 
		$$
		\sym^2\sigma_p=\tilde\mu_1^2\oplus\tilde\mu_1\tilde\mu_2\oplus\tilde\mu_2^2.
		$$
		The conductor exponent $\mathfrak f(\sym^2\sigma_p)$ is equal to $\mathfrak f(\mu_1^2)+\mathfrak f(\mu_1\mu_2)+\mathfrak f(\mu_2^2)=3$.

		Now we assume $\pi_{p}$ is if type (b) i.e., in particular a supercuspidal representation. To compute the conductor in this case we follow the definition of functorial lifts,
		\begin{equation*}
			\pi_{p}\otimes\pi_{p} = \sym^2\pi_{p}\,\boxplus\widetilde\chi_{\pi, p}.
		\end{equation*}
		Definition implies a relation between conductor exponents $\mathfrak f(\sym^2\pi_{p})=\mathfrak f(\pi_{ p}\otimes\pi_{p})-1$. So, the computation has reduced to the computation of $\mathfrak f(\pi_{p}\otimes\pi_{ p})$, which we compute using~\cite{MR1606410}. Following the section 6.1 of loc. cit. we attach to $\pi_{p}$ a simple stratum $[\mathfrak{A}, m, 0, \beta]$ where $m$ is called as \textit{depth} of the representation $\pi_{p}$ (to avoid the confusion we call it depth instead of \textit{level}). Since, the conductor exponent pf $\pi_p$ is $\mathfrak f(\pi_{p})=2$, it is a depth zero representation, which follows from the definition or can be deduced from the conductor formula of Godement-Jacquet $\mathfrak f(\pi_{p})=2\left(1+\frac{m}{e}\right)$, where $e$ is the ramification index $e(\Q_p[\beta]/\Q_p)$ which is either $1$ or $2$.\\
		In the hypothesis we have assumed that $\chi_\pi$ to be non-quadratic which implies that $\pi_{p}$ and its dual $\bar{\pi}_{p}$ are \textit{completely distinct} in the sense of~$\S 6.2$~in loc. cit. So,~\cite[\S6.5, Theorem(ii)]{MR1606410}~gives $\mathfrak f(\pi_{p}\times\pi_{p})=4$. This completes the proof.
	\end{proof}

	\section{Setting up the problem}\label{sec: sec3}
	
	\subsection{The $\delta$-method}\label{sec:delta method}
	Let $\delta:\Z\to\{0, 1\}$ be a function defined by 
	\[
	\delta(n)=\begin{cases}
		1 &\text{if $n=0$,}\\
		0 &\text{otherwise.}
	\end{cases}
	\]
	We seek a Fourier expansion of $\delta(n)$ in the range $[-N, N]$ for $N\gg 1$. Among the various such expansions of $\delta$, we choose the version of Duke, Friedlander, and Iwaniec~\cite[(20.158)]{MR2061214}. Set $Q=\sqrt N$, then we have
	\begin{equation}\label{eq:dfi}
		\delta(n)=\frac{1}{Q}\sum_{q\leq Q}\frac{1}{q}\sumstar_{a \md{q}}e\left(\frac{an}{q}\right)\int_\R g(q, x)e\left(\frac{nx}{qQ}\right)\, dx
	\end{equation}
	for $n\in\Z\cap[-N, N]$, and $*$ on the second sum indicates $(a, q)=1$. In this expansion the only thing which is not explicitly given is the function $g(q, x)$. But, the following properties of $g(q, x)$ (see~\cite[(20.158) and (20.159)]{MR2061214} and~\cite[Lemma 15]{MR4333413}) will suffices for our purpose, 
	\begin{equation}\label{eq: g(q, x)-1}
		g(q, x)=1+O\left(\frac{Q}{q}\left(\frac{q}{Q}+|x|\right)^A\right),\,\,\, g(q, x)\ll |x|^{-A}\,\,\, \text{for}\,\,\,A>1
	\end{equation}
	The second property of the equation says that the support of the integral is essentially $[-N^\eps, N^\eps]$, so up-to a negligible error term we can rewrite the expansion~\eqref{eq:dfi} as 
	\begin{equation}\label{eq:DFI}
		\delta(n)=\frac{1}{Q}\sum_{q\leq Q}\frac{1}{q}\sumstar_{a \md{q}}e\left(\frac{an}{q}\right)\int_\R W(x)g(q, x)e\left(\frac{nx}{qQ}\right)\, dx,
	\end{equation}
	where $W(x)$ is a smooth function supported in the interval $[-2N^\eps, 2N^\eps]$ and $W\equiv 1$ in $[-N^\eps, N^\eps]$. It also follows from~\eqref{eq: g(q, x)-1} that $q\ll Q^{1-\eps}$ and $x\ll Q^{-\eps}$, then $g(q, x)$ can be replaced by $1$ at a cost of negligible error term. In the complementary range the following estimate will suffice,
	\begin{equation}\label{eq:g(q, x)-2}
		x^j\frac{\partial^j}{\partial x^j}g(q, x)\ll \log Q \,\,\, \text{min}\left(\frac{Q}{q}, \frac{1}{|x|}\right).
	\end{equation}
	By Cauchy and Parseval's identity we have
	\begin{equation}\label{eq:g(q, x)-3}
		\int_\R (|g(q, x)|+|g(q, x)|^2)\, dx\ll Q^\eps
	\end{equation}
	it means $g(q, x)$ is $1$ on average in $L^1$ and $L^2$-sense.\\

	Using $\delta$-symbol, we first rewrite~\eqref{eq:s_f(x)} as follows, 
	\begin{equation}\label{eq:osc}
		S_\pi(X)=\sum_m\sum_n\lambda_\pi(m)\delta(m-n^2)U\left(\frac{m}{X^2}\right)V\left(\frac{n}{X}\right)
	\end{equation}
	for some nice function $U$ supported in $[1/4, 9/4]$ and $U\equiv 1$ on the support of $V$. Set $l=m-n^2$, instead of applying~\eqref{eq:DFI} directly into this, it is beneficial to rewrite $\delta(m-n^2)$ as 
	\begin{equation*}
		\delta(l)=\delta(l\equiv 0\md{p})\times \delta\left(\frac{l}{p}\right).
	\end{equation*}
	Factoring the delta-symbol in this fashion, known as ``conductor-lowering trick", helps us to \textit{reduce} the size of the modulus of additive characters appearing in the Fourier expansion~\eqref{eq:DFI}. Now expanding the second delta-symbol as in~\eqref{eq:DFI} and along with, detecting the congruence condition $\delta(l\equiv 0\md{p})$ by the orthogonality of additive characters, helps us to rewrite $\delta(l)$ as, 
	\begin{equation*}
		\delta(l)=\frac{1}{Q}\sum_{q\leq Q}\frac{1}{pq}\sumstar_{a\md{q}}\sum_{b\md{p}}e\left(\frac{a+bq}{pq}l\right) \int_\R W(x)g(q, x)e\left(\frac{lx}{pqQ}\right)\, dx.
	\end{equation*}
	Since $l\ll X^2$, we set $Q:=\sqrt{X^2/p}=X/\sqrt{p}\ll p^{1+\eps}$. Now $(a+bq)$ runs over $p\phi(q)$ many distinct congruence classes modulo $pq$. When $(p, q)=1$, among them there are $\phi(q)$ many non-invertible congruence classes (in other words, there are $\phi(pq)$ many invertible congruence classes) corresponds to $b=0$. And in the other case, namely $p\mid q$ there are $\phi(pq)=p\phi(q)$ invertible congruence classes. Depending on these cases we further decompose the above expansion, 
	\begin{equation*}
		\begin{split}
			\delta(l)&=\frac{1}{Q}\underset{p\nmid q}{\sum_{q\leq Q}}\frac{1}{pq}\sumstar_{a\md{pq}}e\left(\frac{al}{pq}\right)\int_\R W(x)g(q, x)e\left(\frac{lx}{pqQ}\right)\, dx\\
			&+\frac{1}{Q}\underset{p\nmid q}{\sum_{q\leq Q}}\frac{1}{pq}\sumstar_{a\md{q}}e\left(\frac{al}{pq}\right)\int_\R W(x) g(q, x) e\left(\frac{lx}{pqQ}\right)\, dx\\
			&+\frac{1}{Q}\sum_{q\leq Q/p}\frac{1}{p^2q}\sumstar_{a\md{p^2q}}e\left(\frac{al}{p^2q}\right)\int_\R W(x)g(q, x) e\left(\frac{lx}{p^2qQ}\right)\, dx. 
		\end{split}
	\end{equation*}
	Applying this to~\eqref{eq:osc} we obtain, 
	\begin{equation}\label{eq: decomp of S_f(X)}
		S_\pi(X)=\M(X)+\Er_1(X)+\Er_2(X)
	\end{equation}
	where the $\M$-term is given by,
	\begin{equation}\label{eq:main term}
		\begin{split}
			\M(X):=\frac{1}{Q}\int_\R W&(x) \underset{p\nmid q}{\sum_{q\leq Q}}\frac{g(q, x)}{pq}\sumstar_{a\md{pq}}\\
			&\times\sum_m \lambda_f(m)e\left(\frac{am}{pq}\right)e\left(\frac{xm}{pqQ}\right)U\left(\frac{m}{X^2}\right)\\
			&\times\sum_n e\left(-\frac{an^2}{pq}\right)e\left(-\frac{xn^2}{pqQ}\right)V\left(\frac{n}{X}\right)\, dx.
		\end{split}
	\end{equation} 
	The other remaining terms are given by
	\begin{equation}\label{eq:error 1}
		\begin{split}
			\Er_1(X):=\frac{1}{Q}\int_\R W&(x) \underset{p\nmid q}{\sum_{q\leq Q}}\frac{g(q, x)}{pq}\sumstar_{a\md{q}}\\
			&\times\sum_m \lambda_f(m)e\left(\frac{am}{pq}\right)e\left(\frac{xm}{pqQ}\right)U\left(\frac{m}{X^2}\right)\\
			&\times\sum_n e\left(-\frac{an^2}{pq}\right)e\left(-\frac{xn^2}{pqQ}\right)V\left(\frac{n}{X}\right)\, dx
		\end{split}
	\end{equation}
	and
	\begin{equation}\label{eq:error 2}
		\begin{split}
			\Er_2(X):=\frac{1}{Q}\int_\R W&(x) {\sum_{q\leq Q/p}}\frac{g(pq, x)}{p^2q}\sumstar_{a\md{p^2q}}\\
			&\times\sum_m \lambda_f(m)e\left(\frac{am}{p^2q}\right)e\left(\frac{xm}{p^2qQ}\right)U\left(\frac{m}{X^2}\right)\\
			&\times\sum_n e\left(-\frac{an^2}{p^2q}\right)e\left(-\frac{xn^2}{p^2qQ}\right)V\left(\frac{n}{X}\right)\, dx.
		\end{split}
	\end{equation}
	In the following sections, we focus on the analysis of the $\M$-term, which is the hardest and is responsible for the final bound. The other remaining terms will contribute to error which are discussed briefly in~\S\ref{sec: End remarks}.

	\subsection{Summation formul\ae}
	To begin our analysis we first need to dualize both the $m$ and $n$-sum.
    
	\subsubsection{Analysis of $m$-sum}
	The key ingredient we need to dualize the $m$-sum in~\eqref{eq:main term} is the Vorono\"i summation formula which captures the automorphy of a cusp form in a form of an identity between a weighted sum of Fourier coefficients of a cusp form and another weighted sum involving Fourier coefficients of its dual form. We record a version of this obtained in~\cite[\S A.3, 4, 5]{MR1915038},
	\begin{theorem}[Vorono\"i summation formula]\label{thm:voronoi}
		Let $\pi$ be an irreducible cuspidal automorphic representation of $\GL_2(\A_\Q)$ with central character $\omega_\pi$, a adelic lift a Dirichlet character $\chi_\pi$. Let $q(\pi)$ be the arithmetic conductor of $\pi$, and $(a,\, q)=(q,\, q(\pi))=1$. Then for any $X\gg 1$ and $U\in C_c^\infty(\R_{\geq0})$ we have the following identity,
		\begin{equation}
			\begin{split}
				\sum_{n\geq 1}\lambda_\pi(n)e\left(\frac{an}{q}\right)U\left(\frac{n}{X}\right)&=\chi_\pi(-q)i^{k(\pi_\infty)}\eps(\pi)\frac{X}{q\sqrt {q(\pi)}}\\&\times\underset{\pm}{\sum_{n\geq 1}}\lambda_{\bar \pi}(n)e\left(\pm\frac{\overline{a\,q(\pi)}n}{q}\right)U\pm\left(\frac{m}{q(\pi)q^2/X}\right)
			\end{split}
		\end{equation}
		where the integral transform is given by Bessel transforms,
        \begin{equation}
            U_{\pm}(\xi)=\int_0^\infty U(y)\mathcal K_{\pm}(4\pi\sqrt{y\xi})\, dy.
        \end{equation}
        The Bessel kernels depends on the Archimedean local component $\pi_\infty$,
		\begin{itemize}
			\item If $\pi_\infty$ belongs to Discrete series i.e., $\pi$ corresponds to a holomorphic form of weight $\kappa$, then $k(\pi_\infty)=\kappa$, and the kernels are:
			\begin{equation*}
				\mathcal K_+(x)=2\pi i^\kappa J_{\kappa-1}(x),\hspace{0.2cm}\mathcal K_-(x)=0.
			\end{equation*}
			\item If $\pi_\infty$ belongs to Principal series i.e., $\pi$ corresponds to a Maa\ss\ form with eigenvalue $1/4+t^2$ then $k(\pi_\infty)=0$ and the kernels are:
			\begin{equation*}
				\mathcal K_+(x)=\frac{-\pi}{\sin(\pi it)}\left(J_{2it}(x)-J_{-2it}(x)\right),
			\end{equation*}
            and 
            \begin{equation*}
                \mathcal K_-(x)=4\beta_\pi\cosh(\pi t)K_{2it}(x),
            \end{equation*}
            where $\beta_\pi=\pm1$ is the eigenvalue of the corresponding Maa\ss\ form under the reflection operator depending on this we call $\pi$ even or odd; $J_\nu$ and $K_\nu$ are the standard Bessel functions.
		\end{itemize}
	\end{theorem}
	
	In our case, let us denote 
	$$
	T:=\sum_{m}\lambda_\pi(m)e\left(\frac{am}{pq}\right)e\left(\frac{xm}{pqQ}\right)U\left(\frac{m}{X^2}\right)
	$$
	
	In this case, the level of the form is $q(\pi)=p^2$, and the modulus of additive character is $pq$, so they are jointly ramified at the prime $p$. As $(q, p)=1$, the level is neither co-prime to the modulus nor divides it, and therefore we can't apply either the aforementioned identity or its more general version~\cite[Theorem A.4]{MR1915038}. However, we can easily bypass this issue of joint ramification by a simple application of additive reciprocity or the Chinese remainder theorem, followed by an application of the multiplicative Fourier transform.\\
	By additive reciprocity, we have
	\begin{equation}\label{eq:crt}
		\frac{1}{pq}=\frac{\bar{q}}{p}+\frac{\bar{p}}{q}\md{1}.
	\end{equation}
	If $p\nmid u$, then by the multiplicative Fourier transform, we can expand the additive character of modulus $p$ in terms of multiplicative characters, 
	\begin{equation}\label{eq:mult Fourier}
		e\left(\frac{u}{p}\right)=-\frac{1}{\phi(p)}+\frac{\sqrt{p}}{\phi(p)}\sumstar_{\chi\md{p}}\eps(\chi)\bar{\chi}(u),
	\end{equation}
	where $\eps(\chi)$ is the normalized Gau\ss\ sum. 
    \begin{lemma}
        Let $\pi$ be an irreducible cuspidal automorhic form of conductor $q(\pi)=p^{\mathfrak f(\pi)}$, and central character $\omega_\pi$, adelic lift of a primitive Dirichlet character $\chi_\pi$ modulo $p^{\mathfrak f(\chi_\pi)}$. Let $\mathfrak f(\pi)>\mathfrak f(\chi_\pi)\geq 1$, then 
        $$
        \lambda_{\pi}(m)=0,\,\hspace{0.2cm}\text{if}\hspace{0.2cm}\,p|m.
        $$
    \end{lemma}
    \begin{proof}
        Theorem 4.6.17(3) of~\cite{MR1021004} gives $\lambda_\pi(p)=0$. Now the Hecke-relation
        \begin{equation*}
            \lambda_\pi(m)\lambda_\pi(n)=\sum_{d|(m,\,n)}\chi_\pi(d)\lambda_\pi(mn/d^2).
        \end{equation*}
        along with $\lambda_\pi(p)=0$ concludes the lemma.
    \end{proof}
    By the above lemma we have $p\nmid m$ in $T$ and therefore by applying~\eqref{eq:crt} and~\eqref{eq:mult Fourier} we have, 
	\begin{equation}
		T=T_*+T_1
	\end{equation}
	where, 
	\begin{equation}\label{eq:T_*}
		T_*:=\frac{\sqrt{p}}{\phi(p)}\sumstar_{\chi\md{p}}\eps(\chi)\chi(\bar a q)\sum_m \lambda_\pi(m)\bar\chi(m)e\left(\frac{am\bar p}{q}\right)e\left(\frac{mx}{pqQ}\right)U\left(\frac{m}{X^2}\right),
	\end{equation}
	and
	\begin{equation}
		T_1:= -\frac{1}{\phi(p)}\sum_m \lambda_\pi(m)e\left(\frac{am\bar p}{q}\right)e\left(\frac{mx}{pqQ}\right)U\left(\frac{m}{X^2}\right). 
	\end{equation}
	To apply the Vorono\"i summation formula on ~\eqref{eq:T_*} we need the following lemma, 
	\begin{lemma}[Computation of $q(\pi_p\otimes\tilde\chi)$]
		Let $\pi\simeq\bigotimes_\nu^\prime\pi_\nu$ be as usual. Then
		\begin{itemize}
			\item If $\pi_{p}$ is a principal series representation, then for all but at most two primitive characters $q(\pi\otimes\tilde\chi)=p^2$, and for those two exceptional characters $q(\pi\otimes\tilde\chi)=p$.
			\item If $\pi_{p}$ is twist minimal supercuspidal representation then $q(\pi\otimes\tilde\chi)=p^2$.
		\end{itemize}
	\end{lemma}
	\begin{proof}
		First we assume $\pi_{p}$ is minimal supercuspidal i.e. $\pi_{p}$ is supercuspidal such that $q(\pi_{p})\leq q(\pi_{p}\otimes\xi)$
		for any character $\xi$ of $\Q_p^\times$, so in particular we have $q(\pi_{p}\otimes\tilde\chi)\geq p^2$. Now  by~\cite[Proposition 3.4.]{MR0476703} we have $q(\pi_{p}\otimes\tilde\chi)\leq \min\{2q(\tilde\chi), q(\pi_{p})\}=p^2$. This established the case of supercuspidal representation.\\
		Now, let $\pi_{p}$ be a Principle series, $\pi_{p}=\pi(\tilde\mu_{1}, \tilde\mu_{2})$, then $\mathfrak f(\mu_{1})+\mathfrak f(\mu_{2})=2$ and $\mathfrak f(\mu_{1}\mu_{2})=1$, so $\mathfrak f(\mu_{1})=\mathfrak f(\mu_{2})=1$. If $\chi\notin\{\mu_{1}^{-1}, \mu_{2}^{-1}\}$ then $\mathfrak f(\chi\mu_{1})=\mathfrak f(\chi\mu_{2})=1$ which in turn implies the result and hence this completes the proof. 
	\end{proof}
	Depending on the local representations, we denote the collection of all such exceptional characters by $E$, and by the above lemma, we have $\sharp E\leq 2$. We write, 
	\begin{equation}\label{eq:t_*=t_0+t_2}
		T_*=\frac{\sqrt{p}}{\phi(p)}\sumstar_{\chi\notin E}\cdots + \frac{\sqrt{p}}{\phi(p)}\sumstar_{\chi\in E}\cdots=:T_0+T_2
	\end{equation}
	So we have decomposed our $m$-sum $T$ as, 
	\begin{equation*}
		T=T_0+T_1+T_2.
	\end{equation*}
	Among these terms, the crucial one is $T_0$ because the contribution of the remaining terms can be handled easily as $T_1$ (resp. $T_2$) comes with a prior saving of $p$ (resp. $\sqrt{p}$) and since their analysis will be quite similar to $T_0$, so we drop these terms from now on.\\
	Following the convention~\ref{n&c} we call $e\left(\frac{xyX^2}{pqQ}\right)U(y)$ by $U(y)$ and applying the Theorem~\ref{thm:voronoi} to $T_0$ we have, 
	\begin{equation*}
		T_0=i^{k(\pi_\infty)}\chi_\pi(-q)\frac{X^2}{\phi(p)q}\underset{\pm}{\sum_m}\lambda_{\bar\pi}(m)T(aq, m; p)U_{\pm}\left(\frac{m}{p^2q^2/X^2}\right)
	\end{equation*}
	where, 
	\begin{equation}
		T(x, y; p):=\frac{1}{\sqrt{p}}\sumstar_{\chi\notin E}\bar{\chi}(x)\chi(y)\eps(\chi)\eps(\pi\otimes\widetilde{\chi}^{-1}).
	\end{equation}
	
	\begin{lemma*}[Properties of Bessel functions]
		For $x>0$, and $r\in\R\cup i\R$
		\begin{equation*}
			J_r(x)=e^{ix}U_r(x)+e^{-ix}U_{-r}(x),\hspace{1.5mm}\text{and}\hspace{1.5mm}x^{j}K_r^{(j)}\ll_{j, r}\frac{e^{-x}}{\sqrt{x}},
		\end{equation*}
		where, 
		\begin{equation*}
			x^jU_{\pm r}^{(j)}(x)\ll_{j, r}1/\sqrt{x}.
		\end{equation*}
	\end{lemma*}
	We proceed our calculation with $U_{+}$, by the above lemma we could proceed with $U_{-}$ similarly as below. In this case the analysis of discrete series, and principal series would be similar, so it would be enough to treat the later case. Unwinding the definition of $U_+$ and extracting the oscillations of $J$-Bessel function, we see that this part is essentially sum of four sums of following type, 
	\begin{equation*}
		\begin{split}
			T_0=\chi_f(q)\frac{X^{3/2}}{\phi(p)}\sqrt{\frac{p}{q}}\sum_{m}\frac{\lambda_{\bar\pi}(m)}{m^{1/4}}&e\left(-\frac{\overline{ap}m}{q}\right)T(aq, m; p)\int_0^\infty U(y)\sqrt{\frac{Xm^{1/2}}{pq}}\\
			&\times U_{\pm 2it}\left(\frac{4\pi X\sqrt{my}}{pq}\right)e\left(\frac{xyX^2}{pqQ}\pm\frac{2X\sqrt{my}}{pq}\right)\, dy
		\end{split}
	\end{equation*}
	We now record a lemma which will be used throughout the paper, 
	
	\begin{lemma*}[Integration by parts bound]
		Let $Y\geq1$, $X, Z>0$, and suppose that $w$ is a smooth function of compact supported on $[Z, 2Z]$ so that $w^{(j)}(t)\ll \left(\frac{X}{Z}\right)^j$. Also suppose that $\phi$ is a smooth function and satisfies, for $j\geq 2$, $\phi^{(j)}(t)\ll \frac{Y}{Z^j}$, for some $R$ with $Y/X\geq R$ and all $t$ in the support of $w$. If $|\phi^\prime(t)|\gg Y/Z$ for all $t\in [Z, 2Z]$, then 
		\begin{equation*}
			\int_\R w(t)e\left(\phi(t)\right)\, dt\ll_A \frac{Z}{R^A}
		\end{equation*}
	\end{lemma*}
	
	By the application of this lemma, our $m$-sum becomes negligible if $m\gg p^{1+\eps}$, so we restrict $m\ll p^{1+\eps}$ at a cost of negligible error term $O(p^{-A})$,
	\begin{equation}\label{eq:after voronoi}
		\begin{split}
			T_0=\chi_f(q)\frac{X^{3/2}}{\phi(p)}\sqrt{\frac{p}{q}}\sum_{m\ll p^{1+\eps}}&\frac{\lambda_{\bar f}(m)}{m^{1/4}}e\left(-\frac{\overline{ap}m}{q}\right)T(aq, m; p)\\
			&\times\int_0^\infty U_{m, q}(y)e\left(\frac{xyX^2}{pqQ}\pm\frac{2X\sqrt{my}}{pq}\right)\, dy.
		\end{split}
	\end{equation}
    where, 
    \begin{equation*}
        U_{m, q}(y):=U(y)\sqrt{\frac{Xm^{1/2}}{pq}}U_{\pm 2it}\left(\frac{4\pi X\sqrt{my}}{pq}\right).
    \end{equation*}
	
	\subsubsection{Analysis of $n$-sum}
	We call the $n$-sum from the equation~\eqref{eq:main term} by $G$, 
	\begin{equation*}
		G:= \sum_n e\left(-\frac{an^2}{pq}\right)e\left(-\frac{xn^2}{pqQ}\right)V\left(\frac{n}{X}\right).
	\end{equation*}
	We need the smooth variant of Poisson summation formula, 
	\begin{proposition}[Poisson summation formula]
		Let $K:\Z\to\C$ be a $q$-periodic function on $\Z$, and $V\in C_c^\infty(\R)$ be a smooth function. For $X\gg1$,
		\begin{equation*}
			\sum_{n\in\Z}K(n)V\left(\frac{n}{X}\right)=\frac{X}{\sqrt{q}}\sum_{n\in\Z}\widehat{K}(n)\widehat{V}\left(\frac{n}{q/X}\right),
		\end{equation*}
		where, $\widehat{K}$ is the Fourier transform on $\Z/q\Z$,
		\begin{equation*}
			\widehat{K}(n):=\frac{1}{\sqrt{q}}\sum_{x\md{q}}K(x)e\left(-\frac{nx}{q}\right).
		\end{equation*}
	\end{proposition}
	Applying this lemma we have, 
	\begin{equation*}
		G=\frac{X}{\sqrt{pq}}\sum_n G(-a, -n; pq)\int_\R V(z)e\left(z\xi X-\frac{xz^2X^2}{pqQ}\right)\, dz
	\end{equation*}
	Where, $G(-a, -n; pq)$ is the quadratic Gau\ss\, sum, defined as
	\begin{equation}\label{eq:defn of quadratic gauss}
		G(a, b; c):=\frac{1}{\sqrt{q}}\sum_{x\md{c}}e\left(\frac{ax^2+bx}{c}\right).
	\end{equation} 
	By integration by parts bound we can restrict the $n$-sum up-to $p^{1/2+\eps}$ at a negligible cost, 
	\begin{equation}\label{eq:poisson}
		G=\frac{X}{\sqrt{pq}}\sum_{n\ll p^{1/2+\eps}}G(-a, -n; pq)\int_\R V(z)e\left(z\xi X-\frac{xz^2X^2}{pqQ}\right)\, dz.
	\end{equation}
	Plugging~\eqref{eq:after voronoi} and~\eqref{eq:poisson} in~\eqref{eq:main term} we have, 
	\begin{equation}
		\M(X)=\frac{X^{3/2}}{\phi(p)}\sum_{q\leq Q}\frac{\chi_\pi(q)}{q^{3/2}}\sum_{m\ll p^{1+\eps}}\sum_{n\ll p^{1/2+\eps}}\frac{\lambda_{\bar\pi}(m)}{m^{1/4}}K(m, n; pq)I(m, n, q)
	\end{equation}
	where, 
	\begin{equation}\label{eq: defn of K(m, n, pq)}
		K(m, n; pq):=\frac{1}{\sqrt{pq}}\sumstar_{a\md{pq}}e\left(-\frac{\overline{ap}m}{q}\right)T(aq, m; p)G(-a, -n; pq)
	\end{equation}
	and
	\begin{equation}
		\begin{split}
			\mathcal I(m, n, q)&:=\int_\R W(x)g(q, x)\int_0^\infty U_{m, q}(y)\int_\R V(z)\\&
			\times e\left(\frac{xX^2}{pqQ}(y-z^2)\pm \frac{2X\sqrt{my}}{pq}+\frac{nzX^2}{pq}\right)\, dx\,dy\,dz.
		\end{split}
	\end{equation}
	By smooth dyadic-partition of unity we decompose $q$-sum in dyadic blocks $q\sim C$ for $C\leq Q$ we have,
    \begin{equation*}
        \M(X)\ll_\eps p^\eps\sup_{C\ll Q}\left|\M(C,\,X)\right|
    \end{equation*}
	\begin{equation}\label{eq: Main}
		\M(C,\,X)=\frac{X^{3/2}}{\phi(p)}\sum_{q\sim C}\frac{\chi_\pi(q)}{q^{3/2}}\sum_{m\ll p^{1+\eps}}\sum_{n\ll p^{1/2+\eps}}\frac{\lambda_{\bar\pi}(m)}{m^{1/4}}K(m, n; pq)\mathcal I(m, n, q).
	\end{equation}
	In the next section we simplify the character sum $K(m, n, pq)$ and the integral transform $\mathcal I(m, n, q)$.

	\section{Evaluation of character sum}\label{sec: eva of char sum}
	\subsection{Evaluation of $K(m, n; pq)$}
	Recall that our character sum is given by,
	\begin{equation*}
		K(m, n, pq):=\frac{1}{\sqrt{pq}}\sumstar_{a\md{pq}}e\left(-\frac{\overline{ap}m}{q}\right)T(aq, m; p)G(-a, -n; pq)
	\end{equation*}
    \begin{lemma}\label{thm: evaluation of K(m, n, pq)}
        We have, 
        \begin{equation}
            K(m, n, pq)=\eps_2(\psi)H(-n^2/4m; p)\mathcal{G}_q(m, n)+O(1/p),
        \end{equation}
        where $\eps_2(\psi)$ is defined in the Lemma~\ref{thm:Hasse-Davenport}; $H(-n^2/4m; p)$, and $\mathcal{G}_q(m, n)$ are given by the equations~\eqref{eq:defn of H(lambda)}, and~\eqref{eq:defn of G_q(m ,n)}.
    \end{lemma}

    \begin{proof}
	We first need the twisted multiplicativity of the quadratic Gau\ss\, sum
	
	\begin{lemma*}[Twisted multiplicativity]
		Let $c=rs$ with $(r,\,s)=1$, then we have
		\begin{equation*}
			G(a, b; c)=G(as, b, r)G(ar, b; s).
		\end{equation*}
	\end{lemma*}
	Hence we have, 
	\begin{equation*}
		K(m, n, pq)=\frac{1}{\sqrt{pq}}\sumstar_{a\md{pq}}e\left(-\frac{\overline{ap}m}{q}\right)T(aq, m; p)G(-aq, -n; p)G(-ap, -n; q).
	\end{equation*}
	Since, $p\nmid q$, hence by an application of Chinese remainder theorem we can rewrite $K(m, n, pq)$ as,
	\begin{equation*}
		\begin{split}
			\frac{1}{\sqrt{p}}\sumstar_{a\md{p}}T(aq, m; p)G(-aq, -n; p)\times
			\frac{1}{\sqrt{q}}\sumstar_{a\md{q}}e\left(-\frac{\overline{ap}m}{q}\right)G(-ap, -n; q).
		\end{split}
	\end{equation*}
	By a couple of change of variables we have,
	\begin{equation*}
		\frac{1}{\sqrt{p}}\sumstar_{a\md{p}}T(a, m; p)G(-a, -n; p)\times
		\frac{1}{\sqrt{q}}\sumstar_{a\md{q}}e\left(-\frac{\overline{a}m}{q}\right)G(-a, -n; q).
	\end{equation*}
    Now the proof follows from Lemma~\ref{thm: evaluation of s_{m, n, p}}, and Proposition~\ref{prop: a mod q}.
	\end{proof}
	
	\subsection{Evaluation of $a\md{p}$-sum}
	Here we're going to evaluate the $a\md{p}$ sum,
	\begin{equation*}
	      \frac{1}{\sqrt{p}}\sumstar_{a\md{p}}T(a, m; p)G(-a, -n; p)=:\mathfrak S_{m, n, p}.
	\end{equation*}
    To evaluate this sum we need some additional lemma which we have recorded in following section,
    \subsubsection{Some preliminary lemmas}
	\begin{lemma}[Evaluation of $\eps(\pi\otimes\tilde\chi)$]
		Let $\pi$ be as usual and $\chi\notin E$. Then if $\pi_p=\pi(\tilde\mu_1, \tilde\mu_2)$, we have
		\begin{equation}\label{eq: principal series root number}
			\eps(\pi\otimes\tilde\chi)=\eps(\mu_1\chi)\eps(\mu_2\chi),
		\end{equation}
		And if $\pi_{p}$ is the twist-minimal supercuspidal representation, we have
        \begin{equation}\label{eq: supercuspidal root number}
            \eps(\pi\otimes\tilde\chi)=-\chi(-1)G((\chi\circ N)\bar{\eta}_\pi).
        \end{equation}
        Here $G((\chi\circ N)\bar{\eta}_\pi)$ is a Gau\ss\ sum over $\F_{p^2}$ given by~\eqref{eq: F_p^2 Gauss sum}.
	\end{lemma}
	\begin{proof}
		Let $\pi$ be a ramified principal series representation $\pi_{f, p}=\pi(\tilde\mu_1, \tilde\mu_2)=\tilde\mu_1\boxplus\tilde\mu_2$. Then, 
		\begin{equation*}
			\pi_{f, p}\otimes\tilde\chi=\tilde\mu_1\tilde\chi\boxplus\tilde\mu_2\tilde\chi. 
		\end{equation*}
		Using the Proposition~3.5 of \cite{MR0401654} we can rewrite the $\GL(2)$-local epsilon factor as a product of two $\GL(1)$-local epsilon factors,
		\begin{equation*}
			\eps\left(\pi_{f, p}\otimes\tilde\chi, \psi_p\right)=\eps\left(\tilde\mu_1\tilde\chi,\, \psi_p\right)\eps\left(\tilde\mu_2\tilde\chi,\, \psi_p\right).
		\end{equation*}
		Now it remains to show that for $i=1, 2$,
		\begin{equation*}
			\eps(\tilde\mu_i\tilde\chi, \psi_p)=\eps(\mu_i\chi)
		\end{equation*}
		Note that as $\chi\notin E$, so $\tilde\mu_i\tilde\chi$ is ramified character of $\Q_p^\times$ of level zero (in the sense of~\cite[Definition 1.8]{MR2234120}) and now equality follows from ~\cite[\S 23.6 and \S 23.7]{MR2234120}.\\
		
        Now we turn to the case of twist-minimal supercuspidal representation. As we have seen in the proof of Lemma~\ref{thm: computation of conductor}, $\pi_{p}$ is a depth zero representation and therefore up to $\GL_2(\Q_p)$-conjugacy we attach a cuspidal type, a triple $(\mathfrak A, ZU_\mathfrak A, \Lambda)$~\cite[\S 15.5]{MR2234120}, such that $\pi_{p}\simeq c\text{--Ind}_{ZU_\mathfrak A}^{\GL_2(\Q_p)}\Lambda$ and $\Lambda|_{U_{\mathfrak A}}$ is the inflation of an irreducible cuspidal representation, say $\xi$, of $\GL_2(\F_p)$. We call this type as cuspidal type of the first kind. Therefore if we twist this form by a level zero character i.e. $\pi_{p}\otimes\tilde\chi$ it also becomes a cuspidal type of first kind, because of the following isomorphism,
        \begin{equation*}
            c\text{--Ind}_{ZU_\mathfrak A}^{\GL_2(\Q_p)}\Lambda\otimes\tilde\chi\simeq c\text{--Ind}_{ZU_\mathfrak A}^{\GL_2(\Q_p)}(\Lambda\otimes\tilde\chi)
        \end{equation*}
        Therefore the restriction $(\Lambda\otimes\tilde\chi)|_{U_{\mathfrak A}}$ is the inflation of the following irreducible cuspidal representation $\xi\otimes\chi^{-1}$ of $\GL_2(\F_p)$. Hence we get,
        \begin{equation*}
            \eps(\pi_{p}\otimes\tilde\chi,\, \psi_p)=\eps(\xi\otimes\bar\chi,\,\psi_p)
        \end{equation*}
        So now we just need to know the finite field epsilon factor which is on the right hand side and it has been already computed in~\cite[Theorem 1]{MR3655759}. This gives, 
        \begin{equation}\label{eq: F_p^2 Gauss sum}
            \eps(\pi_{p}\otimes\tilde\chi,\, \psi_p)= -\frac{1}{p}\sum_{x\in\F_{p^2}^\times}\psi(Tr_{\F_{p^2}/\F_p}(x))\overline{\eta}_\pi(x)\chi(-Nr_{\F_{p^2}/\F_p}(x))=:-\chi(-1)G((\chi\circ Nr_{\F_{p^2}/\F_p})\bar{\eta}_\pi)
        \end{equation}
        where, $\eta_\pi$ is a non-trivial character of $\F_{p^2}^\times$ such that $\eta_\pi^{p-1}\neq 1$. This character attached to $\pi_p$ called as Green's character (see~\cite[\S3.1]{MR3655759}).\\
        This completes the proof of this lemma.
	\end{proof}
    Along with this we need the following identity quoting from~\cite[\S 5.6]{MR0955052},
    \begin{lemma}[Hasse-Davenport identity]\label{thm:Hasse-Davenport}
        Let $\F_q$ be the finite field, $N$ be a divisor of $q-1$, $\psi$ be a non-trivial additive character of $\F_q$, and $\psi_N$ be the additive character $\psi_N(x):=\psi(Nx)$. Let us denote by $\rho_1, ..., \rho_N\in \hat{\F_q^\times}$ of order dividing $N$. Then for any multiplicative character $\chi$ of $\F_q^\times$, we have
        \begin{equation}
            -\eps(\psi_N, \chi^N) = \eps_N(\psi)\prod_{j=1}^{N}\eps(\psi, \chi\rho_j),
        \end{equation}
        where
        \begin{equation*}
            \eps_N(\psi)=q^{-1/2}\prod_{j=1}^N\eps(\psi, \rho_j)^{-1}\in S^1.
        \end{equation*}
    \end{lemma}
    
     In the next section we will recall some results about hypergeometric sums.

    \subsubsection{Hypergeometric sums}\label{sec:hyp geo sum}
     Let $\F_q$ be a finite field of charcteristic prime $p$, and $(m, n)$ is given pair of non-negative interger such that $m+n\geq 1$. Let $\boldsymbol{\chi}=(\chi_i)_{1\leq i\leq m}$ and $\boldsymbol{\rho}=(\rho_j)_{1\leq j\leq n}$ of tuples of multiplicative characters over $\F_q^\times$. For any $\lambda\in\F_q^\times$, the hypergeometric sum $\hyp(\F_q^{m+n}; \boldsymbol{\chi}; \boldsymbol{\rho}; \lambda)$ is given by the following sum,
    \begin{equation*}
        \hyp(\F_q^{m+n}; \boldsymbol{\chi}; \boldsymbol{\rho}; \lambda)=\frac{(-1)^{m+n-1}}{p^{(m+n-1)/2}}\sum_{N(x)=\lambda N(y)}\prod_{i=1}^{m}\chi_i(x_i)\prod_{j=1}^n\bar\rho_j(y_j)e\left(\frac{T(x)-T(y)}{p}\right),
    \end{equation*}
    where, 
    \begin{equation*}
	N(x)=x_1\cdots x_m,\hspace{0.25cm} N(y)=y_1\cdots y_n,
    \end{equation*}
    \begin{equation*}
	T(x)=x_1+\cdots+x_m,\hspace{0.25cm}\text{and}\hspace{0.25cm} T(y)=y_1+\cdots+y_n.
    \end{equation*}
    In our analysis we have encountered a new exponential sum which in the spirit of~\cite[Theorem 8.8.5]{MR0955052} can be considered of as a generalization of hypergeometric sum, we will call this sum as \textit{exotic hypergeomtric sum}.\\
    To define this we recall the finite \'etale algebra $B_\nu(\F_q)$ over $\F_q$ of rank $\nu\geq 1$ can be written in the following form, 
    \begin{equation*}
	B_\nu(\F_q)=\F_{q^{\nu_1}}\times\cdots\times\F_{q^{\nu_k}},\hspace{0.25cm}\nu_1+\cdots+\nu_k=\nu.
    \end{equation*}
    Let $x=(x_1,\dots, x_\nu)\in B(\F_q)^\times$. Let $\chi:B_\nu(\F_q)^\times\to\C$ be a multiplicative characer, 
    \begin{equation*}
	\chi(x)=\prod_{i=1}^\nu\chi_i(x_i),\hspace{0.25cm}\chi_i\in\widehat{\F_{q^{\nu_i}}^\times},
    \end{equation*}and for a non-trivial additive character $\psi:\F_q\to\C^\times$ we define an additive character $\psi:B(\F_q)\to\C^\times$ by, 
    \begin{equation*}
	\psi(T_\nu(x))=\sum_{1\leq i\leq \nu}\psi(T_{\F_{q^{\nu_i}}/\F_q}(x_i)).
    \end{equation*}
    The exotic hypergeometric sum, $\hyp(B_m(\F_q)\times B_n(\F_q); \chi; \rho; \psi)$ is given by, 
    \begin{equation}\label{eq:exotic hypergeo}
	\hyp(B_m(\F_q)\times B_n(\F_q); \chi; \rho; \psi; \lambda)=\frac{(-1)^{m+n-1}}{p^{(m+n-1)/2}}\sum_{N_m(x)=\lambda N_n(y)}\chi(x)\bar\rho(y)\psi(T_m(x)-T_n(y)),
    \end{equation}
    where the homomorphism,\
    \begin{equation*}
	N_\nu:B_\nu(\F_q)^\times\to\F_q^\times,\hspace{0.25cm} N(x)=\prod_{i=1}^\nu N_{\F_{q^{\nu_i}}/\F_q}(x_i). 
    \end{equation*}

    Now we record some geometric properties of the hypergeometric sums. For this we will use facts from the theory of trace functions, a good reference for this would be~\cite{FKM15}. 
     
    \begin{lemma}[Geometric properties]\label{thm:geo prop}
     Let $\ell$ be a prime different from $p$. If the pair of characters $\chi$ and $\rho$ remains disjoint over any finite extension of $\F_q$, then there exists a geometrically irreducible $\ell$-adic middle-extension sheaf $\mathcal H(\chi; \rho)$ of pure of weight $0$ with the trace function given by~\eqref{eq:exotic hypergeo}. It has rank $\max\{m, n\}$. This sheaf is lisse on $\mathbb G_m$, except if $m=n$, in which case it is lisse on $\mathbb G_m-\{1\}$.
    \end{lemma}
    \begin{proof}[Proof (Sketch)]
        Let $D_c^b(\mathbb G_{m, \F_q},\, \bar{\Q}_\ell)$ be the triangulated category of complexes of $\bar{\Q}_\ell$-sheaves defined in~\cite[1.1.3]{MR1081536}. We define the \textit{exotic hypergeometric complex} to be the object, 
	   \begin{equation*}
		\hyp(B_m(\F_q)\times B_n(\F_q); \chi; \rho; \psi):=\text{Kl}(B_m(\F_q); \chi; \psi)[1]\star_!\text{inv}^\star\text{Kl}(B_n(\F_q); \rho; \psi)[1],
	   \end{equation*}
	   in $D_b^c(\mathbb{G}_{m,\, \F_q}, \bar{\Q}_\ell)$. Here $\text{Kl}(B(\F_q); \chi; \psi)$ is the exotic Kloosterman complex defined in~\cite[8.8.6. Definition]{MR0955052}. By the Grothendieck-Lefschetz trace formula we have, 
	   \begin{equation*}
		-\text{tr}(\text{Frob}_{\lambda, p}|\hyp(B_m(\F_q)\times B_n(\F_q); \chi; \rho; \psi)_{\bar \lambda})=\hyp(B_m(\F_q)\times B_n(\F_q); \chi; \rho; \psi; \lambda).
	   \end{equation*}
	   To deduce the geometric properties we follow the proof of~\cite[8.8.5. Theorem]{MR0955052}. These properties are invariant under finite extension of the ground field $\F_q$ over which $B_\nu(\F_q)$ becomes isomorphic to the $\nu$-fold product $\F_q\times\cdots\times\F_q$. In terms of the corresponding coordinates $x_1,\dots, x_\nu$ we have, 
	   \begin{equation*}
		\psi(x)=\prod\psi(x_i),\hspace{0.25cm}\chi(x)=\prod\chi_i(x_i),\hspace{0.25cm}\text{and}\hspace{0.25cm}\rho(y)=\prod\rho(y_j).
	   \end{equation*}
	   Therefore we have, 
	   \begin{equation*}
		\begin{split}
			\hyp(B_m(\F_q)\times B_n(\F_q); \chi; \rho; \psi)&\simeq\text{Kl}(\F_q^m; \chi; \rho; \psi)[1]\star_!\text{inv}^\star\text{Kl}(\F_q^n; \rho; \psi)[1]
		\end{split} 
	   \end{equation*}
	   where the last complex is isomorphic to classical hypergeometric complexes\\ $\hyp(!; \prod\chi_i; \prod\rho_j; \psi)$ defined by~\cite[Remark 8.4.3]{MR1081536} and disjointness property gurantess the existence of a geometrically irreducible sheave which is isomorphic to classical hypergeometric sheave under base-change. The geometric properties follows immediately from that of the hypergeometric sheaves (see~\cite[Theorem 8.4.2; Theorem 8.4.13]{MR1081536}).
    \end{proof}
    
    \subsubsection{Final evaluation}
    
    \begin{lemma}[Evaluation of $\mathfrak S_{m, n, p}$]\label{thm: evaluation of s_{m, n, p}}
        We have, 
        \begin{equation}
            \mathfrak S_{m, n, p}=\eps_2(\psi)H(-n^2/4m; p)+ O(1/\sqrt p),
        \end{equation}
        where, $\eps_2(\psi)$ is of absolute value one, defined in the Lemma~\ref{thm:Hasse-Davenport}, and for $\lambda\neq 0$ the character sum $H(\lambda, p)$ is defined by, 
        \begin{equation}\label{eq:defn of H(lambda)}
            H(\lambda; p)=
            \begin{cases}
                \hyp(\F_p^2\times\F_p^2; 1, \rho; \bar{\mu}, \bar{\mu}; \lambda) &\text{if $\pi_p$ is of type (a)},\\
                \eta_f(-1)\hyp(\F_p^2\times\F_{p^2}; 1, \rho; \eta_\pi; \lambda) &\text{if $\pi_p$ is of type (b)}.
            \end{cases}
        \end{equation}
    \end{lemma}
    \begin{note*}
        Sometimes to make our mathematical expressions clear we denote $H(-n^2/4m; p)$ by $C(4m, n^2; p)$. Mainly we have used this convention in Section\ref{sec: generic}.
    \end{note*}
    \begin{proof}

	To evaluate it we first open each summand $T(a, m; p)$ and $G(-a, -n; p)$, and after rearranging the sums we have, 

	\begin{equation*}
		\mathfrak S_{m, n, p}=\frac{1}{p^{\frac{3}{2}}}\underset{\chi\notin E}{\sumstar_{\chi\md{p}}}\chi(m)\eps(\chi)\eps(\pi\times \bar{\chi})\sum_{x\md{p}}\left(-\frac{nx}{p}\right)\sumstar_{a\md{p}}\bar{\chi}(a)e\left(-\frac{ax^2}{p}\right)
	\end{equation*}
	  After a change of variable the last sum evaluates to $\chi(-x^2)\eps(\bar{\chi})\sqrt{p}$ and using the identity $\eps(\bar\chi)=\chi(-1)\overline{\eps(\chi)}$ it becomes $\chi^2(x)\overline{\eps(\chi)}\sqrt{p}$. Plugging this $\mathfrak S_{m, n, p}$ becomes, 
	 \begin{equation*}
		\begin{split}
			\mathfrak S_{m, n, p}=\frac{1}{p}\underset{\chi\notin E}{\sumstar_{\chi\md{p}}}\chi(m)\eps(\chi)\eps(\pi\times\bar\chi)\sum_{x\md{p}}\chi^2(x)e\left(-\frac{nx}{p}\right)
		\end{split}
	 \end{equation*}
	 Now by an application of change of variable we get, 
	 \begin{equation}\label{eq: root number wt sum}
		\mathfrak S_{m, n, p}=\frac{1}{\sqrt{p}}\underset{\chi\notin E}{\sumstar_{\chi\md{p}}}\chi(m)\bar{\chi}(n^2)\eps(\chi^2)\eps(\pi\times \bar\chi).
	 \end{equation}
	 
	Next step would be to execute the $\chi\md{p}$-sum, for this first we need to explicate the root number of the twisted form $\eps(\pi\times\bar\chi)$ which we have recorded in the following lemma.

    Now plugging the evaluation~\eqref{eq: principal series root number} in ~\eqref{eq: root number wt sum}, we get,
    \begin{equation*}
        \mathfrak S_{m, n, p}=
            \frac{1}{\sqrt p}\underset{\chi\notin E}{\sumstar_{\chi\md{p}}}\chi(m)\bar{\chi}(n^2)\eps(\chi^2)\eps(\mu_1\bar\chi)\eps(\mu_2\bar\chi) 
    \end{equation*}
    Both of their treatments are similar then we just record the proof principal series case. Let $\rho$ be the quadratic character modulo $p$, and using the Hasse-Davenport identity in the principal series case we rewrite $\mathfrak S_{m, n, p}$  as, 
    \begin{equation*}
        \begin{split}
            \mathfrak S_{m, n, p}&=\frac{\eps_2(\psi)}{\sqrt{p}}\underset{\chi\notin E}{\sumstar_{\chi\md{p}}}\chi(4m)\bar{\chi}(n^2)\eps(\chi)\eps(\chi\rho)\eps(\mu_1\bar\chi)\eps(\mu_2\bar\chi)\\&=\frac{-\eps_2(\psi)}{\sqrt{p}}{\sum_{\chi\md{p}}}\chi(4m)\bar{\chi}(n^2)\eps(\chi)\eps(\chi\rho)\eps(\mu_1\bar\chi)\eps(\mu_2\bar\chi)+O(1/\sqrt{p}).
        \end{split}
    \end{equation*}
    Opening each of these four normalised Gau\ss\ sums and executing the $\chi\md{p}$-sum we have, 
    \begin{equation*}
        \mathfrak S_{m, n, p}=\frac{-\eps_2(\psi)}{p^{3/2}}\underset{\frac{xy}{zw}\equiv -\frac{n^2}{4m}\md{p}}{\sum_{x, y, z, w\in \F_p^\times}}\rho(y)\mu_1(z)\mu_2(w)e\left(\frac{x+y-z-w}{p}\right)+ O(1/\sqrt p).
    \end{equation*}
    And if $\pi_{p}$ is twist-minimal supercuspdial then similarly opening the Gau\ss\ sum and executing the sum over $\chi\md{p}$ we get, 
    \begin{equation*}
        \begin{split}
            \mathfrak S_{m, n, p}
        &=-\frac{1}{\sqrt p}\underset{\chi\notin E}{\sumstar_{\chi\md{p}}}\chi(-m)\bar{\chi}(n^2)\eps(\chi^2)G((\bar{\chi}\circ N)\bar{\eta}_\pi)\\
        &=\frac{-\eps_2(\psi)}{p^{3/2}}\underset{\frac{xy}{N(z)}\equiv-\frac{n^2}{4m}\md{p}}{\sum_{x, y\in\F_p^\times}\sum_{z\in \F_{p^2}^\times}}\rho(y)\bar{\eta}_f(-z)e\left(\frac{x+y-Tr_{\mathbb{F}_{p^2}/\mathbb{F}_p}(z)}{p}\right)+O(1/\sqrt p).
        \end{split}
    \end{equation*}
    This concludes the proof of the Lemma~\ref{thm: evaluation of K(m, n, pq)}.
    \end{proof}

    \[
    H(\lambda):=\frac{-1}{p^{3/2}}\underset{\text{Nr}_{\mathbb{F}_{p^2}/\mathbb{F_p}}(x)=\lambda yz}{\sum_{x\in \mathbb{F}_{p^2}^\times}\sum_{y, z\in \mathbb{F}_p^\times}}\chi(x)\rho(y)e\left(\frac{\text{Tr}_{\mathbb{F}_{p^2}/\mathbb{F}_p}(x)-y-z}{p}\right)
    \]

    \begin{corollary}[Square-root cancellation]\label{cor: sq cancellation}
	We have $$H(\lambda; p)=O(1),$$ which in particular implies $\mathfrak S_{m, n, p}\ll 1$. Here the implied constants are absolute.
    \end{corollary}

    \begin{proof}
	$H(\lambda; p)$ is given by classical hypergeometric sum in special and principal series case, and the square-root cancellation is a standard fact. Now in the supercuspidal case, it is given by exotic hypergeometric sum defined in~\eqref{eq:exotic hypergeo}. Lemma~\ref{thm:geo prop} ensures the existence of an $\ell$-adic sheaf and after the quadratic field extension of $\F_p$ this is isomorphic to hypergeometric sheaf $\mathcal{H}(!; 1\circ N_{\F_{p^2}/\F_p}, \rho\circ N_{\F_{p^2}/\F_p}; \eta_\pi, \eta_\pi^p)$.  Regularity of $\eta_\pi$ implies that, pair of characters are disjoint which in turn implies that sheaf is geometrically irreducible and of pure of weight-zero. Hence, we are done.
    \end{proof}

	\subsection{Evaluation of $a\md{q}$-sum}
	In this section we evaluate the following sum,
	\begin{equation}\label{eq:a mod q}
		\frac{1}{\sqrt{q}}\sumstar_{a\md{q}}e\left(\frac{\overline{a}m}{q}\right)G(a, -n; q)=:\mathcal{G}_q(m ,n)
	\end{equation}

	\begin{proposition}[Evaluation of $a\md{q}$-sum]\label{prop: a mod q}
		Let $q=2^kr=2^kcq_1^2q_2^2$ where, $2^k||q$, $\mu^2(c)=1$, $q_1|c^\infty$, and $(c, q_2)=1$. The evaluation of $\mathcal{G}_q(m, n)$ is given by the following,
		\begin{equation}\label{eq:defn of G_q(m ,n)}
			\begin{cases}
				\Delta_q(4m-n^2), &\text{$k=0$}.\\
				\sqrt{2}e\left(\frac{m}{2}\right)1_{n\equiv 1\md{2}}\Delta_q(4m-n^2), &\text{$k=1$}.\\
				\frac{1}{2}\Kl_2\left(\frac{4m-n^2}{2^{k}}; 2^2\right)1_{n\equiv 0\md{2}}1_{4m\equiv  n^2\md{2^k}}\Delta_q(4m-n^2)&\\
				\hspace{4cm}+\frac{1}{2^{k}}c\left(\frac{4m-n^2}{2^k}; 2^k\right)\Delta_q(4m-n^2), &\text{$k>1$; $2|k$}.\\
				\frac{1}{2}\Kl_2\left(\frac{4m-n^2}{2^{k-1}}; 2^3\right)1_{n\equiv0\md{2}}1_{4m\equiv n^2\md{2^{k-1}}}\Delta_q(4m-n^2),&\text{$k>1$; $2\nmid k$}.
			\end{cases}
		\end{equation}
		where, 
		\begin{equation}
			\Delta_q(4m-n^2)=q_1q_2^{-1}\eps_c^2\left(\frac{(4m-n^2)/q_1^2}{c}\right)1_{4m\equiv n^2\md{q_1^2}}c(4m-n^2; q_2^2).
		\end{equation}
	\end{proposition}
	
	\begin{proof}
		Firstly for the evaluation we need the following lemma concerning the evaluation of the quadratic Gau\ss\, sum.
		\begin{lemma*}[Evaluation of the quadratic Gau\ss\, sum] If $q$ is odd,
			\begin{equation}\label{eq:c odd}
				G(a, b; q)= \eps_c \left(\frac{a}{q}\right)e\left(-\frac{\overline{4a}b^2}{q}\right) 
			\end{equation}
			where 
			\begin{equation*}
				\eps_c = \begin{cases}
					1, &\text{$c\equiv 1\md{4}$}.\\
					i, &\text{$c\equiv 3\md{4}$}.
				\end{cases}
			\end{equation*}
			If $q=2^k$ then,
			\begin{equation}\label{eq:c even}
				G(a, b; q)=
				\begin{cases}
					\sqrt{2}\,1_{b\equiv 1\md{2}}, &\text{$k=1$.}\\
					\frac{1}{2^{k/2}}1_{b\equiv 0\md{2}}e\left(\frac{\bar{a}b^2/4}{q}\right)\left(1+e\left(\frac{a}{2^2}\right)\right), &\text{$k>1$ and $k\equiv 0\md{2}$}.\\
					\frac{1}{2^{(k-1)/2}}1_{b\equiv 0\md{2}}e\left(\frac{\bar{a}b^2/4}{q}\right)e\left(\frac{a}{2^3}\right), &\text{$k>1$ and $k\equiv 1\md{2}$.}
				\end{cases}
			\end{equation}
		\end{lemma*}
		It can easily be derived from~\cite[Lemma 5.4.5]{MR1420620}, \cite[Theorem 1.3.4]{MR1625181}, and~\cite[Lemma 5.1]{MR4624955}.\\
		
		Now we turn to proof of the proposition. We give a proof just for the first and last case. Remaining cases will follow in the similar fashion. First we treat the case where $k=0$. By the equation~\eqref{eq:c odd}, and a small change of variable~\eqref{eq:a mod q} becomes,
		\begin{equation*}
			\frac{\eps_q}{\sqrt{q}}\sumstar_{a\md{cq_1^2q_2^2}}\left(\frac{a}{c}\right)e\left(\frac{a(4m-n^2)}{cq_1^2q_2^2}\right)
		\end{equation*}
		and by the Chinese remainder theorem we rewrite this as, 
		\begin{equation*}
			\begin{split}
				\frac{\eps_q}{\sqrt{q}}&\sum_{a\md{cq_1^2}}\left(\frac{a}{c}\right)e\left(\frac{a(4m-n^2)}{cq_1^2}\right)\sumstar_{a\md{q_2^2}}e\left(\frac{a(4m-n^2)}{q_2^2}\right)\\&
				=c_{q_2^2}(4m-n^2)\frac{\eps_q}{\sqrt{q}}\sum_{a\md{cq_1^2}}\left(\frac{a}{c}\right)e\left(\frac{a(4m-n^2)}{cq_1^2}\right)
			\end{split}
		\end{equation*}
		To evaluate the $a\md{cq_1^2}$-sum we write $a=b+rc$ where $b\in \Z/c\Z$ and $r\in \Z/q_1^2\Z$, so the inner sums becomes,
		\begin{equation*}
			\begin{split}
				&\sum_{b\md{c}}\left(\frac{b}{c}\right)e\left(\frac{b(4m-n^2)}{cq_1^2}\right)\sum_{r\md{q_1^2}}e\left(\frac{r(4m-n^2)}{q_1^2}\right)\\&
				=q_1^2\sum_{b\md{c}}\left(\frac{b}{c}\right)e\left(\frac{b(4m-n^2)/q_1^2}{c}\right)1_{4m-n^2\equiv 0\md{q_1^2}}\\&
				=\eps_cq_1^2\sqrt{c}\left(\frac{(4m-n^2)/q_1^2}{c}\right)1_{4m-n^2\equiv 0 \md{q_1^2}}.
			\end{split}
		\end{equation*}
		Hence combining these, the evaluation in the odd case follows.
		
		Now we turn to the case $k\equiv 1\md{2}$. By Chinese remainder theorem, and couple of change of variable we can rewrite the equation~\eqref{eq:a mod q} as,
		\begin{equation*}
			\frac{1}{\sqrt{r}}\sumstar_{a\md{r}}e\left(\frac{\bar{a}m}{r}\right)G(a, -n; r)\times\frac{1}{\sqrt{2^k}}\sumstar_{a\md{2^k}}e\left(\frac{\bar{a}m}{2^k}\right)G(a, -n; 2^k)
		\end{equation*}
		Here the $a\md{r}$-sum can be evaluated as before, so we just need to evaluate the remaining one. 
		\begin{equation*}
			\begin{split}
				\frac{1}{\sqrt{2^k}}\sumstar_{a\md{2^k}}\cdots=\frac{1}{2^{k-1/2}}1_{n\equiv 0\md{2}}\sumstar_{a\md{2^k}}e\left(\frac{a(m-\frac{n^2}{4})}{2^k}\right)e\left(\frac{\bar{a}}{2^3}\right)
			\end{split}
		\end{equation*}
		Write $a=2^3b+c$, with $b\in\Z/2^{k-3}\Z$, and $c\in\left(\Z/2^3\Z\right)^\times$. Hence the second sum becomes, 
		\begin{equation*}
			\begin{split}
				\sumstar_{a\md{2^k}}\cdots&=\sumstar_{b\md{2^{k-3}}}e\left(\frac{b(m-\frac{n^2}{4})}{2^{k-3}}\right)\sumstar_{c\md{2^3}}e_{2^3}\left(c\left(\frac{4m-n^2}{2^{k-1}}+\bar{c}\right)\right)\\&
				=2^{k-\frac{3}{2}}1_{4m\equiv n^2\md{2^{k-1}}}\Kl_2\left(\frac{4m-n^2}{2^{k-1}}; 2^3\right)
			\end{split}
		\end{equation*}
		Hence the result follows.
	\end{proof}

	\section{Reduction of integral transform}
	In this section we are goint to simplify the integral transform $\mathcal{I}(m, n, q)$ that has appeared in our equation~\eqref{eq: Main}. Recall,
	\begin{equation*}
		\begin{split}
			\mathcal{I}(m, n, q)=\int_\R& W\left(\frac{x}{\mathfrak X}\right)g(q, x)\int_{0}^{\infty}U(y)U_{m, q}(y)\int_\R V(z)\\
			&\times e\left(\frac{xX^2}{pqQ}(y-z^2)\pm\frac{2X\sqrt{my}}{pq}+\frac{nzX^2}{pq}\right)\, dx\, dy\, dz.
		\end{split}
	\end{equation*}
	We record the simplification as a form of a lemma,
	\begin{lemma}
		Let $F$ be a smooth function supported on $[0, 2A]$ for some large $A$ and $F\equiv 1$ on $[0, A]$, then we have
		\begin{equation}\label{eq: simp Integral}
			\mathcal I(m, n, q)= \sqrt{|\mathfrak X|}\frac{q^2}{C^2}\frac{C}{Q}\frac{p^{1/4}}{m^{1/4}} U\left(\frac{m}{M}\right)V\left(\frac{|n|}{N}\right)F\left(\frac{|4m-n^2|}{p^{1+\eps}C/Q}\right) \tilde I_q (m, n) + O(p^{-2026})
		\end{equation}
		where, 
		\begin{equation*}
			\tilde I_q(m, n):= \int_\R \frac{W(x)}{\sqrt{x}}g\left(q, \frac{xC\mathfrak X}{q}\right)e\left(\pm \frac{4m-n^2}{4xpC|\mathfrak X|/Q}\right)\, dx.
		\end{equation*}
	\end{lemma}
	
	\begin{proof}
            Our first step in this analysis is the following change of variables which would make the phase function free from the variable $q$,
		\begin{equation*}
			x\leadsto C\mathfrak X x/q,\,\,\, y\leadsto q^2y/C^2,\,\,\,\text{and}\,\,\, z\leadsto qz/C.
		\end{equation*} 
		It transforms our integral transform $\mathcal{I}(m, n; q)$ into 
		\begin{equation*}
			\begin{split}
				 \frac{q^2\mathfrak X}{C^2}\int_\R W\left(\frac{Cx}{q}\right)&g\left(q, \frac{Cx\mathfrak X}{q}\right)\int_{0}^{\infty}U\left(\frac{q^2y}{C^2}\right)U_{m, C}\left(y\right)\int_\R V\left(\frac{qx}{C}\right)\\
				&\times e\left(\frac{x\mathfrak X Q}{C}(y-z^2)\pm\frac{2X\sqrt{my}}{pC}+\frac{nzX}{pC}\right)\, dx\, dy\, dz.
			\end{split}    
		\end{equation*}
         Note that now the variable $q$ has appeared in our smooth weight functions $U, V$ and $W$, but we can easily separate $q$ from them at a cost of a negligible error, just by a simple application of Mellin inversion . To keep our expression clean, we drop the $s$-integrals from our analysis,
		\begin{equation}\label{eq: int transform}
			\begin{split}
				\frac{q^2\mathfrak X}{C^2}\int_\R W\left({x}\right)g\left(q, \frac{Cx\mathfrak X}{q}\right)&\int_{0}^{\infty}U(y)e\left(\frac{xy\mathfrak X Q}{C}\pm \frac{2X\sqrt{my}}{pC}\right)\, dy\\
				&\times\int_\R V(z)e\left(\frac{nzX}{pC}-\frac{xz^2\mathfrak X Q}{C}\right)\, dz\, dx,
			\end{split}
		\end{equation}
		with new weight functions $W,\, U$, and $V$. For our further analysis we need the following result on the stationary phase,
		
		\begin{lemma}(Method of stationary phase)
			Let $w(x)$ be a nice function supported on $x\asymp X$, and suppose for all $j\geq 1$ the phase function $\phi(x)$ satisfies, 
			\begin{equation*}
				X^j \partial_x^j \phi(x)\ll Y.
			\end{equation*}
			Now suppose for some $x_0\asymp X$, known as stationary point, we have $\phi^\prime(x_0)=0$ and $X^2\partial_x^2\phi(x)\gg Y$. If $Y\gg R\geq 1$ for some $R$, then for any large $A>0$ we have
			\begin{equation}\label{eq: stationary phase}
				\int_\R w(x)e\left(\phi(x)\right)\, dx = \frac{X}{\sqrt{Y}}e\left(\phi(x_0)+\frac{1}{8}\right)w(x_0)+O_A\left(\frac{X}{R^A}\right).
			\end{equation}
		\end{lemma}

		\begin{proof}
			See~\cite[Proposition 8.2]{MR3127809} and~\cite[Lemma 4.3]{MR4624955}.
		\end{proof}

		Now we apply this lemma on the $y$-integral first. Here phase function admits a stationary point iff $m\asymp p\mathfrak X^2 =:M$ otherwise it is negligibly small by repeated integration by parts and up to an negligible error the $y$-integral equals

		\begin{equation*}
			\int_{0}^{\infty}U(y)e\left(\frac{xy\mathfrak X Q}{C}\pm \frac{2X\sqrt{my}}{pC}\right)\, dy=\frac{p^{1/4}}{m^{1/4}}\sqrt{\frac{C}{Q}}e\left(\mp \frac{mQ}{xpC|\mathfrak X|}\right)U\left(\frac{m}{M}\right)+O(p^{-2026}).
		\end{equation*}

		Let $N:=\sqrt{p}|\mathfrak X|$, then similarly for the $z$ integral we have,

		\begin{equation*}
			\int_\R V(z)e\left(\frac{nzX}{pC}-\frac{xz^2\mathfrak X Q}{C}\right)\, dz=\sqrt{\frac{C}{xQ|\mathfrak X|}}e\left(\pm \frac{n^2Q}{4pC|\mathfrak X|}\right)V\left(\frac{|n|}{N}\right)+O(p^{-2026}).
		\end{equation*}

		Now plugging them in~\eqref{eq: int transform} we have,

		\begin{equation*}
			\begin{split}
				\sqrt{|\mathfrak X|}\frac{q^2}{C^2}\frac{C}{Q}\frac{p^{1/4}}{m^{1/4}} U\left(\frac{m}{M}\right)V\left(\frac{|n|}{N}\right)\int_\R \frac{W(x)}{\sqrt{x}}g\left(q, \frac{xC\mathfrak X}{q}\right)&e\left(\pm \frac{(4m-n^2)}{4xpC|\mathfrak X|/Q}\right)\, dx\\& \hspace{1.4cm}+O(p^{-2026}).
			\end{split}
		\end{equation*}
		This would be enough if $C$ is large, $C\gg Q^{1-\eps}$. But we need to improve this when $C\ll Q^{1-\eps}$. In this case if $|\mathfrak X|\ll p^{-\eps}$ then we could replace $g(q, xC\mathfrak X/Q)$ by $1$ and otherwise, i.e. $p^{-\eps}\ll |\mathfrak X|\ll p^\eps$, it behaves like a smooth function, therefore we can take it with our weight function. So, in any case the $x$-integral becomes,

        \begin{equation*}
            \int_{x\asymp 1} e\left(\pm \frac{(4m-n^2)}{4xpC|\mathfrak X|/Q}\right)\, dx
        \end{equation*}
        And the integration by parts bound yields that it is negligibly small unless we have $|4m-n^2|\ll p^{1+\eps}C/Q$. Observe that this inequality is still valid when $C\gg Q^{1-\eps}$ (note that our $\eps$ are different thanks to our epsilon convention). 
        Let $F$ be a nice function supported on $[0, 2A]$ and $F\equiv 1$ on $[0, A]$ for some large $A$. Then we can replace the above restriction in terms of $F$, so we can write the $x$-integral as 
        \begin{equation*}
            F\left(\frac{|4m-n^2|}{p^{1+\eps}C/Q}\right)\int_\R \frac{W(x)}{\sqrt{x}}g\left(q, \frac{xC\mathfrak X}{q}\right)e\left(\pm \frac{(4m-n^2)}{4xpC|\mathfrak X|/Q}\right)\, dx +O(p^{-2026}).
        \end{equation*}

        Now plugging this expression of $x$-integral in the previous expression of integral transform we have, 

        \begin{equation*}
            \begin{split}
                \sqrt{|\mathfrak X|}\frac{q^2}{C^2}\frac{C}{Q}\frac{p^{1/4}}{m^{1/4}} U\left(\frac{m}{M}\right)V\left(\frac{|n|}{N}\right)&F\left(\frac{|4m-n^2|}{p^{1+\eps}C/Q}\right)\int_\R \frac{W(x)}{\sqrt{x}}g\left(q, \frac{xC\mathfrak X}{q}\right)\\&\hspace{1.4cm}\times e\left(\pm \frac{(4m-n^2)}{4xpC|\mathfrak X|/Q}\right)\, dx+O(p^{-2026}).
            \end{split}
        \end{equation*}
	   This completes the proof.
	\end{proof}

	\section{Simplifications}
	
	In this section we simplify our main equation $\M(C, X)$ given by the equation~\eqref{eq: Main}. In order to simplify the expression we plug~\eqref{eq: simp Integral} into~\eqref{eq: Main}. Up to a negligible error we have, 

    \begin{equation*}
        \begin{split}
            \M(C, X)=\sqrt{|\mathfrak X|}\frac{pC\sqrt{Q}}{\phi(p)}\sum_{q\sim C}\frac{\chi_\pi(q)}{q^{3/2}}\sum_{m\ll p^{1+\eps}}\sum_{n\ll p^{1/2+\eps}}\frac{\lambda_{\bar\pi}(m)}{\sqrt{m}}K(m, n; pq)\tilde{I}_q(m , n)&\\&\hspace{-5.5cm}\times U\left(\frac{m}{M}\right)V\left(\frac{|n|}{N}\right)H\left(\frac{|4m-n^2|}{p^{1+\eps}C/Q}\right)
        \end{split}
    \end{equation*}
	
	Opening the integral $\tilde I_q(m , n)$ we get,

    \begin{equation*}
        \begin{split}
            \M(C, X)=\sqrt{|\mathfrak X|}\frac{p\sqrt{Q}}{\phi(p)\sqrt C}\int_\R\frac{W(x)}{\sqrt x}\sum_{q\sim C}\chi_{\pi}(q)\,g\left(q, \frac{xC|\mathfrak X|}{q}\right)&\\&\hspace{-5cm}\times\sum_{m\sim M}\sum_{n\sim N}\frac{\lambda_{\bar\pi}(m)}{\sqrt{m}}K(m, n; pq)F\left(\frac{4m-n^2}{p^{1+\eps}C/Q}\right)\, dx
        \end{split}
    \end{equation*}
    Note that we have suppressed the weight functions of $m$, and $n$ under the asymptotic notation; also note that $F\left(\frac{4m-n^2}{p^{1+\eps}C/Q}\right)$ is a new weight function which is the product $e\left(\pm\frac{4m-n^2}{p^C|\mathfrak X|/Q}\right)H\left(\frac{|4m-n^2|}{p^{1+\eps}C/Q}\right)$. We can consider the function $g(q, xC|\mathfrak X|/q)$ as a weight function of $q$, so it can also be suppressed under the notation $q\sim C$. Finally by triangle inequality we have, 

    \begin{equation}
        \begin{split}
            \M(C, X)\ll_\eps\int_{x\asymp 1}\left|\sqrt{\frac{Q}{C}}\sum_{q\sim C}\chi_\pi(q)\sum_{m\sim M}\sum_{n\sim N}\frac{\lambda_{\bar\pi}(m)}{\sqrt{m}}K(m, n; pq)F\left(\frac{|4m-n^2|}{p^{1+\eps}C/Q}\right)\right|.
        \end{split}
    \end{equation}
	
	Therefore it suffices to bound, 
	\begin{equation}
		\M_0(C,\,X):=\sqrt{\frac{Q}{C}}\sum_{q\sim C}\chi_\pi(q)\sum_{m\sim M}\sum_{n\sim N}\frac{\lambda_{\bar\pi}(m)}{\sqrt{m}}K(m, n; pq)F\left(\frac{|4m-n^2|}{p^{1+\eps}C/Q}\right).
	\end{equation}
	Evaluation of the character sum $K(m, n, pq)$ depends on the exponent of $2$ appears in the unique prime factorization of $q$ and therefore depending on the exponent, we now subdivide $\M_u(C,\,X)$ in four sub-sums. Precisely, 
	\begin{equation*}
		\M_0(C, X)=\M^{\text{odd}}(C,X)+\M^{\text{even}}(C,X),
	\end{equation*}
	where,
	\begin{equation}
		\M^{\text{odd}}(C, X):=\sqrt{\frac{Q}{C}}\underset{(q, 2)=1}{\sum_{q\sim C}}\chi_\pi(q)\sum_{m\sim M}\sum_{n\sim N}\frac{\lambda_{\bar\pi}(m)}{\sqrt{m}}K(m, n; pq)F\left(\frac{|4m-n^2|}{p^{1+\eps}C/Q}\right),
	\end{equation}
	and
	\begin{equation}
		\M^{\text{even}}(C, X):=\sqrt{\frac{Q}{C}}\left(\sum_{2||q}\cdots+\underset{k\equiv 0\md{2}}{\sum_{4|q}}\cdots+\underset{k\equiv 1\md{2}}{\sum_{4|q}}\cdots\right).
	\end{equation}
	Treatment of the each sub-sums of $\M^{\text{even}}(C,\,X)$ will be similar to the treatment of $\M^{\text{odd}}(C,\,X)$. For a clear exposition we only record the treatment for the odd case. In this case plugging the evaluation of character sum we get, 
	\begin{equation*}
		\begin{split}
			\M^{\text{odd}}(C, X)=\sqrt{\frac{Q}{C}}\sum_{cq_1^2q_2^2\sim C}\eps_c^2\chi_\pi(cq_1^2q_2^2)q_1q_2^{-1}\underset{4m-n^2\equiv 0\md{q_1^2}}{\sum_{m\sim M}\sum_{n\sim N}}\frac{\lambda_{\bar\pi}(m)}{\sqrt{m}}&\\&\hspace{-9cm}\times \left(\frac{(4m-n^2)/q_1^2}{c}\right)H(-n^2/4m; p)c_{q_2^2}(4m-n^2)F\left(\frac{|4m-n^2|}{p^{1+\eps}C/Q}\right)
		\end{split}
	\end{equation*}
	Expanding the Ramanujan sum $c_{q_2^2}(4m-n^2)$ as following
	$$
	c_{q_2^2}(4m-n^2)=\sum_{d|(4m-n^2, q_2^2)}d\mu(q_2^2/d)=\underset{q_2^2=dl}{\sum_{4m-n^2\equiv 0\md{d}}}d\mu(l),
	$$
	we have
	\begin{equation*}
		\begin{split}
			\M^\text{odd}(C, X)=\sqrt{\frac{Q}{C}}\underset{dl=\square}{\sum_{cq_1^2dl\sim C}}\eps_c^2\chi_\pi(cq_1^2dl)\mu(l)\sqrt{\frac{q_1^2d}{l}}\underset{4m-n^2\equiv 0\md{q_1^2d}}{\sum_{m\sim M}\sum_{n\sim N}}\frac{\lambda_{\bar\pi}(m)}{\sqrt{m}}&\\&\hspace{-7cm}\times \left(\frac{(4m-n^2)/q_1^2}{c}\right)H(-n^2/4m; p)F\left(\frac{|4m-n^2|}{p^{1+\eps}C/Q}\right).
		\end{split}
	\end{equation*}
	Now because of the congruence condition $4m-n^2\equiv 0\md{q_1^2d}$ we determine the fraction $\frac{4m-n^2}{q_1^2d}$ by $r\in\Z$, and define a sequence,
	\begin{equation}\label{eq:alpha}
		\alpha_\pi(v):=\underset{4m-n^2=v}{\sum_{m\sim M}\sum_{n\sim N}}\frac{\lambda_{\bar\pi}(m)}{\sqrt{m}}H(-n^2/4m; p).
	\end{equation}
    
    For technical simplifications we decompose $\alpha_\pi(\cdot)$ as,
    \begin{equation*}
        \alpha_\pi(v)=\alpha_\pi^\prime(v)+\alpha_\pi^{\prime\prime}(v)
    \end{equation*}
    where $\alpha_\pi^\prime(v)$(resp. $\alpha_\pi^{\prime\prime}(v)$) restricts $n$-sum to odd (resp. even). We rewrite $\M^\text{odd}(C, X)$ as,
    \begin{equation}\label{eq:decomp. of M(C, X)}
    \M^\text{odd}(C, X)=\sqrt{\frac{Q}{C}}\left(\M^\text{odd}(C, X; \alpha_\pi^\prime)+\M^\text{odd}(C, X; \alpha_\pi^{\prime\prime})\right).
    \end{equation}
    where,
    \begin{equation*}
		\begin{split}
			\M^\text{odd}(C, X; \alpha^\prime_\pi)=\underset{dl=\square}{\sum_{cq_1^2dl\sim C}}\eps_c^2\chi_\pi(cq_1^2dl)\mu(l)\sqrt{\frac{q_1^2d}{l}}\left(\frac{d}{c}\right)\sum_{r\ll p^{1+\eps}C/q_1^2dQ}\left(\frac{r}{c}\right)\alpha_\pi^\prime(rq_1^2d), 
		\end{split}
	\end{equation*}
    
    and 
    \begin{equation*}
		\begin{split}
			\M^\text{odd}(C, X; \alpha^\prime_\pi)=\underset{dl=\square}{\sum_{cq_1^2dl\sim C}}\eps_c^2\chi_\pi(cq_1^2dl)\mu(l)\sqrt{\frac{q_1^2d}{l}}\left(\frac{d}{c}\right)\sum_{r\ll p^{1+\eps}C/q_1^2dQ}\left(\frac{r}{c}\right)\alpha_\pi^{\prime\prime}(rq_1^2d). 
		\end{split}
	\end{equation*}

    Both of the above terms can be treated in similar fashion, so for a clear exposition we present the analysis only for $\M^\text{odd}(C, X; \alpha_\pi^\prime)$.\\
	
	Observe that, the presence of $\mu(l)$ in the above equation forces us to conclude that $l$ is square-free and we already have $dl=\square$, so we can further decompose $d$ as $d=ld_{\square}^2$. Now we further write $r$ as product of its square-free and square-full part i.e., $r=st^2$ and by smooth dyadic partitions we localize sizes of the variables, $s\sim S$ and $t\sim T$ such that 
	\begin{equation}\label{eq: st^2 bound}
	ST^2\ll \frac{p^{1+\eps}C}{q_1^2d_\square^2 lQ}.  
	\end{equation}
	After rearranging the sums we finally rewrite $\M^\text{odd}(X)$ as,
	
	\begin{equation}\label{eq:final expression}
		\begin{split}
			\M^\text{odd}(C, X; \alpha_\pi^\prime)=\sum_{q_1ld_\square\ll \sqrt{C}}&\chi_\pi^2(q_1ld_\square)\mu(l)q_1d_\square\sum_{t\sim T}\\&\times\sumflat_{s\sim S}\alpha_\pi^\prime(slt^2q_1^2d_\square^2)\underset{(c, t)=1}{\sumflat_{c\sim \frac{C}{q_1^2l^2d_{\square}^2}}}\eps_c^2\chi_\pi(c)\left(\frac{sl}{c}\right)
		\end{split}
	\end{equation}


    \subsection{Trivial estimate of $\M(X)$} We recover here the trivial bound for $\M(X)$. For this it would be enough to bound $\M^\text{odd}(C, X; \alpha_\pi^\prime)$. At first we estimate the expression trivially i.e. just by using the triangle inequality, and for this we need the following easy yet crucial bound of $\alpha_\pi^\prime$,
 
    \begin{lemma}[Pointwise estimate of $\alpha_\pi^\prime(v)$]\label{thm: pointwise bd of alpha}
        Let $v\in\Z$. We have, 
        \begin{equation}
            \alpha_\pi^\prime(v)\ll p^\eps.
        \end{equation}
    \end{lemma}
    \begin{proof}
	By the temperedness of $\pi$ and the square-root cancellation bound for the character sum $H(4m/n^2; p)$ we have,
	\begin{equation*}
		\alpha_\pi^\prime(v)\ll_\eps \underset{4m-n^2=v}{\sum_{m\sim M}\sum_{n\sim N}}\frac{1}{\sqrt{m}}\ll \sum_{n\ll p^{1/2+\eps}}\frac{1}{\sqrt{n^2+v}}\ll p^\eps.
	\end{equation*}
    \end{proof}

    We plug this estimate in the equation~\eqref{eq:final expression}, and use the bound \eqref{eq: st^2 bound}  for $ST^2$, 
    $$
    \M_u^\text{odd}(C, X; \alpha_\pi^\prime)\ll_\eps\sum_{q_1ld_\square\ll \sqrt{C}}\frac{STC}{q_1d_\square l^2}\ll_\eps pC\sum_{q_1d_\square l\ll \sqrt{C}}\frac{1}{q_1^2d_\square^2l^3}\ll pC.
    $$
    This implies, 
    \[
    \M^\text{odd}(C, X)\ll p^{1+\eps}\sqrt{CQ}, 
    \]
    which gives
    \[
    \M^\text{odd}(X)\ll\sup_{C\ll Q}\M^\text{odd}(C, X)\ll p^{1+\eps}Q=Xp^{1/2+\eps}.
    \]
    Therefore we finally have,
    \[
    \M(X)\ll X p^{1/2+\eps}.
    \]
    
    So to obtain the trivial bound we need to save $\sqrt{p}$. For this we exploit the presence of quadratic character to get square-root cancellation for one of the over $s$ and $c$. We would achieve this square root cancellation by applying the quadratic large sieve inequality. 
	\begin{theorem}[Quadratic large sieve]\label{thm:large sieve}
		Let $M, N\gg 1$, then for any sequence of complex numbers $(a_n)$ we have
		\begin{equation}\label{eq:large-sieve-1}
			\sumflat_{m\leq M}\left|\sumflat_{n\leq N}a_n\left(\frac{n}{m}\right)\right|^2\ll (M+N)(MN)^\eps\sum_{n\leq N}|a_n|^2,
		\end{equation}
		and
		\begin{equation}\label{eq:large-sieve-2}
			\sumflat_{m\leq M}\left|\sum_{n\leq N} a_n \left(\frac{n}{m}\right)\right|^2\ll (M+N)(MN)^\eps\sum_{n_1n_2=\square}|a_{n_1}a_{n_2}|,
		\end{equation}  
		where the sums are restricted to odd square-free integers.
	\end{theorem}
	\begin{proof}
		For a proof see~\cite[Theorem 1, Corollary 2]{MR1347489}.
	\end{proof}
	Applying the Cauchy's inequality, Lemma~\ref{thm: pointwise bd of alpha}, and followed by the quadratic large sieve we can bound $\M^\text{odd}(C, X)$ by,
	\begin{equation*}
		\begin{split}
			\sqrt{\frac{Q}{C}}\sum_{q_1d_\square l\ll \sqrt{C}}q_1d_\square\sum_{t\sim T}\left(\frac{C\sqrt{S}}{q_1^2d_\square^2l^2}+\frac{S\sqrt{C}}{q_1d_\square l}\right)&\ll \sum_{q_1ld_\square\ll\sqrt{C}}\frac{\sqrt{QCST^2}}{q_1l^2d_\square}+\frac{ST\sqrt{Q}}{l}\\
			&\hspace{-1cm}\ll_\eps \frac{pC}{\sqrt{Q}}\sum_{q_1ld_\square\ll\sqrt{C}}\frac{1}{q_1^2d_\square^2l^2}=O\left(p^{1+\eps}C/\sqrt{Q}\right).
		\end{split}
	\end{equation*}
	Therefore taking supremum over $C\ll Q$ we get our trivial bound,
    \[
    \M(X)\ll \sqrt{X}p^{3/4+\eps}.
    \]
    Hence it is instructed that to obtain a non-trivial bound we need to further exploit the presence of quadratic character and the oscillation of the sequence $\alpha_\pi(v)$.

	\section{Non-generic case}\label{sec: non-generic}
	In this section we have obtained a non-trivial upper bound for $\M^\text{odd}(C, X)$ in some cases, and as remarked earlier it would be enough to estimate the term $\M^\text{odd}(C, X; \alpha_\pi^\prime)$. We localize the size of $q_1d_\square l$ by smoothly decomposing the left most sum in the definition~\eqref{eq:final expression} in dyadic blocks of size $D$, where $D$ is some parameter which will be optimized later.

	\begin{equation}\label{eq:final_2}
			\begin{split}
				\M^\text{odd}(C, X; \alpha_\pi^\prime)=&\sum_{q_1ld_\square\sim D}\chi_\pi^2(q_1ld_\square)\mu(l)q_1d_\square\sum_{t\sim T}\sumflat_{s\sim S}\alpha_\pi^\prime(slt^2q_1^2d_\square^2)\\&\hspace{3.7cm}\times\underset{(c, t)=1}{\sumflat_{c\sim \frac{C}{q_1^2l^2d_{\square}^2}}}\eps_c^2\chi_\pi(c)\left(\frac{sl}{c}\right).
			\end{split}
	\end{equation}

	First we prove the following lemma which gives a non-trivial bound of this when $D$ is large.

	\begin{lemma}[Large $D$]\label{thm: large D}
		Let $D\geq p^{\delta_1}$ for some $\delta_1>0$. We have,
		\begin{equation}\label{eq:non-gen 1}
			\M^\text{odd}(C,X)\ll \sqrt{X}p^{\frac{3}{4}-\delta_1+\eps}.
		\end{equation}
	\end{lemma}

	\begin{proof}
		The proof is quite immediate, following the lines of proof of trivial estimate for $\M^\text{odd}(C, X)$ from the previous section we have,
		\begin{equation}
			\M^\text{odd}(C, X)\ll \frac{pC}{\sqrt{Q}}\sum_{q_1ld_\square\sim D}\frac{1}{q_1^2l^2d_\square^2}\ll\sqrt{X}p^{3/4+\eps}/D\leq \sqrt{X}p^{\frac{3}{4}-\delta_1+\eps}.
		\end{equation}
		Hence we have our result.
	\end{proof}

	Now it is left to treat the case when $D$ is small. In this case by further exploiting the presence of quadratic character we would be able to obtain a non-trivial bound when $T$ is quite large or in words when $S$ is smaller than its generic size. And in this regard we have the following lemma.

	\begin{lemma}[Small $D$ and large $T$]\label{thm: small D and large T}
		Let $D\leq p^{\delta_1}$, and $T\geq p^{\delta_2}$ for some $\delta_1$, and $\delta_2>0$. We have,
		\begin{equation}
			\M^\text{odd}(C, X)\ll_\eps \sqrt{X}p^{\frac{3}{4}-\delta_2}+\sqrt{X}p^{\frac{3}{4}+\frac{\delta_1}{4}-\frac{\delta_2}{2}}. 
		\end{equation}
	\end{lemma}

	\begin{proof}
		First we need to detect the co-primality condition $(c, t)=1$. Using the M\"obius function we write it as,
		\begin{equation*}
			1_{(c, t)=1}=\sum_{e|(c, t)}\mu(e).
		\end{equation*}
	We write $c=c_1e$, hence~\eqref{eq:final_2} becomes,                                                                        
		\begin{equation*}
			\begin{split}
				\M^\text{odd}(C, X; \alpha_\pi^\prime)&=\sum_{q_1ld_\square\sim D}\chi_\pi^2(q_1d_\square l)\mu(l)q_1d_\square\sum_{t\sim T}\sum_{e|t} \mu(e)\chi_\pi(e)\\&
				\times \sumflat_{s\sim S}\alpha_\pi^\prime(slt^2q_1^2d_\square^2)\left(\frac{sl}{e}\right)\sumflat_{c_1\sim \frac{C}{eq_1^2l^2d_\square^2}}\eps_{c_1e}^2\chi_\pi(c_1)\left(\frac{sl}{c_1}\right)
			\end{split}
		\end{equation*}
		Let the gcd $(s, l)=e_1$, we write $s=s_1e_1$ and $l=l_1e_1$, therefore we have,
		\begin{equation}\label{eq: main eq for lemma 2}
			\begin{split}
				\M^\text{odd}(C, X; \alpha_\pi^\prime)&=\sum_{q_1l_1e_1d_\square\sim D}\chi_\pi^2(q_1l_1e_1d_\square)\mu(l_1e_1)q_1d_\square\sum_{t\sim T}\sum_{e|t}\mu(e)\chi_\pi(e)\left(\frac{l_1}{e}\right)\\&
				\hspace{1.1cm}\times\sumflat_{s_1\sim S/e_1}\alpha_\pi^\prime(s_1l_1e_1^2t^2q_1^2d_\square^2)\left(\frac{s_1}{e}\right)\sumflat_{c_1\sim C_1}\eps_{c_1e}^2\chi_\pi(c_1)\left(\frac{s_1l_1}{c_1}\right)
			\end{split}
		\end{equation}
		To avoid notational cumbersome we have defined, 
		\begin{equation}\label{eq: C_1}
			C_1:=\frac{C}{eq_1^2l_1^2e_1^2d_\square^2}.
		\end{equation}
		By quadratic reciprocity law we could write the $c_1$-sum as following,
		\begin{equation*}
			\begin{split}
				\sumflat_{c_1\sim C_1}\eps_{c_1e}^2\chi_\pi(c_1)\left(\frac{s_1l_1}{c_1}\right)=\underset{c_1\equiv 1\md{4}}{\sumflat_{c_1\sim C_1}}+(-1)^{\frac{s_1l_1-1}{2}}\underset{c_1\equiv 3\md{4}}{\sumflat_{c_1\sim C_1}}\eps_{c_1e}^2\chi_\pi(c_1)\left(\frac{c_1}{s_1l_1}\right)
			\end{split}
		\end{equation*}
		As $\eps_{c_1e}^2=\pm 1\md{4}$ according to $c_1e\equiv\pm 1\md{4}$, so the right side of the above equation can be rewritten as,
		\begin{equation*}
		     \eps_{e}^4\left(\underset{c_1\equiv 1\md{4}}{\sumflat_{c_1\sim C_1}}+(-1)^{\frac{s_1l_1-1}{2}}\underset{c_1\equiv 3\md{4}}{\sumflat_{c_1\sim C_1}}\right)(-1)^{\frac{c_1e-1}{2}}\chi_\pi(c_1)\left(\frac{c_1}{s_1l_1}\right)
		\end{equation*}
		As both of the sums are are similar up to some signs depending on $s_1, l_1, c_1$ and $e$. So the analysis will be similar to the following one, 
		\begin{equation*}
			\underset{c_1\equiv 1 \md{4}}{\sumflat_{c_1\sim C_1}}\chi_\pi(c_1)\left(\frac{c_1}{s_1l_1}\right)
		\end{equation*}
		Now detecting the congruence condition $c_1\equiv 1 \md{4}$ using additive characters modulo $4$ we have,
		\begin{equation*}
			\sum_{\alpha\md{4}}4e\left(-\frac{\alpha}{4}\right)\sumflat_{c_1\sim C_1}\chi_\pi(c_1)e\left(\frac{\alpha c_1}{4}\right)\left(\frac{c_1}{s_1l_1}\right).
		\end{equation*}
		Now we want to apply Poisson summation formula on the $c_1$-sum but as it runs over only square-free integers, we need to detect the square-freeness by M\"obius to free from square-free restriction. Let $\Delta\geq 1$ be a parameter, 
		$$
		\mu^2(c_1)=\underset{\delta\leq \Delta}{\sum_{\delta^2|c_1}}\mu(\delta)+\underset{\delta>\Delta}{\sum_{\delta^2|c_1}}\mu(\delta).
		$$
		Write $c_1=c_0\delta^2$ then $\M^\text{odd}(C, X; \alpha_\pi^\prime)$ can be written as a combination of the sums of following type,           
		\begin{equation*}
			\mathcal{S}_1+\mathcal{S}_2,
		\end{equation*}
		where, 
		\begin{equation*}
			\begin{split}
				\mathcal{S}_1&=\sum_{q_1l_1e_1d_\square\sim D}\chi_\pi^2(q_1l_1e_1d_\square)\mu(l_1e_1)q_1d_\square\sum_{t\sim T}\sum_{e|t}\mu(e)\chi_\pi(e)\left(\frac{l_1}{e}\right)\\&\hspace{4cm}\times\sum_{\alpha\md{4}}e\left(-\frac{\alpha}{4}\right)\underset{(\delta,\, l_1)=1}{\sum_{\delta\leq\Delta}}\mu(\delta)\chi_\pi^2(\delta)\\&
				\times\underset{(s_1, \delta)=1}{\sumflat_{s_1\sim S/e_1}}\alpha^\prime(s_1l_1e_1^2t^2q_1^2d_\square^2)\left(\frac{s_1}{e}\right)\sum_{c_0\sim \frac{C_1}{\delta^2}}\chi_\pi(c_0)e\left(\frac{\alpha c_0\delta^2}{4}\right)\left(\frac{c_0}{s_1l_1}\right)
			\end{split}
		\end{equation*}
		and 
		\begin{equation*}
			\begin{split}
				\mathcal{S}_2&=\sum_{q_1l_1e_1d_\square\sim D}\chi_\pi^2(q_1l_1e_1d_\square)\mu(l_1e_1)q_1d_\square\sum_{t\sim T}\sum_{e|t}\mu(e)\chi_\pi(e)\left(\frac{l_1}{e}\right)\\&\hspace{4cm}\times\sum_{\alpha\md{4}}e\left(-\frac{\alpha}{4}\right)\underset{(\delta,\, l_1)=1}{\sum_{\delta>\Delta}}\mu(\delta)\chi_\pi^2(\delta)\\&
				\times\underset{(s_1, \delta)=1}{\sumflat_{s_1\sim S/e_1}}\alpha^\prime(s_1l_1e_1^2t^2q_1^2d_\square^2)\left(\frac{s_1}{e}\right)\sum_{c_0\sim \frac{C_1}{\delta^2}}\chi_\pi(c_0)e\left(\frac{\alpha c_0\delta^2}{4}\right)\left(\frac{c_0}{s_1l_1}\right)
			\end{split}
		\end{equation*}

		We first deal with the $\mathcal S_2$-sum. For this sum we simply use the quadratic large sieve~\eqref{eq:large-sieve-2} with the quadratic charcter $\left(\frac{c_0}{s_1}\right)$. Following the line of proof of trivial estimate from previous section, we have,

		\begin{equation}\label{eq: est of S_2}
			\begin{split}
				\mathcal{S}_2\ll_\eps \sqrt{\frac{C}{Q}}\sum_{q_1l_1e_1d_\square\sim D}\sum_{t\sim T}\sum_{e|t}\sum_{\delta>\Delta}\left(\frac{p\sqrt{C}/T^2\delta}{q_1^2l_12e_1^3d_\square^2\sqrt{e}}+
			\frac{C\sqrt{p}/T\delta^2}{q_1^2l_1^{5/2}e_1^3d_\square^2}\right)&\\&\hspace{-4.5cm}
			\ll_\eps \sqrt{\frac{C}{Q}}\left(\sqrt{X}p^{\frac{3}{4}-\delta_2}+\frac{X}{\Delta}\right).
			\end{split}
		\end{equation}
		Now we estimate $\mathcal{S}_1$. First we note that $c_0$-sum is running over odd numbers, so to make it free from this arithmetic condition we rewrite $c_0$-sum as following, 
		\begin{equation*}
			\sum_{c_0\sim \frac{C_1}{\delta^2}}\chi_\pi(c_0)e\left(\frac{\alpha c_0\delta^2}{4}\right)\left(\frac{c_0}{s_1l_1}\right)-\left(\frac{2}{s_1l_1}\right)\sum_{c_0\sim \frac{C_1}{2\delta^2}}\chi_\pi(c_0)e\left(\frac{\alpha c_0\delta^2}{2}\right)\left(\frac{c_0}{s_1l_1}\right)
		\end{equation*}

		Now we apply the Poisson summation formula separately on both of these sums. Since the analysis will be same for both of the sums, so we just treat the first sum. After the applying the summation formula we have, 

		\begin{equation*}
			\frac{C_1}{4ps_1l_1\delta^2}\sum_{c_0}\sum_{x\md{4ps_1l_1}}\chi_\pi(x)e\left(\frac{\alpha x\delta^2}{4}\right)\left(\frac{x}{s_1l_1}\right)e\left(\frac{c_0x}{4ps_1l_1}\right)\int_{\xi\sim 1}e\left(\frac{c_0C_1\xi}{4ps_1l_1\delta^2}\right)\, d\xi.
		\end{equation*}

		Up to a negligible error by the integration by-parts bound we truncate the $c_0$-sum to $c_0\ll 4p^{1+\eps}s_1l_1\delta^2/C_1$. Now we evaluate the character sum, using the Chinese remainder theorem we rewrite the $x\md{4ps_1l_1}$-sum as, 

		\begin{equation*}
			\sum_{x\md{4p}}\chi_\pi(x)e\left(\frac{\alpha x\delta^2}{4}\right)e\left(\frac{c_0\overline{s_1l_1}x}{4p}\right)\sum_{y\md{s_1l_1}}\left(\frac{y}{s_1l_1}\right)e\left(\frac{c_0\overline{4p}y}{s_1l_1}\right)
		\end{equation*}

		Since $s_1l_1$ is odd and square-free so, by a simple change of variable the $y\md{s_1l_1}$-sum becomes, 

		\begin{equation*}
			\sqrt{s_1l_1}\left(\frac{c_0p}{s_1l_1}\right).
		\end{equation*}

		And again by the Chinese remainder theorem and a change of variable, we could write $x\md{4p}$-sum as,

		\begin{equation*}
			\sqrt{p}\,\chi_\pi(4\bar{c_0}s_1l_1)\eps(\chi_\pi)\sum_{x\md{4}}e\left(\frac{\alpha ps_1l_1\delta^2 x}{4}+\frac{c_0x}{4}\right).
		\end{equation*}

		We remark that in the above the sum we have not evaluated the the $x\md{4}$-sum because otherwise it would get reduce to a congruence condition involving $s_1$ and $c_0$ which in turn would break the bilinear structure between the sums over $s_1$ and $c_0$. By combining them we get required evaluation for the  character sum. Hence up to an negligible error and suppressing the wight function under the asymptotic sign, dual of the $c_0$-sum becomes,
		\begin{equation*}
			\begin{split}
				&\frac{C_1}{4\delta^2\sqrt{ps_1l_1}}\chi_\pi(4s_1l_1)\eps(\chi_\pi)\left(\frac{p}{s_1l_1}\right)\sum_{x\md{4}}e\left(\frac{\alpha s_1l_1\delta^2 x}{4}\right)\\&\hspace{2.3cm}\times\sum_{c_0\ll \frac{4p^{1+\eps}s_1l_1\delta^2}{C_1}}\overline{\chi}_\pi(c_0)e\left(\frac{c_0x}{4}\right)\left(\frac{c_0}{s_1l_1}\right).
			\end{split}
		\end{equation*}
		Now we plug this dual expression of $c_0$-sum in $\mathcal S_1$, and using triangle inequality we could estimate $\mathcal S_1$ as follows,
		\begin{equation*}
			\begin{split}
				\mathcal{S}_1&\ll_\eps\frac{C}{\sqrt{p}}\sum_{q_1l_1e_1d_\square\sim D}\frac{1}{q_1l_1^{5/2}e_1^2d_\square}\sum_{t\sim T}\sum_{e|t}\frac{1}{e}\sum_{x\md{4}}\sum_{\delta\leq \Delta}\frac{1}{\delta^2}\\&
				\times \sumflat_{s_1\sim \frac{S}{e_1}}|s_1|^{-1/2}|\alpha_\pi^\prime(s_1l_1e_1^2t^2q_1^2d_\square^2)|\left|\sum_{c_0\ll_\eps \frac{pSl_1\delta^2}{e_1C_1}}\beta(c_0, x, l_1)\left(\frac{c_0}{s_1}\right)\right|,
			\end{split}
		\end{equation*}
		where
		\begin{equation*}
			\beta(c_0, x, l_1):=\bar{\chi}_\pi(c_0)e\left(\frac{c_0x}{4}\right)\left(\frac{c_0}{l_1}\right).
		\end{equation*}
		Now we could apply the quadratic large sieve inequality~\eqref{eq:large-sieve-2} and the bound Lemma~\ref{thm: pointwise bd of alpha} which implies, 
		\begin{equation*}
			\begin{split}
				\mathcal{S}_1&\ll_\eps \frac{C}{\sqrt{p}}\sum_{q_1l_1e_1d_\square\sim D}\frac{1}{q_1l_1^{5/2}e_1^2d_\square}\sum_{t\sim T}\sum_{e|t}\frac{1}{e}\sum_{\delta\leq \Delta}\frac{1}{\delta^2}\\&\hspace{1.1cm}\times\left(\frac{S\delta\sqrt{p}}{\sqrt{C}}q_1l_1^{3/2}d_\square\sqrt{e}+\frac{pS\delta^2}{C}q_1^2l_1^3e_1d_\square^2\right)
			\end{split}
		\end{equation*}

		Further simplification gives,

		\begin{equation}\label{eq: estimate of S_1}
			\mathcal{S}_1\ll\sqrt{\frac{C}{Q}}\left(\sqrt{X}p^{3/4-\delta_2+\eps}+\Delta\sqrt{D}p^{3/2-\delta_2+\eps}\right).
		\end{equation}

		Combining the estimates~\eqref{eq: estimate of S_1} and \eqref{eq: est of S_2} we have, 
		\begin{equation}
			\M^\text{odd}(C, X; \alpha_\pi^\prime)\ll_\eps \sqrt{X}p^{\frac{3}{4}-\delta_2}+X\Delta^{-1}+\Delta\sqrt{D}p^{3/2-\delta_2}.
		\end{equation}
		Optimizing the parameter $\Delta$ in the above equation completes the proof of this lemma.

	\end{proof}

	\section{Generic case}\label{sec: generic}
	
	   In previous section we have obtained a non-trivial bound for $\M_u^\text{odd}(C, X)$ in various cases (see Lemma~\ref{thm: large D} and Lemma~\ref{thm: small D and large T}). The only remaining case is when $D$ and $T$ both of them are small which we call as generic case. In this case we would obtain non-trivial bound for $\M_u^\text{odd}(C, X)$ by improving the $\ell^2$-norm estimate of sequence $\alpha^\prime(\cdot)$ in $p$-aspect. We record our result in this direction as the following lemma. 
	
	\begin{lemma}[Small $D$ and small $T$]\label{thm:Small $D$ and small $T$}
		Let $D\leq P^{\delta_1}$ and $T\leq p^{\delta_2}$, then we have
		\begin{equation}
			\M^\text{odd}(C, X)\ll_\eps Xp^{\delta_2-\frac{1}{16}}+\sqrt{X}p^{\frac{3}{4}-\frac{1}{16}}.
		\end{equation}
	\end{lemma}

	\begin{proof}
		We star by recalling the equation $\M^\text{odd}(C, X; \alpha_\pi^\prime)$ as given in~\eqref{eq:final_2}, 
		\begin{equation*}
			\begin{split}
				\M^\text{odd}(C, X; \alpha_\pi^\prime)=&\sum_{q_1ld_\square\sim D}\chi_\pi^2(q_1ld_\square)\mu(l)q_1d_\square\sum_{t\sim T}\sumflat_{s\sim S}\alpha_\pi^\prime(slt^2q_1^2d_\square^2)\\&\hspace{3.7cm}\times\underset{(c, t)=1}{\sumflat_{c\sim \frac{C}{q_1^2l^2d_{\square}^2}}}\eps_c^2\chi_\pi(c)\left(\frac{sl}{c}\right).
			\end{split}
		\end{equation*}

		The initial step is similar to the proof of trivial estimate, we apply Cauchy's inequality over $s$-sum and then apply the quadratic large sieve~\eqref{eq:large-sieve-1}. We have, 
		\begin{equation}\label{thm: generic ineq 1}
			\M^\text{odd}(C, X; \alpha_\pi^\prime)\ll_\eps\sum_{q_1ld_\square\sim D}q_1d_\square\sum_{t\sim T}\left(\frac{C}{q_1^2l^2d_\square^2}+\frac{\sqrt{CS}}{q_1ld_\square}\right)\sqrt{\tilde\alpha(lq_1^2d_\square^2t^2)},
		\end{equation}
		where,
		\begin{equation*}
			\tilde\alpha(\omega):=\sum_{s\sim S}\left|\alpha_\pi^\prime(s\omega)\right|^2.
		\end{equation*}

		Now our objective is to estimate the size of $\tilde\alpha(\omega)$. Opening the absolute value square we have,
		\begin{equation*}
			\begin{split}
				\tilde\alpha(\omega)&=\sum_{s\sim S}\underset{4m_1-n_1^2=s\omega=4m_2-n_2^2}{\sum_{m_1\ll p^{1+\eps}}\sum_{n_1\ll p^{1/2+\eps}}\sum_{m_2\ll p^{1+\eps}}\sum_{n_2\ll p^{1/2+\eps}}}\frac{\lambda_\pi(m_1)}{\sqrt{m_1}}\frac{\lambda_{\bar\pi}(m_2)}{\sqrt{m_2}}C(4m_1, n_1^2; p)\overline{C(4m_2, n_2^2; p)}
			\end{split}
		\end{equation*}

		Note that we have used denoted $H(-n^2/4m; p)$ for notational purpose. Note that the above expression of $\tilde\alpha(\omega)$ already contains many terms, which would make the equations in the further analysis quite cumbersome. So to avoid this we will treat the oscillatory integrals $\mathcal{J}(\cdot)$ as smooth weight functions and suppress the weight function  under the asymptotic sign. Hence the simplified expression is,

		\begin{equation*}
			\tilde\alpha(\omega)=\sum_{s\sim S}\underset{4m_1-n_1^2=s\omega=4m_2-n_2^2}{\sum_{m_1\ll p^{1+\eps}}\sum_{n_1\ll p^{1/2+\eps}}\sum_{m_2\ll p^{1+\eps}}\sum_{n_2\ll p^{1+\eps}}}\frac{\lambda_\pi(m_1)}{\sqrt{m_1}}\frac{\lambda_{\bar\pi}(m_2)}{\sqrt{m_2}}C(4m_1, n_1^2; p)\overline{C(4m_2, n_2^2; p)}
		\end{equation*}

		The Diagonal contribution of this is given by,

		\begin{equation}\label{eq: first diagonal estimate}
		\underset{4m_1\equiv n_1^2\md{\omega}}{\sum_{m_1\ll p^{1+\eps}}\sum_{n_1\ll p^{1/2+\eps}}}\frac{|\lambda_\pi(m_1)|^2}{m_1}|C(4m_1, n_1^2; p)|^2\ll p^{\frac{1}{2}+\eps},
		\end{equation}

		Where the last estimate follows from the Ramanujan bound for Fourier coefficients (though Ramanujan bound on average is enough), and the square root cancellation bound Lemma~\ref{thm: Square-root cancellation} Note that if $n_1=n_2$ then it forces that $m_1=m_2$. So we are left with the off-diagonal part which we call as $\Omega$.

		\begin{equation*}
			\begin{split}
				\Omega&:=2\sum_{m_1\ll p^{1+\eps}}\frac{\lambda_\pi(m_1)}{\sqrt{m_1}}\underset{n_1\neq n_2}{\underset{n_1, n_2\ll p^{1/2+\eps}}{\sum\,\sum}}\frac{\lambda_{\bar \pi}\left(m_1+\frac{n_2^2-n_1^2}{4}\right)}{\sqrt{4m_1+n_2^2-n_1^2}}\\&\hspace{1.6cm}\times C(4m_1, n_1^2; p)\overline{C(4m_1+n_2^2-n_1^2, n_2^2; p)}
			\end{split}
		\end{equation*}
		In this step we first remove the  $\lambda_\pi(m_1)/\sqrt{m_1}$, by appling Cauchy's inequality on the $m_1$-sum we get,

        \begin{equation*}
            \Omega\ll_\eps \sqrt{\tilde\Omega},
        \end{equation*}
        where, 
        
		\begin{equation*}
			\begin{split}
				\tilde\Omega = \sum_{m_1\ll p^{1+\eps}}\Biggm|\underset{n_1\neq n_2}{\underset{n_1^2\equiv 4m_1\md{\omega}}{\sum_{n_1\ll p^{1/2+\eps}}\sum_{n_2\ll p^{1/2+\eps}}}}&\frac{\lambda_{\bar \pi}\left(m_1+\frac{n_2^2-n_1^2}{4}\right)}{\sqrt{4m_1+n_2^2-n_1^2}}C(4m_1, n_1^2; p)\\&\hspace{0.7cm}\times\overline{C(4m_1+n_2^2-n_1^2, n_2^2; p)}\Biggm|^2.
			\end{split}
		\end{equation*}
		Opening the absolute value square it get equals, 
        
        \begin{equation*}
            \begin{split}
                \tilde\Omega=&\sum_{m_1\ll p^{1+\eps}}\underset{n_1\neq n_2}{\underset{n_1^2\equiv 4m_1\md{\omega}}{\sum_{n_1\ll p^{\frac{1}{2}+\eps}}\sum_{n_2\ll p^{\frac{1}{2}+\eps}}}}\underset{n_3\neq n_4}{\underset{n_3^2\equiv 4m_1\md{\omega}}{\sum_{n_3\ll p^{\frac{1}{2}+\eps}}\sum_{n_4\ll p^{\frac{1}{2}+\eps}}}}\frac{\lambda_{\bar\pi}\left(m_1+\frac{n_2^2-n_1^2}{4}\right)}{\sqrt{4m_1+n_2^2-n_1^2}}\frac{\lambda_{\pi}\left(m_1+\frac{n_4^2-n_3^2}{4}\right)}{\sqrt{4m_1+n_4^2-n_3^2}}\\&\hspace{0.2
                cm}\times
                C(4m_1, n_1^2; p)\overline{C(4m_1+n_2^2-n_1^2, n_2^2; p)}\,\,\overline{C(4m_1, n_3^2; p)}C(4m_1+n_4^2-n_3^2, n_4^2; p)
            \end{split}    
        \end{equation*}
	   The diagonal contribution, say $\tilde\Omega_0$, is given by, 

        \begin{equation*}
            \begin{split}
                \tilde\Omega_0 = \sum_{m_1\ll p^{1+\eps}}\underset{n_1^2\equiv 4m_1\md{\omega}}{\sum_{n_1\ll p^{\frac{1}{2}+\eps}}}\sum_{n_2\ll p^{\frac{1}{2}+\eps}}&\frac{\left|\lambda_\pi\left(m_1+\frac{n_2^2-n_1^2}{4}\right)\right|^2}{4m_1+n_2^2-n_1^2}|C(4m_1, n_1^2; p)|^2 \\&\hspace{2cm}\times|C(4m_1+n_2^2-n_1^2)|^2
            \end{split}
        \end{equation*}
        Using the square-root cancellation bound and there after by identifying $4m_1+n_2^2-n_1^2=u$ we get, 

        \begin{equation*}
            \tilde\Omega_0\ll_\eps \sum_{1\leq u\ll p^{1+\eps}}\frac{|\lambda_\pi(4u)|^2}{u}\sharp\{(m_1, n_1, n_2) : 4m_1+n_2^2-n_1^2=u\}
        \end{equation*}
	   It is obvious that the above cardinality is $O(p^{1+\eps})$ and plugging this estimate and using the Ramanujan bound on average we get, 

	   \begin{equation}\label{eq: omega_0 estimate}
		\tilde\Omega_0\ll_\eps p^{1+\eps}.
	   \end{equation}
	   Now we estimate the contribution, say $\tilde\Omega_{00}$, when at least one of the $n_1$ and $n_2$ is equal to either $n_3$ or $n_4$. Up to $p^\eps$, this contribution is bounded by the sum of the following type,
	   \begin{equation*}
		\begin{split}
			&\sum_{m_1\ll p^{1+\eps}}\sum_{n_1\ll p^{1/2+\eps}}\left(\sum_{n_2\ll p^{\frac{1}{2}+\eps}}\frac{1}{\sqrt{4m+n_1^2-n_2^2}}\right)^2\\&\ll p^{1/2+\eps}\sum_{m_1\ll p^{1+\eps}}\sum_{n_1, n_2\ll p^{1/2+\eps}}\frac{1}{4m_1+n_1^2-n_2^2}
		\end{split}
	   \end{equation*}
	   Following the lines of the proof of diagonal estimate we can bound this sum by $p^{3/2+\eps}$.\\

	   Now we are left with the off diagonal term where none of the $n_i$ is equal to $n_j$, we call it as \textit{super off-diagonal}, which we call $\tilde\Omega_*$. At first we make the following change of variables,  
	   $$
	   4m=4m_1+n_2^2-n_1^2,\hspace{0.5cm}\text{and}\hspace{0.5cm} 4h=n_1^2-n_2^2-n_3^2+n_4^2.
	   $$
	   Due to the first change of variable we replce the congruence, $n_1^2\equiv 4m_1\md{\omega}$ by $n_2^2\equiv 4m\md{\omega}$. Now we can rewrite the off-diagonal contribution as following, 

	   \begin{equation*}
		\tilde\Omega_*=\sum_{m\ll p^{1+\eps}}\sum_{|h|\ll p^{1+\eps}}\frac{\lambda_{\bar\pi}(m)}{\sqrt{m}}\frac{\lambda_{\pi}(4(m+h))}{\sqrt{m+h}}\underset{n_1^2-n_2^2-n_3^2+n_4^2=4h}{\underset{n_i\neq n_j;\, \forall i\neq j}{\underset{n_2^2\equiv 4m\md{\omega}}{\underset{n_1^2\equiv n_3^2\md{\omega}}{\underset{n_1, n_2, n_3, n_4 \ll p^{\frac{1}{2}+\eps}}{\sum\cdots\sum}}}}}\mathfrak{C}(n_1, n_2, n_3, n_4; m),
	   \end{equation*}
	   where, 
	   \begin{equation}\label{eq:mathfrak C}
		\begin{split}
			&\hspace{0.cm}\mathfrak{C}(n_a, n_b, n_c, n_d; m):=C(4m+n_a^2-n_b^2, n_a^2; p)\overline{C(4m, n_b^2; p)}\\&\times\overline{C(4m+n_a^2-n_b^2, n_c^2; p)} C(4m+n_a^2-n_b^2-n_c^2+n_d^2, n_d^2; p).
		\end{split}
	   \end{equation}
	   Now we smooth out the sums over $m$ and $h$ by applying the Cauchy's inequality. We have,

	   \begin{equation}\label{eq: est of omega_*}
		\tilde\Omega_*\ll_\eps \sqrt{\Theta}.
	   \end{equation}
	    where, 

	    \begin{equation*}
		\Theta=\sum_{m\ll p^{1+\eps}}\sum_{|h|\ll p^{1+\eps}}\left|\underset{n_1^2-n_2^2-n_3^2+n_4^2=4h}{\underset{n_2^2\equiv 4m\md{\omega}}{\underset{n_1^2\equiv n_3^2\md{\omega}}{\underset{n_i\neq n_j\, \forall i\neq j}{\underset{n_1, n_2, n_3, n_4 \ll p^{\frac{1}{2}+\eps}}{\sum\cdots\sum}}}}}\mathfrak{C}(n_1, n_2, n_3, n_4; m)\right|^2.
	    \end{equation*}
	    After opening absolute square and executing the $h$-sum it becomes,

	    \begin{equation*}
		\Theta=\sum_{m\ll p^{1+\eps}}\underset{n_1^2-n_2^2-n_3^2+n_4^2=n_5^2-n_6^2-n_7^2+n_8^2}{\underset{n_2^2\equiv 4m \md{\omega}, n_6^2\equiv 4m\md{\omega}}{\underset{n_1^2\equiv n_3^2\md{\omega}, n_5^2\equiv n_7^2\md{\omega}}{\underset{n_i\neq n_j (\text{resp.}\, n_{i+4}\neq n_{j+4});\,1\leq i\neq j\leq 4}{\sum_{n_1}\cdots\sum_{n_4}\sum_{n_5}\cdots\sum_{n_8}}}}}\mathfrak{C}(n_1, n_2, n_3, n_4; m)\overline{\mathfrak{C}(n_5, n_6, n_7, n_8; m)}
	    \end{equation*}
	    Interchanging the sums we could rewrite $\Theta$ as,

	    \begin{equation}\label{eq: before square root cancellation}
		\begin{split}
			\Theta=\underset{n_1^2\equiv n_3^2\md{\omega}, n_2^2\equiv n_6^2\md{\omega}, n_5^2\equiv n_7^2\md{\omega}}{\underset{n_1^2-n_2^2-n_3^2+n_4^2=n_5^2-n_6^2-n_7^2+n_8^2}{\underset{n_i\neq n_j(\text{resp.}\, n_{i+4}\neq n_{j+4})\, 1\leq i\neq j\leq 4}{\sum_{n_1}\cdots\sum_{n_4}\sum_{n_5}\cdots\sum_{n_8}}}}\hspace{-0.7cm}\underset{4m\equiv n_2^2\md{\omega}}{\sum_{m\ll p^{1+\eps}}}\mathfrak{C}(n_1,\cdots, n_4; m)\overline{\mathfrak{C}(n_5,\cdots, n_8; m)}
		\end{split}
	    \end{equation}                                               
	    
          There are some cases where the trivial estimate of $\Theta$ would suffices. We record these in following lemmas.   

	    \begin{lemma}\label{thm: first}
		If $n_2=n_6$ or $n_4=n_8$ then, 
		\begin{equation}
			\Theta\ll p^{7/2+\eps}. 
		\end{equation}
	    \end{lemma}
	    \begin{proof}
		Let us consider the case $n_2=n_6$, and proof in the other case will be similar. By triangle inequality and using the square-root estimate (Corollary~\ref{cor: sq cancellation}) we get, 
		\begin{equation*}
			\Theta\ll p^{3/2+\eps}\underset{n_1^2 + n_7^2 + n_4^2 = n_5^2 + n_3^2 + n_8^2}{\sum\cdots\sum}1=p^{3/2+\eps}\sum_{d\ll p^{1+\eps}}r_3(d)^2.
		\end{equation*}
		where $r_3(d)$ is a number ways an integer can be expressed as a sum of three squares i.e., 
		\begin{equation*}
			r_3(d):=\underset{n_1^2+n_2^2+n_3^2=d}{\sum\cdots\sum}1.
		\end{equation*}
		Using the point-wise estimate $r_3(d)\ll \sqrt{d}$ (see~\cite[\S 20.4]{MR2061214}) we have 
		\begin{equation*}
			\Theta\ll p^{3/2+\eps}\sum_{d\ll p^{1+\eps}}d\ll p^{7/2+\eps}.
		\end{equation*}
	    \end{proof}
	    \begin{lemma}\label{thm: second}
		If $n_1^2+n_4^2$ is equal to one of the $n_2^2+n_3^2$ and $n_3^2+n_8^2$ we have,
		\begin{equation}
			\Theta\ll p^{3+\eps}.
		\end{equation}
	    \end{lemma}
	    \begin{proof}
		We consider the case, $n_1^2+n_4^2=n_2^2+n_3^2$, and other case will follow similarly. In this case, triangle inequality, and along with the square-root cancellation estimate, we can bound $\Theta$ by, 
		\begin{equation*}
			\Theta\ll p^{1+\eps}\left(\underset{n_1^2+n_4^2=n_2^2+n_3^2}{\sum_{n_1}\sum_{n_4}\sum_{n_2}\sum_{n_3}}\;1\right)^2=p^{1+\eps}\left(\sum_{d\ll p^{1+\eps}}r_2^2(d)\right)^2
		\end{equation*}
		where, $r_2(d)$ stands for the number of represenstations of $d$ as sum of two squares, and using its well knwon estimate $r_2(d)\ll d^{1+\eps}$ we obtain our required estimate. 
	    \end{proof}
         Similarly we have,
	    \begin{lemma}\label{thm: third}
		If $n_1^2+n_6^2$ is equal to one of the $n_2^2+n_3^2$ and $n_2^2+n_5^2$ then we have,
		\begin{equation}
			\Theta\ll p^{3+\eps}.
		\end{equation}
	    \end{lemma}\qed

	    Now our job is to bound $\Theta$ non-trivially when $n_i$'s statisfies none of the above relations. We will do so by estimating the $m$-sum in~\eqref{eq: before square root cancellation} non-trivially. We begin by writing the congruence $4m\equiv n_2^2\md{\omega}$ as $4m=n_2^2+k\omega$. The $m$-sum in~\eqref{eq: before square root cancellation} becomes,

	    \begin{equation}\label{eq: sum-k-before poisson}
		\sum_{1\leq k\ll \frac{p^{1+\eps} -n_2^2}{\omega}}\mathfrak{C}\left(n_1, \cdots, n_4; \frac{n_2^2+k\omega}{4}\right)\mathfrak{C}\left(n_5, \cdots, n_8; \frac{n_2^2+k\omega}{4}\right)
	    \end{equation}

	    First step would be to turn it into a complete sum by applying the Poisson summation formula. Up to a negligible error it becomes,

	    \begin{equation}\label{eq: sum-k-after poisson}
		\begin{split}
			\frac{p^{1+\eps}-n_2^2}{p\omega}\sum_{k\ll \frac{p^{1+\eps}\omega}{p^{1+\eps}-n_2^2}}\mathcal C(k),
		\end{split}
	    \end{equation}
	    where the character-sum is given by, 
	    \begin{equation*}
		\begin{split}
			\mathcal C(k):=\sum_{x\md{p}}\mathfrak{C}\left(n_1, \cdots, n_4; \frac{n_2^2+x\omega}{4}\right)\mathfrak{C}\left(n_5, \cdots, n_8; \frac{n_2^2+x\omega}{4}\right)e\left(\frac{xk}{p}\right)
		\end{split}
	    \end{equation*}
	    So it would be enough to estimate this character-sum non-trivially. Unwinding the definition of $\mathfrak{C}(\cdots)$ we could rewrite it as, 
	    \begin{equation}\label{eq:C(k)}
		\begin{split}
			\mathcal C(k)= &\sum_{x\in\F_p}H\left(\frac{x\omega+n_1^2}{n_1^2}\right)H\left(\frac{x\omega+n_2^2}{n_6^2}\right)H\left(\frac{x\omega+n_1^2-n_3^2+n_4^2}{n_4^2}\right)\\&\hspace{0.4cm}\times H\left(\frac{x\omega+n_2^2+n_5^2-n_6^2}{n_7^2}\right)
			\bar{H}\left(\frac{x\omega+n_1^2}{n_3^2}\right)\bar{H}\left(\frac{x\omega+n_2^2}{n_2^2}\right)\\&\hspace{-.2cm}\times \bar{H}\left(\frac{x\omega+n_1^2-n_3^2+n_4^2}{n_8^2}\right)\bar{H}\left(\frac{x\omega+n_2^2+n_5^2-n_6^2}{n_5^2}\right)e\left(\frac{xk}{p}\right)
		\end{split}
	    \end{equation}
	    We have estimated this exponential-sum and proved its square-root cancellation in Lemma~\ref{thm: Square-root cancellation} of the next section.\\
        \begin{lemma}\label{thm: fourth}
            Let $n_i\neq n_j$ and $n_{i+4}\neq n_{j+4}$ for any $1\leq i\neq j\leq 4$. If $n_i$'s satisfy none of the conditions of Lemma~\ref{thm: first}, Lemma~\ref{thm: second}, and Lemma~\ref{thm: third}. Then we have, 
            \begin{equation*}
                \Theta\ll p^{7/2+\eps}.
            \end{equation*}
        \end{lemma}
        \begin{proof}
            Plugging the bound of $\mathcal C(k)\ll \sqrt p$ (Lemma~\ref{thm: Square-root cancellation}) in \eqref{eq: sum-k-after poisson} we bound \eqref{eq: sum-k-before poisson} from above by $p^{1/2+\eps}$. Hence combining this with the expression~\eqref{eq: before square root cancellation} of $\Theta$ we get,
            \begin{equation*}
                \Theta\ll p^{1/2+\eps}\sum_{h\ll p^{1+\eps}}\left(\underset{n_1^2+\cdots+n_4^2=h}{\sum_{n_1}\cdots\sum_{_{n_4}}}\;1\right)^2=p^{1/2+\eps}\sum_{h\ll p^{1+\eps}}r_4(h)^2
            \end{equation*}
            Similarly as before, using the point-wise bound, $r_4(h)\ll h$ we get our expected bound.
        \end{proof}
        
        Finally combining all of the above estimates in~\eqref{thm: generic ineq 1} we get, 
        \begin{equation*}
            \M^\text{odd}(C, X; \alpha_\pi^\prime)\ll p^{7/16+\eps}\sum_{q_1ld_\square\sim D}\sum_{t\sim T}\left(\frac{C}{q_1d_\square l}+\frac{\sqrt{CS}}{l}\right)
        \end{equation*}
        Therefore, 
        \begin{equation*}
            \M^\text{odd}(C, X)\ll p^{7/16+\eps}\sum_{q_1ld_\square\sim D}\sum_{t\sim T}\left(\frac{\sqrt{CQ}}{q_1d_\square l}+\frac{\sqrt{CS}}{l}\right)
        \end{equation*}
        Now we use the bounds, $C\ll Q={X}/{\sqrt{p}}$, and $ST^2\ll p^{1+\eps}/q_1^2d_\square^2l$. After this executing outer summations we get our desired bounds. This completes the proof of Lemma~\ref{thm:Small $D$ and small $T$}.
	    \end{proof}

	    \section{Estimate of an exponential-sum}\label{sec: exp sums} 

        Our objective of this section is to estimate the exponential sum $\mathcal C(k)$. To do so, we use the method of $\ell$-adic cohomology.  

	    First we make few relabeling as follows, 

	    \begin{equation}\label{eq: change of variables}
		\begin{split}
			&n_1\mapsto m_1;\; n_2\mapsto m_6;\; n_3\mapsto m_5;\; n_4\mapsto m_4,\\
			& n_5\mapsto m_7;\; n_6\mapsto m_2;\; n_7\mapsto m_3;\; n_8\mapsto m_8.
		\end{split}
	    \end{equation}
	    Therefore~\eqref{eq:C(k)} becomes 

	    \begin{equation}
		\begin{split}
			\mathcal C(k)= &\sum_{x\in\F_p}H\left(\frac{x\omega+m_1^2}{m_1^2}\right)H\left(\frac{x\omega+n_6^2}{m_2^2}\right)H\left(\frac{x\omega+m_6^2+m_7^2-m_2^2}{m_3^2}\right)\\&\hspace{0.3cm}\times H\left(\frac{x\omega+m_1^2-m_5^2+m_4^2}{m_4^2}\right)
			\bar{H}\left(\frac{x\omega+m_1^2}{m_5^2}\right)\bar{H}\left(\frac{x\omega+m_6^2}{m_6^2}\right)\\&\hspace{-.4cm}\times \bar{H}\left(\frac{x\omega+m_6^2+m_7^2-m_2^2}{m_7^2}\right)\bar{H}\left(\frac{x\omega+m_1^2-m_5^2+m_4^2}{m_8^2}\right)e\left(\frac{xk}{p}\right)
		\end{split}
	    \end{equation}
	    Note that the relation $n_1^2-n_2^2-n_3^2+n_4^2=n_5^2-n_6^2-n_7^2+n_8^2$ transforms into
	    \begin{equation}
		m_1^2+m_2^2+m_3^2+m_4^2=m_5^2+m_6^2+m_7^2+m_8^2.
	    \end{equation}
	     We have shown in \S\ref{sec:hyp geo sum} that we can attach an $\ell$-adic sheaf to $H(\lambda)$, whose geometric properties are listed in Lemma~\ref{thm:geo prop}. Consider the following sheaf, 
		\begin{equation}
			\mathcal F=\bigotimes_{1\leq i\leq 4}(\gamma_i^*\mathcal H\otimes D(\gamma_{i+4}^*\mathcal H))\otimes \mathcal L_{\psi(k X)}
		\end{equation}
		where, 
		\begin{equation*}
			\gamma_i=
			\begin{pmatrix}
				c & b_i\\
				0 & m_i^2
			\end{pmatrix}\in \PGL_2(\overline\F_p)\hspace{0.2cm}\text{for}\hspace{0.2cm}1\leq i\leq 8,\hspace{0.2cm}\text{and}\hspace{0.2cm} b_i=b_{i+4}\neq 0.
		\end{equation*}
		with, 
		\begin{equation*}
			b_1=m_1^2,\; b_2=m_6^2,\; b_3=m_6^2+m_7^2-m_2^2,\hspace{0.2cm}\text{and}\hspace{0.2cm} b_4=m_1^2+m_4^2-m_5^2.
		\end{equation*}
		And $\gamma_i^*\mathcal H$ is the pullback sheaf of $\mathcal H$ along $\gamma_i$, from geometric properties Lemma~\ref{thm:geo prop} it follows that its singularity set is,

		Consider the following set $U\subset \A^1(\F_p)$
		\begin{equation*}
			U(\F_p)=\A^1(\F_p)\backslash\bigcup_{1\leq i\leq 8}\left\{\frac{-b_i}{c},\; \frac{m_i^2-b_i}{c}\right\}
		\end{equation*}
		So, up to an error of size $O(1)$ the sum $\mathcal C(k)$ becomes, 
		\begin{equation*}
			\begin{split}
				\mathcal C(k)&=\sum_{x\in U(\F_p)}H(\gamma_1\cdot x)H(\gamma_2\cdot x)H(\gamma_3\cdot x)H(\gamma_4\cdot x)\\&
				\times \bar{H}(\gamma_5\cdot x)\bar{H}(\gamma_6\cdot x)\bar{H}(\gamma_7\cdot x)\bar{H}(\gamma_8\cdot x)e\left(\frac{kx}{p}\right)
			\end{split} 
		\end{equation*}
        By the Grothendieck-Lefschetz trace formula we have, 
        \begin{equation*}
            \mathcal C(k)=\sum_{x\in U(\F_p)}\text{tr}(\text{Fr}_{x}|\mathcal F)=\sum_{0\leq i\leq 2}(-1)^{i}\;\text{tr}(\text{Fr}_x|H_c^i(U_{\bar{\F}_p}; \mathcal F))
        \end{equation*}
        We will show that the extreme cohomologies are zero. Since $U\neq P_{\F_p}^1$, the lower extreme cohomology group vanishes, i.e., $H_c^0(U_{{\bar F}_p}; \mathcal F)=0$. To establish the square-root cancellation we need to show the other extreme of the cohomology group $H_c^2(U_{\bar{F}_p}; \mathcal F)$ is zero. For this we would follow~\cite{FKM15}.\\
        First we need to show, the sheaf $\mathcal F$ is $U$-generous in the sense of Definition 2.1 of loc. cit.,\\
        \begin{enumerate}[(i)]
            \item It is clear that all of the pullbacks are geometric irreducible and point-wise pure of weight $0$.
            \item Next we check that, for any $i\neq j$ there is no rank one sheaf $\mathcal L$ such that, 
            \begin{equation*}
                \gamma_i^*\mathcal H\not\simeq\gamma_j^*\mathcal H\otimes\mathcal L,\hspace{0.2cm}\text{or},\hspace{0.2cm}D(\gamma_i^*\mathcal H)\not\simeq\gamma_j^*\mathcal H\otimes\mathcal L.
            \end{equation*}
            We show this by checking their respective singularity sets. As $\mathcal L$ is a rank one sheaf, so the singularity set of $\gamma_j^*\mathcal H\otimes\mathcal L$ is same as the singularity set of $\gamma_j^*\mathcal H$. Singularity set of $\gamma_i^*\mathcal H$ is 
            \begin{equation*}
                \left\{\frac{-b_i}{c}, \frac{m_i^2-b_i}{c}, \infty\right\}
            \end{equation*}
            For $1\leq i, j\leq 4$ (or $5\leq i, j\leq 8$), using~\eqref{eq: change of variables}, Lemma~\ref{thm: first}, Lemma~\ref{thm: second}, and Lemma~\ref{thm: third}, we get that 
            $b_i\neq b_j$ if $i\neq j$. This in turn implies that their respective singularity sets are different. The same argument would work if $1\leq i\leq 4$, $5\leq j\leq 8$, and $j\neq i+4$. If $j=i+4$ then, with~\eqref{eq: change of variables} and Lemma~\ref{thm: first} we have $m_i\neq m_j$ which in turn again implies their singularity sets are different. This proves that 
            $$\gamma_i^*\mathcal H\not\simeq\gamma_j^*\mathcal H\otimes\mathcal L.$$ Similarly we also have, 
            $$D(\gamma_i^*\mathcal H)\not\simeq\gamma_j^*\mathcal H\otimes \mathcal L.$$
            \item Depending on the local representations our sheafs are,
            \begin{equation*}
                \mathcal H(1, \rho; \bar\mu, \bar\mu),\hspace{0.2cm}\mathcal H(1, \rho; \bar\mu_1, \bar\mu_2),\hspace{0.2cm}\text{and}\hspace{0.2cm}\mathcal H(1, \rho; \eta_f).
            \end{equation*}
            To determine $G^0$ we need to check various induction assumption like Kummer induced, Beyli induced, and inverse-Beyli induced,\\
            \begin{enumerate}[---]
                \item Let $d=2$, first sheaf is already not Kummer induced. Second sheaf is also not Kummer induced as $\mu_1\mu_2^{-1}$ is already a non-quadratic character. After base change to $\F_{p^2}$ the last sheaf becomes $\mathcal H(1\circ\text N, \rho\circ\text N; \eta_f, \eta_f^p)$, and it is also non-Kummer induced as $\eta_f^{2(p-1)}\neq 1$ because $\eta_f$ is a regularized character of $\F_{p^2}$ i.e. $\eta_f^{p-1}\neq 1$ and $\text{ord}(\eta_f)|(p^2-p)$.
                \item Write $2=1+1$; sheafs are not $(1, 1)$-Beyli induced. As for the first sheaf $\mu$ is non-quartic. Second sheaf would be Beyli-induced if $\mu_1^2=\rho=\mu_2^2$ it in particular implies $\mu_1\mu_2^{-1}$ is non-quadratic which we have assumed to be not. Similarly as, $\eta_f^{2(p-1)}\neq 1$, the last sheaf is non $(1, 1)$-Beyli induced. 
                \item Those sheafs are non $(1, 1)$-inverse Beyli-induced too as $\mu^2$, $\mu_1\mu_2$, and $\eta_f^{p+1}$ are non-trivial.
            \end{enumerate}
            Now let $\Lambda=\rho\chi_1\chi_2$, where $(\chi_1, \chi_2)$ is one of the pairs $(\br\mu, \br\mu)$; $(\br\mu_1, \br\mu_2)$; or $(\eta_f, \eta_f^p)$. Note that according to our assumptions $\Lambda\neq 1$ therefore there are only two possibilities, 
            \begin{enumerate}
              \item Suppose, $\Lambda^2=1$. In principal series case, it would imply $\mu_1\mu_2$ is quadratic; and in supercuspidal case it implies $\eta_f^{2(p+1)}=1$ which in turn forces that $\eta_f^{p-1}=1$.
              \item So the only possible case is $\Lambda^2\neq 1$, then we have $G^0=1$ or $\Sp(2)$. And according to our hypothesis $G^0=\Sp(2)$.
            \end{enumerate}
            \item For $i\neq j$, the paris $(\Sp(2), \text{Std}_i)$ and $(\Sp(2), \text{Std}_j)$ are Goursat-adapted~\cite[Example 1.8.1]{MR1081536}. 
        \end{enumerate}
        Therefore we have established $U$-generosity of $\mathcal F$.
        
        Now depending on the frequency $k$, we show the vanishing of second cohomology group.\\ First we consider the case of non-zero frequency $k\neq 0$. In this case $e(kx/p)$ gives rise to a non-trivial sheaf $\mathcal L_{\psi(kx)}$ of rank one which is smaller than the product of ranks of $\gamma_i^*\mathcal H$. Hence~\cite[Corollary 2.10.]{FKM15} gives, $H_c^2(U_{\F_p}; \mathcal F)=0$.\\ So we are left with the zero frequency $k=0$. To deal with this case, we follow the proof of~\cite[Theorem 2.7]{FKM15}. Suppose $H_c^2(U_{\F_p};\mathcal F)\neq 0$. Then by the Theorem 2.7 and Theorem 2.9 we have a geometric isomorphism 
        \begin{equation*}
            1\simeq\bigotimes_{1\leq i\leq 4}\Lambda_i(\gamma_i^*\mathcal H)\otimes\Lambda_{i+4}(D(\gamma_{i+4}^*\mathcal H))
        \end{equation*}
        where $\Lambda_j$ are irreducible representations of $G^0=\Sp(2)$ of $\mathcal H$ such that $\Lambda_i$ (resp. $\Lambda_{i+4}$) is a sub-representation of $\text{Std}$ (resp. $D(\text{Std})$). Therefore the trivial representation is sub-representation of Std and $D(\text{Std})$ but this is a contradiction.\\
        Therefore in any caes we have,
        \begin{equation*}
            \mathcal C(k)=-\text{tr}(\text{Fr}_x|H_c^1(U_{\F_p}; \mathcal F))
        \end{equation*}
        By Deligne's proof of the Riemann hypothesis over, as the sheaf $\mathcal F$ is of pure of weight $0$, all eigenvalues of Frobenius acting on the cohomology space have modulus $\leq \sqrt{p}$, and hence 
        \begin{equation*}
            \mathcal C(k)\ll \dim\,H_c^1(U_{\F_p}; \mathcal F)\times\sqrt{p}.
        \end{equation*}
        Finally, using the Euler-Poincar\'e formula, we can bound dimension in terms of conductor of $\mathcal F$; 
        \begin{equation*}
            \dim\,H_c^1(U_{\F_p}; \mathcal F)\leq C(\mathcal F)^2. 
        \end{equation*}
        Moreover from the definition of conductor and by the Theorem ?? we see that, 
        \begin{equation*}
            C(\mathcal F)\leq C(\mathcal L)\prod_{i=1}^{4}C(\mathcal \gamma_i^*\mathcal H)C(\gamma_{i+4}^*\mathcal H)=C(\mathcal L)C(\mathcal H)^8=O(1).
        \end{equation*}
        Summing up these we have, 
        \begin{lemma}[Square-root cancellation]\label{thm: Square-root cancellation}
            Let $\pi$ be as in the Theorem~\ref{thm: thm1} we have, 
            \begin{equation}\label{eq: sq-root cancellation of C(k)}
                \mathcal C(k)\ll \sqrt{p},
            \end{equation}
            where the implied constant is absolute. 
        \end{lemma}
        \qed

	\section{Estimation of error terms}\label{sec: End remarks}
	In this section we discuss the remaining cases~\eqref{eq:error 1} and~\eqref{eq:error 2}. Treatment of these cases would be similar to the $\M(X)$ and yields smaller contribution. 

	\subsection{Analysis of $\Er_1(X)$}  Note that in the equation~\eqref{eq:error 1} we have $(p, q)=1$, and so making a change of variable $a\leadsto ap$ we can rewrite this equation as, 

	\begin{equation*}
		\begin{split}
			\Er_1(X):=\frac{1}{Q}\int_\R W&\left(\frac{x}{\mathfrak X}\right) \underset{p\nmid q}{\sum_{q\leq Q}}\frac{g(q, x)}{pq}\sumstar_{a\md{q}}\\
			&\times\sum_m \lambda_\pi(m)e\left(\frac{am}{q}\right)e\left(\frac{xm}{pqQ}\right)U\left(\frac{m}{X^2}\right)\\
			&\times\sum_n e\left(-\frac{an^2}{q}\right)e\left(-\frac{xn^2}{pqQ}\right)V\left(\frac{n}{X}\right)\, dx
		\end{split}
	\end{equation*}
	Trivial estimate yields, 
	\begin{equation*}
		\Er_1(X)\ll_\eps XQ^2 
	\end{equation*}
	Therefore we need to save $Q^2=O(p^{2+\eps})$ and little more. Note that trivial estimate of $\M(X)$ was $X^3$, so we have already saved $p$ in this case. 

	Observe that in the $m$-sum the additive character has modulus $q$ which is co-prime to the level $p^2$, which makes the dual side of $m$-sum quite simpler. Up to an negligible error Voronoi summation formula transforms it into, 
	\begin{equation*}
		\begin{split}
			\chi_\pi(-q)\frac{\eps(\pi)X^2}{\sqrt{pq}}\sum_{m\ll p^{1+\eps}}\lambda_{\bar\pi}(m)e\left(-\frac{\overline{ap^2}m}{q}\right)\int_{0}^\infty  \frac{U(y)}{y^{1/4}}e\left(\frac{xyX^2}{pqQ}\pm \frac{2X\sqrt{my}}{pq}\right) \, dy 
		\end{split}
	\end{equation*}
	We have exactly saved the same amount as much we have saved in general case. 
	Now in the $n$-sum we have we will save more here, as the conductor of additive character has conductor $Q$ and the initial length is $X=Q\sqrt{p}$, so after applying the Poisson summation formula we will only left with zero frequency. Therefore up to negligible term it would be same as,
	\begin{equation*}
		\frac{X}{\sqrt{q}}\left(\frac{-a}{q}\right)\int_\R V(z) e\left(-\frac{xz^2Q}{q}\right)\, dz.
	\end{equation*}
     Plugging these expressions, and after some simplifications we can rewrite $\Er_1(X)$ as
	\begin{equation*}
		\begin{split}
			\eps(\pi)\mathfrak{X}\frac{X^2}{p}\underset{p\nmid q}{\sum_{q\leq Q}}\frac{\chi_{\pi}(-q)}{q^2}&\sum_{m\ll p^{1+\eps}}\frac{\lambda_{\bar\pi}(m)}{m^{1/4}}\sumstar_{a\md{q}}\left(\frac{a}{q}\right)e\left(\frac{am}{q}\right)\int_\R W(x)g(q, x\mathfrak{X})\\
			&\times\int_\R V(z) e\left(-\frac{x\mathfrak{X}z^2}{q}\right)\int_0^\infty  e\left(\frac{xy\mathfrak X X^2}{pqQ}\pm \frac{2X\sqrt{my}}{pq}\right)\, dx\, dz\, dy.
		\end{split}
	\end{equation*}
    The $a\md{q}$-sum gets reduce to the quadratic character $\left(\frac{m}{q}\right)\sqrt{q}$. 
    
    Now we turn our attention to the analysis of three-fold integral transform. Analysis of $y$-integral will be exactly same as in the general case. So, we have, 
    \begin{equation*}
        \int_{0}^{\infty}U(y)e\left(\frac{xy\mathfrak X Q}{q}\pm \frac{2X\sqrt{my}}{pq}\right)\, dy=\frac{p^{1/4}}{m^{1/4}}\sqrt{\frac{q}{Q}}e\left(\mp \frac{mQ}{xpq|\mathfrak X|}\right)U\left(\frac{m}{M}\right)+O(p^{-2026}).
    \end{equation*}

    Therefore we have,
    \begin{equation}\label{eq: end rmk for error 1}
        \Er_1(X)\ll_\eps \frac{X^2p^{1/4}}{p\sqrt{Q}}\sum_{q\leq Q}\frac{1}{q}\sum_{m\ll p^{1+\eps}}\frac{|\lambda_f(m)|}{\sqrt{m}}\ll_\eps\frac{X}{\sqrt{p}}.
    \end{equation}

    \subsection{Analysis of $\Er_2(X)$} In this section we estimate the contribution of $\Er_2(X)$. Recall, 

    \begin{equation*}
		\begin{split}
			\Er_2(X):=\frac{\mathfrak X}{Q}\int_{x\sim 1} W&\left({x}\right) {\sum_{q\leq Q/p}}\frac{g(pq, x\mathfrak X)}{p^2q}\sumstar_{a\md{p^2q}}\\
			&\times\sum_{m\sim X^2} \lambda_f(m)e\left(\frac{am}{p^2q}\right)e\left(\frac{xm\mathfrak X}{p^2qQ}\right)\\
			&\times\sum_{n\sim X} e\left(-\frac{an^2}{p^2q}\right)e\left(-\frac{xn^2\mathfrak X}{p^2qQ}\right)\, dx.
		\end{split}
	\end{equation*}
	Here our trivial estiamte is $X^2\sqrt{p}$, so to obtain non-trivial estimate we need to save $X\sqrt{p}$ and little more. Note that as $q$ has to be bigger than $1$, it restricts $X$ to its generic size i.e. $p^{3/2}\leq X \ll p^{3/2+\eps}$ and which in turn implies $q\leq O(p^\eps)$. As the $q$-sum wouldn't contribute more than $p^\eps$, so for sake of simplicity we set $q=1$ in $\Er_2(X)$.\\
	Voronoi summation transform the $m$-sum to, 
	\begin{equation*}
		\asymp \chi_\pi(\bar a)\frac{X^2}{p^2}\sum_{m\ll p^{1+\eps}}\lambda_{\bar\pi}(m)e\left(-\frac{\bar{a}m}{p^2}\right)\int_{y\sim 1}e\left(\frac{xy\mathfrak X X^2}{p^2Q}\right)J_{\kappa-1}\left(\frac{4\pi X\sqrt{my}}{p^2}\right)\, dy
	\end{equation*}
	Note that the integral transform is non-oscillatory, so we drop them from our future analysis. Similarly Poisson summation transform the $n$-sum to, 
	\begin{equation*}
		\frac{X}{p}\sum_{n\ll p^{1/2+\eps}}G(-a, -n; p^2)\int_{z\sim 1}e\left(\frac{nzX}{p^2}-\frac{x\mathfrak X y^2X^2}{p^2Q}\right)\, dz.
	\end{equation*}
	For similar reason we will also drop this integral transform. Now plugging these dual expressions we get, 
	\begin{equation*}
		\Er_2(X)\ll \frac{X^2}{p^{9/2}}\sum_{m\ll p^{1+\eps}}\sum_{n\ll p^{1/2+\eps}}|\lambda_\pi(m)||\mathfrak C(m, n; p^2)|
	\end{equation*}
	where, 
	\begin{equation*}
		\mathfrak C(m, n; p^2):=\sumstar_{a\md{p^2}}\chi_\pi(\bar a)G(-a, -n; p^2)e\left(-\frac{\bar{a}m}{p^2}\right)
	\end{equation*}
	We need the following lemma,
	\begin{lemma*}We have,
		\begin{equation*}
			\mathfrak C(m, n; p^2)\ll p^{3/2} \delta(4m-n^2\equiv 0\md{p}).
		\end{equation*}
	\end{lemma*} 
	\begin{proof}
		Writing explicitly the quadratic Gau\ss\ sum we have, 
		\begin{equation*}
			\sumstar_{a\md{p^2}}\chi_\pi(\bar a)e\left(\frac{\br{4a}n^2-m\bar a}{p^2}\right)
		\end{equation*}
        By a change of variable: $a\mapsto\br{4a}$, it becomes,
        \begin{equation*}
            \chi_\pi(4)\sumstar_{a\md{p^2}}\chi_\pi(a)e\left(\frac{(n^2-4m)a}{p^2}\right).
        \end{equation*}
        Let us rewrite $a$ as $x+py$ with $x\in(\Z/p\Z)^\times$, and $y\in\Z/p\Z$
        \begin{equation*}
            \chi_\pi(4)\sumstar_{x\md{p}}\sum_{y\md{p}}\chi_\pi(x)e\left((n^2-4m)\left(\frac{x}{p^2}+\frac{y}{p}\right)\right)
        \end{equation*}
        Evaluating the $y\md{p}$-sum it becomes,
        \begin{equation*}
            \chi_\pi(4)\delta(4m\equiv n^2\md{p})\,p\sumstar_{x\md{p}}\chi_\pi(x)e\left(\frac{(n^2-4m)x}{p^2}\right),
        \end{equation*}
        by bounding the $x\md{p}$-sum by $\sqrt{p}$ we get our required estimate.
	\end{proof}
    
    Now we plug this bound in the last estimate of $\Er_2(X)$.
	We have,
	\begin{equation}\label{eq: end rmk for error 2}
		\Er_2(X)\ll \frac{X^2}{p^{3}}\underset{4m-n^2\equiv 0\md{p}}{\sum_{m\ll p^{1+\eps}}\sum_{n\ll p^{1/2+\eps}}}|\lambda_\pi(m)|\ll_\eps \frac{X}{p}.
	\end{equation}

    \section{Proof of main results}

    \begin{proof}[Proof of Theorem~\ref{thm: thm1}]
    Combining the estimates of $\M_u^\text{odd}(X)$ given in the Lemma~\ref{thm: large D}, Lemma~\ref{thm: small D and large T}, Lemma~\ref{thm:Small $D$ and small $T$} and the error term estimates given in the previous section, with~\eqref{eq: decomp of S_f(X)} we have, 

    \begin{equation*}
        \begin{split}
        S_\pi(X)&\ll\sqrt{X}p^{\frac{3}{4}-\delta_1+\eps}+\sqrt{X}p^{\frac{3}{4}-\delta_2+\eps}+\sqrt{X}p^{\frac{3}{4}+\frac{\delta_1}{4}-\frac{\delta_2}{2}+\eps}\\&\hspace{0.2cm}+\sqrt Xp^{\frac{3}{4}+\delta_2-\frac{1}{16}+\eps}+\sqrt{X}p^{\frac{3}{4}-\frac{1}{16}+\eps}+\sqrt{X}p^{\frac{3}{4}-\frac{1}{2}+\eps},
        \end{split}
    \end{equation*}
    Optimizing the first and third term from the right hand side we have, 
    \begin{equation*}
        S_\pi(X)\ll \sqrt X p^{\frac{3}{4}-\frac{2\delta_2}{5}+\eps}+\sqrt X p^{\frac{3}{4}+\delta_2-\frac{1}{16}+\eps}+\sqrt{X}p^{\frac{3}{4}-\frac{1}{16}+\eps},
    \end{equation*}
    and finally optimizing $\delta_2$ we have, 
    \begin{equation*}
        S_\pi(X)\ll \sqrt{X}p^{\frac{3}{4}-\frac{1}{56}+\eps}.
    \end{equation*}
    Plugging this estimate in~\eqref{eq:appx by s_f(x)} gives the subconvexity bound, 
    \begin{equation}
        L\left(\frac{1}{2},\, \sym^2\pi\right)\ll p^{\frac{3}{4}-\frac{1}{56}+\eps}.
    \end{equation}
    This completes the proof of Theorem~\ref{thm: thm1}.
    \end{proof}

    Now we prove the Corollary~\ref{thm:cor main}

    \begin{proof}[Proof of Corollary~\ref{thm:cor main}]
        We have already proved the subconvexity bound for the required $L$-function. To prove this corollary it is enough to understand the set $\mathcal S$ for the admissible primes $p$. 

        Recall the definition of set $\mathcal S$, 
        \[
        \mathcal S=\{\pi_p~\text{is of type}~(a)\,\text{or}~(b): G_{\text{geom}}^0\neq\{1\}\}
        \]
         In this set we are excluding all such representations which gives rise to exotic hypergeometric sheaves of finite geometric monodromy equivalently $G_\text{geom}^0=\{1\}$. Our hypergeometric sheaves have rank $2$. In the appendix we have classified all such hypergeometric sheaves of finite $G_\text{geom}$. We will use such classifications depending on the type of local representations. 
         \begin{enumerate}
             \item Let $\pi_p\simeq\pi(\tilde\mu_1, \tilde\mu_2)$ of type $(a)$. The associated hypergeometric sheave over $\F_p$ in this case is 
             $$
             \mathcal{H}(1, \rho; \bar\mu_1, \bar\mu_2)
             $$ 
             To each pair $(1, \rho)$ and $(\bar\mu_1, \bar\mu_2)$ we associate a multiset in $\Q/\Z$,
             \[
             (1, \rho)\mapsto \left\{0, \frac{1}{2}\right\}\hspace{0.5cm}\text{and}\hspace{0.5cm}(\bar\mu_1, \bar\mu_2)\mapsto\{y_1, y_2\}
             \]
             So to $\mathcal H(1, \rho; \bar\mu_1, \bar\mu_2)$ we can associated the following hypergeometric differential equation 
             \[
             z(1-z)g^{\prime\prime}(z)+(y_1+y_2-1/2)g^\prime(z)-(1/2-y_1)(1/2-y_2)g(z)=0
             \]
             and the triple of local exponent differences up to sign is given by
             \[
             (e_0, e_1, e_\infty)=\left(\frac{1}{2},\, \frac{1}{2}+y_1+y_2,\, y_1-y_2\right)
             \]
             Now to determine the finite type of the monodromy we use Schwarz's list given in Table~\ref{tab: schwarz list}. As $\mu_1\mu_2^{\pm 1}$ both are non-trivial and non-quadratic, i.e. $y_1\neq \pm y_2$, there is no possibility of monodromy group to be Dihedral. Multisets in the other cases is given in the following table
             
             \begin{table}[htpb]
        \centering
        \renewcommand{\arraystretch}{1.2}
        \begin{tabular}{lcccll}
        \hline
        \hspace{0.3cm}Type & $(e_1, e_\infty)$ & $\mathbf{y}=\{y_1, y_2\}$ & Condition on $p$ & Can occur?\\
        \hline
       \hspace{0.2cm}Dihedral & $(1/2, r/s)$ & --- & --- & \hspace{0.3cm}No \\   
       Tetrahedral & $(1/3, 1/3)$ & $\{1/12, 3/4\}$ & $p\equiv 1\md{12}$  & \hspace{0.3cm}Yes\\
       Octahedral & $(1/3, 1/4)$ & $\{1/24, 19/24\}$ & $p\equiv 1\md{24}$ & \hspace{0.3cm}Yes\\
       Octahedral & $(1/4, 1/3)$ & $\{1/24, 17/24\}$ & $p\equiv 1\md{24}$ & \hspace{0.3cm}Yes\\
       Icosahedral & $(1/3, 1/5)$ & $\{1/60, 49/60\}$ &  $p\equiv 1\md{60}$ & \hspace{0.3cm}Yes\\
       Icosahedral & $(1/5, 1/3)$ & $\{1/60, 41/60\}$ & $p\equiv 1\md{60}$  & \hspace{0.3cm}Yes\\
       Icosahedral & $(1/3, 2/5)$ & $\{7/60, 43/60\}$ & $p\equiv 1\md{60}$& \hspace{0.3cm}Yes\\
       Icosahedral & $(2/5, 1/3)$ & $\{7/60, 47/60\}$ & $p\equiv 1\md{60}$ & \hspace{0.3cm}Yes\\
       Icosahedral & $(1/5, 2/5)$ & $\{1/20, 13/20\}$ & $p\equiv 1\md{20}$ & \hspace{0.3cm}Yes\\
       Icosahedral & $(2/5, 1/5)$ & $\{1/20, 17/20\}$ & $p\equiv 1\md{20}$ & \hspace{0.3cm}Yes\\
       \hline
       \end{tabular}
       \vspace{0.1cm}
    \end{table}
 
    As the orders of $\mu_1$, and $\mu_2$ divides $p-1$, so the above finite monodromy groups can occur if $p$ satisfies one of the following congruences: 
    \[
    p\equiv 1\md{12},\hspace{0.2cm} p\equiv1\md{20},\hspace{0.2cm} p\equiv 1\md{24},\hspace{0.2cm} p\equiv1\md{60}.
    \]
    If $p$ is admissible prime, then it doesn't satisfy any of the above congruences which in turn implies that for such primes the hypergeometric sheaves attached to $\pi_p$ doesn't have finite monodromy group.

    \item Now let $\pi_p$ is of type $(b)$. In this case the our associated sheaf is, $\mathcal H(1, \rho;\, \eta_\pi)$, an exotic hypergeometric sheaf over $\F_p$. To analyze this case we need to perform base change from ambient field $\F_p$ to $\F_{p^2}$. So by base change from $\F_p$ to $\F_{p^2}$ this becomes isomorphic to 
    \[
    \mathcal{H}(1, \rho\circ\text{N}_{\F_p^2/\F_p};\, \eta_\pi, \eta_\pi^p).
    \]
    Similarly as before we associate multisets to these of characters:
    
    \[
    \left(1, \rho\circ\text{N}_{\F_p^2/\F_p}\right)\mapsto\left\{0, \frac{1}{2}\right\}\hspace{0.5cm}\text{and}\hspace{0.5cm}\left(\eta_\pi, \eta_\pi^p\right)\mapsto \left\{y_1:=\frac{k}{p^2-1}, y_2:=\frac{pk}{p^2-1}\right\}
    \]
    
    The local exponent differences of the associated hypergeometric differential equation: 
    \[
    (e_0, e_1, e_\infty) = \left\{\frac{1}{2}, \frac{1}{2}+\frac{k}{p-1}, \frac{k}{p+1}\right\}.
    \]
    since $\eta_\pi^{2(p\pm 1)}\neq 1$ again there is no chance for Dihedral case. Complete list is given as follows, 

    \begin{table}[htpb]
        \centering
        \renewcommand{\arraystretch}{1.2}
        \begin{tabular}{lcccll}
        \hline
        \hspace{0.3cm}Type & $(e_1, e_\infty)$ & $\mathbf{y}=\{y_1, y_2\}$ & Condition on $p$ & Can occur?\\
        \hline
       \hspace{0.2cm}Dihedral & $(1/2, r/s)$ & --- & --- & \hspace{0.3cm}No \\   
       Tetrahedral & $(1/3, 1/3)$ & $\{1/12, 3/4\}$ & --- & \hspace{0.3cm}No\\
       Octahedral & $(1/3, 1/4)$ & $\{1/24, 19/24\}$ & $p\equiv 19\md{24}$ & \hspace{0.3cm}Yes\\
       Octahedral & $(1/4, 1/3)$ & $\{1/24, 17/24\}$ & $p\equiv 17\md{24}$ & \hspace{0.3cm}Yes\\
       Icosahedral & $(1/3, 1/5)$ & $\{1/60, 49/60\}$ &  $p\equiv 49\md{60}$ & \hspace{0.3cm}Yes\\
       Icosahedral & $(1/5, 1/3)$ & $\{1/60, 41/60\}$ & $p\equiv 41\md{60}$  & \hspace{0.3cm}Yes\\
       Icosahedral & $(1/3, 2/5)$ & $\{7/60, 43/60\}$ & $p\equiv 49\md{60}$& \hspace{0.3cm}Yes\\
       Icosahedral & $(2/5, 1/3)$ & $\{7/60, 47/60\}$ & $p\equiv 41\md{60}$ & \hspace{0.3cm}Yes\\
       Icosahedral & $(1/5, 2/5)$ & $\{1/20, 13/20\}$ & --- & \hspace{0.3cm}No\\
       Icosahedral & $(2/5, 1/5)$ & $\{1/20, 17/20\}$ & --- & \hspace{0.3cm}No\\
       \hline
       \end{tabular}
       \vspace{0.1cm}
    \end{table}

    Tetrahedral case can occur if $3$ divides $p\pm 1$ i.e., $3|(p-1, p+1)$ which is impossible for odd $p$. And by the same reason there is no possibility for the last two cases of Icosahedral, $(1/5, 2/5), (2/5, 1/5)$.

    But for the other cases finite monodromy can occur, and we need to analyze ``condition on $p$''. We present one case and other cases will follow similarly. 

    Consider one of the Octahedral case $(1/3, 1/4)$ and the corresponding multiset is $\mathbf{y}=\{y_1, y_2\}=\left\{\frac{1}{24}, \frac{19}{24}\right\}$. This multiset need to satisfy the a special condition, 
    \[
    y_2\equiv py_1\md{p} \implies \frac{p}{24}\equiv\frac{19}{24}\md{1}\Longleftrightarrow p\equiv 19\md{24}
    \]
    
    So if prime $p$ is admissible then it will avoid all of the above congruences, which implies that the associated sheaf has infinite geometric monodromy group.
    \end{enumerate}

Combining these cases we conclude the proof of the corollary.
    
\end{proof}

    \section*{Appendix} 

    In this section we record some of the known results, mostly for future reference, a criterion to determine when the geometric monodromy group $G_{\text{geom}}$ of hypergeometric sheaves is finite. The main reference for this section is~\cite{MR1081536}. 

    Let's fix a finite field $\F_q$, and $\ell$ be a prime invertible in $\F_q$. And let's fix $\psi$ be a non-trivial $\ell$-adic additive character of $\F_q$. Let $n\geq 1$ be an integer. Consider the following data of multisets: 
    \[
    \boldsymbol{\chi}:=\{\chi_1,\dots, \chi_n\}, \hspace{0.5cm}\boldsymbol{\rho}:=\{\rho_1,\dots, \rho_n\}
    \]
    of $\ell$-adic characters of $\F_q^\times$ which are assumed to be disjoint i.e., $\chi_i\neq\rho_j$ for $1\leq i, j\leq n$. 

    Let $\lambda\in\F_q^\times$, the hypergeometric sum of type $(n, n)$ is given by, 
    \[
    \text{Hyp}(\F_q^{2n}; \boldsymbol{\chi}; \boldsymbol{\rho}; \lambda):= \frac{1}{|\F_q|^{n-\frac{1}{2}}}\underset{\frac{x_1\cdots x_n}{y_1\cdots y_n}=\lambda}{\sum_{x_1,\dots, x_n\in\F_q^\times}\sum_{y_1,\dots, y_n\in\F_q^\times}}\prod_{i=1}^{n}\chi_i(x_i)\bar{\rho_i}(y_i) \psi\left(\sum_{i=1}^n(x_i-y_i)\right)
    \]
    Katz constructed an $\ell$-adic middle-extension sheaf $\mathcal H(\boldsymbol{\chi}, \boldsymbol{\rho})$  on $\mathbb G_{m, \F_q}$ whose trace functions are these hypergeometric sums. 

    For each hypergeometric sum there is an associated hypergeometric differential equation. To define this we first fix an generator $\gamma\in\F_q^\times$. Using $\gamma$ we identify $\boldsymbol{\chi}$ (resp. $\boldsymbol{\rho}$) with multisets $\mathbf{x}=\{x_1,\dots, x_n\}$ (resp. $\mathbf{y}=\{y_1,\dots, y_n\}$) in $\Q/\Z$ by $\chi_i(\gamma)=e(x_i)$ (resp. $\rho_j(\gamma)=e(y_j)$). Viewing these multisets in $[0, 1)$ we have the associated hypergeometric differential equation $\mathrm D_{x, y}f=0$ (see~\cite[\S 3.1]{MR1081536}), where
    \[
    \mathrm D_{\mathbf x, \mathbf y}:=\prod_{i=1}^{n} (\partial - x_i) - z\prod_{i=1}^n (\partial - y_i),\hspace{1cm}\partial=z\partial_z . 
    \]
    Katz gives the following equivalent criterion for geometric monodromy group to be finite~(\cite[Corollary 8.17.15]{MR1081536}), 
    \begin{theorem}
        Let $\boldsymbol{\chi}, \boldsymbol{\rho}, \mathbf{x}$, and $\mathbf{y}$ are as above, and assume that all multiplicities in $\mathbf{X}$, and $\mathbf{Y}$ be equal $1$. Let $\mathcal H(\boldsymbol{\chi}, \boldsymbol{\rho})$ is a hypergeometric sheav of type $(n, n)$. Then the followings are equivalent: 
        \begin{enumerate}[1.]
            \item $\mathcal{H}(\boldsymbol{\chi}, \boldsymbol{\rho})$ has finite $G_\text{geom}$.
            \item For every unit $\alpha$ modulo the common denominator of $x_1,\dots, x_n, y_1, \dots, y_n$, The two subsets, 
            \[
            \mathbf{x}_\alpha=\{\alpha x_1,\cdots, \alpha x_n\}\hspace{0.5cm}\mathbf{y}_\alpha=\{\alpha y_1,\dots, \alpha y_n\}
            \]
            interlace or intertwined on the unit interval i.e., as we walk counterclockwise around the unit circle we alternatively encounter one from each subset.
            \item Differential Galois group $Gal_\text{diff}$ of $\mathrm D_{\mathbf x, \mathbf y}$ is finite.
        \end{enumerate}
        And indeed $G_\text{geom}\simeq Gal_\text{diff}$.
    \end{theorem}
    Now we restrict ourselves to $n=2$. In the differential equation $D_{\mathbf x, \mathbf y}f(z)=0$ we set $f(z)=z^{x_2}g(z)$. So $g$ satisfies the following differential equation 
    \[
    z(1-z)g^{\prime\prime}(z)+[c-(a+b+1)z]g^\prime(z)-abg(z)=0,
    \]
    where
    \[
    a=x_2-y_1,\hspace{0.25cm}b=x_2-y_2,\hspace{0.25cm}c=1+x_2-x_1.
    \]
    This is a Fuchsian equation with regular singularity points $0, 1$, and $\infty$ on $\mathbb P^1$. The local exponent differences at these points (up to sign) $e_0=1-c,\, e_1=c-a-b$, and $e_\infty = a-b$ (see~\cite{900052125}). It is known that this equation has finite monodromy group iff it has algebraic solutions. Schwarz has given a complete list of $(e_0, e_1, e_\infty)$ for which the hypergeometric differential equations have algebraic solutions~(see~\cite[\S 3]{Matsuda1985}). We give the complete list of triples and corresponding projective monodromy group in the Table~\ref{tab: schwarz list}.
    
    \begin{table}[htpb]
        \centering
        \renewcommand{\arraystretch}{1.2}
        \begin{tabular}{lcccll}
        \hline
        \hspace{0.3cm}Type & $e_0$ & $e_1$ & $e_\infty$ & \text{Order}\\
        \hline
        \hspace{0.2cm}Dihedral & $1/2$ & $1/2$ & $m/n$ & \hspace{0.3cm}$2n$\\
        \hline
        Tetrahedral & $1/2$ & $1/3$ & $1/3$ & \hspace{0.3cm}$12$\\
        Tetrahedral & $2/3$ & $1/3$ & $1/3$ & \hspace{0.3cm}$12$\\
        \hline
        Octahedral & $1/2$ & $1/3$ & $1/4$ & \hspace{0.3cm}$24$\\
        Octahedral & $2/3$ & $1/4$ & $1/4$ & \hspace{0.3cm}$24$\\
        \hline
        Icosahedral & $1/2$ & $1/3$ & $1/5$ & \hspace{0.3cm}$60$\\
        Icosahedral & $1/2$ & $1/3$ & $2/5$ & \hspace{0.3cm}$60$\\
        Icosahedral & $1/2$ & $1/5$ & $2/5$ & \hspace{0.3cm}$60$\\
        Icosahedral & $1/3$ & $1/3$ & $2/5$ & \hspace{0.3cm}$60$\\
        Icosahedral & $1/3$ & $2/3$ & $1/5$ & \hspace{0.3cm}$60$\\
        Icosahedral & $1/3$ & $1/5$ & $3/5$ & \hspace{0.3cm}$60$\\
        Icosahedral & $1/3$ & $2/5$ & $3/5$ & \hspace{0.3cm}$60$\\
        Icosahedral & $2/3$ & $1/5$ & $1/5$ & \hspace{0.3cm}$60$\\
        Icosahedral & $1/5$ & $1/5$ & $4/5$ & \hspace{0.3cm}$60$\\
        Icosahedral & $2/5$ & $2/5$ & $2/5$ & \hspace{0.3cm}$60$\\
        \hline
        \end{tabular}
        \vspace{0.1cm}
        \caption{Schwarz list}
        \label{tab: schwarz list}
    \end{table}

	\section*{Acknowledgments}
    The author would like to express his sincere gratitude to Prof. Ritabrata Munshi for suggesting this problem as well as for his continuous guidance and immense support throughout this work. The author is grateful to Prof. Will Sawin for a helpful conversation on hypergeometric sums in Section~\ref{sec:hyp geo sum}. The author is grateful to Prof. Philippe Michel for a helpful conversation on this problem during the conference \textit{Circle Method and Related Topics} held at ICTS, Bengaluru, in 2024. The author is also grateful to Mayukh Dasaratharaman for bringing the reference~\cite{MR3655759} to his attention and for numerous helpful suggestions. The author further thanks Sayan Ghosh and Sampurna Pal for helpful discussions and for their comments on the current draft. And finally, the author gratefully acknowledges the Indian Statistical Institute, Kolkata, for providing an excellent and supportive research environment.

\end{document}